\documentclass{article}

\usepackage[margin=1in]{geometry} 
\usepackage{amsmath,amssymb,amsthm}
\usepackage{amsmath}
\usepackage{amsfonts}

\usepackage{hyperref}
\usepackage{mathrsfs}
\usepackage{mathtools}
\usepackage{graphicx}
\usepackage{tikz-cd}
\usepackage{xcolor}
\usepackage{indentfirst}
\usepackage[shortlabels]{enumitem}
\usepackage[capitalize]{cleveref}
\usepackage[overload]{empheq}
\usepackage[title]{appendix}
\usepackage[T1]{fontenc}
\usepackage{bbm}

\newcommand{\ZZ}{\mathbb{Z}}
\newcommand{\NN}{\mathbb{N}}
\newcommand{\QQ}{\mathbb{Q}}
\newcommand{\FF}{\mathbb{F}}

\DeclareMathOperator{\Spec}{Spec}

\DeclareMathOperator{\Hom}{Hom}

\newcommand{\catC}{\mathcal{C}}
\newcommand{\catD}{\mathcal{D}}
\newcommand{\catE}{\mathcal{E}}

\newcommand{\op}{\text{op}}

\newcommand{\Space}{\mathcal{S}}
\newcommand{\ProVect}{\operatorname{Pro} (\Vect_k^\omega)}

\newcommand{\LSym}{\mathbb{L}\Sym}
\newcommand{\adicn}{\adic_{\leq n}}
\newcommand{\cotn}{\cot_{\leq n}}
\newcommand{\KD}{\mathfrak{D}}

\renewcommand{\nu}{\operatorname{nu}}
\makeatletter
\newcommand{\colim@}[2]{%
  \vtop{\m@th\ialign{##\cr
    \hfil$#1\operator@font colim$\hfil\cr
    \noalign{\nointerlineskip\kern-\ex@}\cr}}%
}

\newcommand{\colim}{%
  \mathop{\mathpalette\colim@{\scriptscriptstyle}}\nmlimits@
}

\DeclareMathOperator{\weirdleq}{\substack{  _{} \vspace{-3.8pt} \\ \eqslantless }}
\DeclareMathOperator{\Solid}{Solid}

\DeclareMathOperator{\Fun}{Fun}
\DeclareMathOperator{\Sp}{Sp}

\DeclareMathOperator{\Map}{Map}
\DeclareMathOperator{\Ind}{Ind}

\DeclareMathOperator{\Lie}{Lie}

\DeclareMathOperator{\CAlg}{CAlg}
\DeclareMathOperator{\Poly}{Poly}
\DeclareMathOperator{\cofib}{cofib}
\DeclareMathOperator{\fib}{fib}
\DeclareMathOperator{\sqz}{sqz}

\DeclareMathOperator{\End}{End}

\DeclareMathOperator{\Set}{Set}

\DeclareMathOperator{\Vect}{Vect}
\DeclareMathOperator{\Pro}{Pro}

\DeclareMathOperator{\Fil}{Fil}
\DeclareMathOperator{\Gr}{Gr}
\DeclareMathOperator{\Tot}{Tot}
\DeclareMathOperator{\adic}{adic}
\DeclareMathOperator{\free}{free}

\renewcommand{\Bar}{\operatorname{Bar}}
\DeclareMathOperator{\id}{id}

\DeclareMathOperator{\Sym}{Sym}

\DeclareMathOperator{\art}{art}
\DeclareMathOperator{\LMod}{LMod}

\DeclareMathOperator{\Mod}{Mod}
\DeclareMathOperator{\Der}{Der}

\DeclareMathOperator{\perf}{perf}

\DeclareMathOperator{\Ab}{Ab}
\DeclareMathOperator{\ds}{ds}
\DeclareMathOperator{\Alg}{Alg}

\DeclareMathOperator{\Moduli}{Moduli}

\DeclareMathOperator{\bNil}{bNil}
\DeclareMathOperator{\Nil}{Nil}

\DeclareMathOperator{\forget}{forget}

\DeclareMathOperator{\triv}{triv}

\DeclareMathOperator{\ab}{ab}
\DeclareMathOperator{\LieAlgd}{LieAlgd}

\DeclareMathOperator{\surj}{surj}
\DeclareMathOperator{\an}{an}
\DeclareMathOperator{\Arr}{Arr}
\DeclareMathOperator{\st}{st}
\DeclareMathOperator{\target}{target}
\DeclareMathOperator{\PolyMod}{PolyMod}
\DeclareMathOperator{\ev}{ev}
\DeclareMathOperator{\aug}{aug}
\DeclareMathOperator{\coAlg}{coAlg}
\DeclareMathOperator{\exact}{exact}
\DeclareMathOperator{\const}{const}

\renewcommand{\epsilon}{\varepsilon}

\usepackage{blindtext}

\usepackage{thmtools,thm-restate}

\newtheorem{theorem}{Theorem}
\newtheorem{lemma}[theorem]{Lemma}
\newtheorem{prop}[theorem]{Proposition}
\newtheorem{cor}[theorem]{Corollary}

\theoremstyle{definition}
\newtheorem{defi}[theorem]{Definition}
\newtheorem{rem}[theorem]{Remark}
\newtheorem{example}[theorem]{Example}
\newtheorem{construction}[theorem]{Construction}
\newtheorem{notation}[theorem]{Notation}

\numberwithin{equation}{section}
\numberwithin{lemma}{section}
\numberwithin{theorem}{section}
\numberwithin{cor}{section}
\numberwithin{defi}{section}
\numberwithin{rem}{section}
\numberwithin{example}{section}
\numberwithin{prop}{section}
\numberwithin{construction}{section}
\numberwithin{notation}{section}
\numberwithin{conjecture}{section}
\numberwithin{problem}{section}

\begin{document}

\title{Nilpotent Deformation Theory}
\author{Sofía Marlasca Aparicio}
\date{}

\maketitle

\begin{abstract}
	Deformation theory studies how varieties and other algebro-geometric objects vary in families. A central part of the subject is formal deformation theory, where one deforms over an Artinian base; such deformation problems are governed by Lie algebraic models. \\

	We pose the question of deforming varieties over nilpotent but not necessarily Artinian bases. These turn out to be classified by the same Lie algebraic models plus some topological structure. More precisely, we will consider partition Lie algebroids in the category of ultrasolid modules, a variation of the solid modules of Clausen and Scholze that give a well-behaved category akin to topological modules. \\

	To approach this result, we decompose deformation problems into n-nilpotent layers. Each of these layers is individually easier to understand, and is classified by simpler variants of partition Lie algebras.
\end{abstract}

\tableofcontents

\section{Introduction}

In formal deformation theory, one considers deformations of algebro-geometric objects over local Artinian bases. If we are working over a base field $k$, we can think of these as finite nilpotent thickenings of $\Spec (k)$. This originally goes back to the work of Kodaira-Spencer (\cite{Kodaira-Spencer}), Schlessinger (\cite{Schlessinger}) and Artin (\cite{Artin}).

In this thesis, we will consider deformations over non-finite nilpotent thickenings. That is, local rings with residue field $k$ and nilpotent maximal ideal that are not necessarily Artinian.

\begin{example}
	\label{example-jetscheme}
	Let $X = \Spec (k[x,y]/(xy))$, two lines glued at the origin. If $A$ is a local Artinian $k$-algebra with residue field $k$, an \emph{$A$-valued arc} is a map $\Spec (A[[t]]) \to X$. In our example, this corresponds to choosing two formal power series $f,g \in A[[t]]$ such that $f g = 0$. We can easily see that the functor sending a $k$-algebra to its space of arcs in $X$ is representable by the scheme
	\begin{align*}
		J := \Spec(k[a_0, a_1, \dots, b_0, b_1, \dots]/(c_0, c_1, \dots))
	\end{align*}
	where $c_i = \sum_{j= 0}^i a_j b_{i -j}$.

	Let $A_\infty := k[a_0, a_1, \dots, b_0, b_1, \dots]/(c_0, c_1, \dots)$ and let $I = (a_0, a_1, \dots, b_0, b_1)$. The trivial $k$-valued arc (that is, $f = g = 0$) is given by the natural map $A_\infty \to A_\infty/I \cong k$. This factorises through the ring $A_\infty/I^n$, which is a nilpotent thickening of $k$. This factorisation corresponds precisely to a deformation of this arc. Additionally, the map $A_\infty \to A_\infty/I^n$ cannot be written as a filtered colimit of maps $A_\infty \to B_i$, where $B_i$ is local and Artinian.
\end{example}

It is well-known that formal deformation theory is governed by Lie algebraic models. This connection was made by Deligne (\cite{Deligne}), Drinfeld (\cite{Drinfeld}), and further developed by Goldman-Millson (\cite{Goldman}), Kontsevich (\cite{Kontsevich}), and Manetti (\cite{Manetti}).

A key insight is to consider deformations over animated rings, which gives rise to the modern language of \emph{formal moduli problems}. That is, a functor that sends a finite nilpotent thickening $A$ to the space of deformations over $A$, satisfying some geometric conditions.

\begin{notation}
	Let $R$ be a ring with a fixed map $R \to k$. We will write $\CAlg_R^{\an}$ for the $\infty$-category of animated $R$-algebras, and $\CAlg_{R/ /k}^{\an}$ for the $\infty$-category of animated $R$-algebras equipped with a compatible map to $k$.

	Finally, let $\CAlg_{R/ /k}^{\an, \surj} \subset \CAlg_{R/ /k}^{\an}$ be the full subcategory of those $A \in \CAlg_{R/ /k}^{\an, \surj}$ such that $\pi_0 (A) \to k$ is surjective.
\end{notation}

\begin{defi}
	An animated ring $A \in  \CAlg_{R/ /k}^{\an, \surj}$ is \emph{local and Artinian} if
	\begin{enumerate}
		\item $\pi_0 (A)$ is local and Artinian with residue field $k$.
		\item $\pi_i (A) \cong 0$ for $i \gg 0$.
		\item $\pi_i (A)$ is a finitely generated $\pi_0 (A)$-module for all $i$.
	\end{enumerate}
	We will write $\CAlg_{R/ /k}^{\an, \art} \subset \CAlg_{R/ /k}^{\an}$ for the full subcategory of local and Artinian algebras.
\end{defi}

\begin{defi}
	An \emph{$R/k$ formal moduli problem} is a functor $F: \CAlg_{R/ /k}^{\an, \art} \to \Space$, where $\Space$ is the $\infty$-category of spaces, such that
	\begin{itemize}
		\item $F(k) \simeq *$.
		\item $F$ preserves $\pi_0$-surjective pullbacks.
	\end{itemize}
\end{defi}

Lurie (\cite{DAGX}) and Pridham (\cite{Pridham}) showed that in characteristic zero, there is an equivalence between differential graded Lie algebras and $k/k$ formal moduli problems.

There are many interesting examples in characteristic $p > 0$ and in mixed characteristic. For example, the Serre-Tate theorem (\cite{KatzSerreTate}) shows that abelian varieties in characteristic $p > 0$ are unobstructed, meaning that they admit a formal lift to characteristic zero. In positive characteristic, the Lie algebraic models for deformation theory are more subtle. Brantner-Mathew (\cite{partition}) showed an equivalence between $k/k$ formal moduli problems in positive characteristic and the category of \emph{partition Lie algebras}. They also proved a similar result for $R/k$ formal moduli problems, where $R$ is a complete local Noetherian ring with residue field $k$.

This was further extended for $R/k$-formal moduli problems, where $R \to k$ is a finite purely inseparable extension of fields, by Brantner-Waldron (\cite{BrantnerWaldron}), and then more generally by Brantner-Magidson-Nuiten (\cite{foliations}). They established a correspondence between $R/k$ formal moduli problems and Lie algebraic models. In particular, they defined a monad on the category of chain complexes of $k$-modules over $L_{k/R}^\vee [1]$, where $L_{k/R}^\vee [1]$ is the shifted dual of the cotangent complex, which we call the \emph{partition Lie algebroid monad}, and the free algebras can be described via the equivariant cohomology of the partition complex. The $\infty$-category of \emph{partition Lie algebroids} is the category of algebras over this monad.

\begin{theorem}[\cite{foliations}, Theorem 5.26]
	There is an equivalence of $\infty$-categories between $R/k$ formal moduli problems and partition Lie algebroids.
\end{theorem}

Our aim is to establish a similar equivalence for functors encoding deformations over nilpotent but not necessarily Artinian bases.

\begin{example}
	Given any animated $k$-algebra $B$ equipped with a map $B \to k$, we have a $k/k$ formal moduli problem given by the functor $\Map_{\CAlg_{k/ /k}^{\an}} (B, -)$. Following \cref{example-jetscheme}, the deformations of the trivial $k$-valued jet on $X$ are given by $\Map_{\CAlg_{k/ /k}^{\an}} (A_\infty, -)$. The associated partition Lie algebra has chain complex $L_{k/A_\infty}^\vee [1]$.

	An important insight is that this chain complex has a natural \emph{topological structure}. That is, $L_{k/A_\infty}$ is not finite, so its dual can be written as an inverse limit of finite type chain complexes. We must replace $k$-vector spaces with a suitable topological analogue.
\end{example}

More generally, we will consider partition Lie algebroids whose underlying chain complex has some topological structure. To make this precise, we first need a suitable abelian category of topological modules.

\subsection{Ultrasolid modules}

Condensed mathematics is a new framework recently introduced by Clausen and Scholze (see \cite{condensed} and \cite{analytic}) that makes topology and algebra interact in a more well-behaved manner. For example, in the case of solid modules, this provides a new generalisation of complete topological modules that forms a complete and cocomplete abelian category (\cite{condensed}, Proposition 7.5). Solid modules over $\QQ$ of $\FF_p$ admit a particularly nice description.

\begin{prop}[\cite{condensed_localisation}, Theorem 1.1.4]
	\label{prop-nicesoliddescription}
	If $k = \QQ$ or $\FF_p$, the category of solid $k$-modules admits compact projective generators of the form $\prod_I k$, where $I$ is any set.
\end{prop}

However, this is no longer true over more general fields (\cite{goodtheorysolid}). In \cite{ultrasolid} we introduced the category of \emph{ultrasolid modules}, where we use \cref{prop-nicesoliddescription} as a definition rather than a property. We recall this construction below.

\begin{defi}
	We write $\Pro (\Vect_k ^\omega)$ for the pro-completion (\cite{categories}, Definition 6.1.1) of finite-dimensional vector spaces. This is the category obtained by formally adding cofiltered limits to $\Vect_k^\omega$. By duality, this is equivalent to $\Vect_k ^\op$, so its objects are of the form $\prod_I k$ for some set $I$ and $\Hom (\prod_I k, k) = \bigoplus_I k$. The product $\prod_I k$ carries a profinite structure by writing it as $\varprojlim _{J \subset I} \bigoplus_J k$, where the limit is taken over all finite subsets of $I$.

	The category $\Solid_k ^\heartsuit$ of \emph{ultrasolid $k$-modules} is the ($1$-categorical) sifted cocompletion (\cite{sifted}, \S 2) of $\ProVect$. That is, $\Solid_k^\heartsuit$ is the category obtained by freely adding filtered colimits and reflexive coequalizers to $\Pro(\Vect_k^\omega)$. This can be identified with finite-product-preserving functors $(\ProVect)^{\op} \to \Ab$ that satisfy a left Kan extension condition (\cite{ultrasolid}, Definition 2.4).

	We will write $\Solid_k$ for the $\infty$-category of chain complexes of ultrasolid modules.
\end{defi}

This gives a complete and cocomplete abelian category with compact projective generators of the form $\prod_I k$ for any field $k$. 

\subsection{Nilpotent formal moduli problems}

We now return to our goal of considering deformations over nilpotent bases. As before, we fix a ring $R$ and a map $R \to k$, where $k$ is a field.

\begin{defi}
	An object $A \in \CAlg_{R/ /k}^{\an}$ is \emph{bounded and nilpotent} if
	\begin{enumerate}
		\item $\pi_i (A) \cong 0$ for $i \gg 0$.
		\item $\pi_0 (A) \twoheadrightarrow k$ is surjective and $\mathfrak{m} := \ker (\pi_0 (A) \to k)$ is nilpotent; that is, $\mathfrak{m}^n = 0$ for $n \gg 0$.
	\end{enumerate}
	We will write $\CAlg_{R/ /k}^{\an, \bNil} \subset \CAlg_{R/ /k}^{\an}$ for the full subcategory of bounded and nilpotent algebras.
\end{defi}

\begin{example}
	\label{example-sqz}
	Given $V \in \Mod_{k, \geq 0}$, we can define the \emph{trivial square-zero extension} $k \oplus V$. That is, an animated ring whose underlying chain complex is $k \oplus V$ with trivial multiplication on $V$. If $V$ is eventually coconnective and not of finite type, the animated ring $k \oplus V$ is bounded and nilpotent but not Artinian.

	By the universal property of the cotangent complex, for any $A \in \CAlg_{R/ /k}^{\an}$, there is an equivalence
	\begin{align*}
		\Map_{\CAlg_{R/ /k}^{\an}} (A, k \oplus V) \simeq \Map_{\Mod_k} (L_{A/R} \otimes_A k, V)
	\end{align*}
	Thus, for any map $\alpha: L_{k/R}[-1] \to V$, we can define the \emph{square-zero extension} as the pullback diagram
	\[
		\begin{tikzcd}
			k \oplus_\alpha V \arrow{r} \arrow{d} & k \arrow{d}{\alpha} \\
			k \arrow{r}{0} & k \oplus V[1]
		\end{tikzcd}
	\]
	We have that $\fib (k \oplus_\alpha V \to k) \simeq V$, and there is trivial multiplication on this fibre. Hence, if $V$ is eventually coconnective, $k \oplus_\alpha V$ is bounded and nilpotent, but not Artinian unless $V$ is of finite type.
\end{example}

\begin{defi}
	A \emph{nilpotent formal moduli problem} is a functor $F : \CAlg_{R/ /k}^{\an, \bNil} \to \Space$ such that
	\begin{enumerate}
		\item $F(k) \simeq *$.
		\item $F$ preserves $\pi_0$-surjective pullbacks.
		\item $F$ preserves $\kappa$-filtered colimits for some cardinal $\kappa$.
	\end{enumerate}
	We will write $\Moduli_{R/k}^\blacksquare$ for the $\infty$-category of nilpotent formal moduli problems.
\end{defi}

\begin{construction}[Tangent complex]
	\label{construction-tangentcomplex}
	Let $F: \CAlg_{R/ /k}^{\an, \bNil} \to \Space$ be a nilpotent formal moduli problem. Write $\Mod_{k, \geq 0}^{< \infty} \subset \Mod_{k, \geq 0}$ for the full subcategory of eventually coconnective chain complexes. We have a functor
	\begin{align*}
		T_{k/F}[1] : (\Mod_{k, \geq 0}^{< \infty})_{L_{k/R}[-1]/} \to \Space
	\end{align*}
	given by
	\begin{align*}
		(L_{k/R}[-1] \xrightarrow{\alpha} V) \mapsto F(k \oplus_\alpha V	)
	\end{align*}
	that is accessible and preserves $\pi_0$-surjective pullbacks. We can identify $T_{k/F}[1] $ with an object in $(\Solid_k)_{/L_{k/R}^\vee [1]}$ (\cref{lemma-cotcomplexasfunctor}), which we will call the \emph{tangent complex} of $F$.
\end{construction}

We will define a sifted-colimit-preserving monad $\Lie^{\pi}_{\Delta, k/R}$ on $(\Solid_k)_{/L_{k/R}^\vee [1]}$, which we will call the \emph{partition Lie algebroid monad}. We use the same name as \cite{foliations}, since whenever $R \to k$ has finite type cotangent complex, this coincides with the partition Lie algebroid monad defined in (op.\ cit) on the subcategory $(\Mod_k)_{/L_{k/R}^\vee [1]} \subset (\Solid_k)_{/L_{k/R}^\vee [1]}$.

However, notice that when $R \to k$ does not have finite type cotangent complex, $L_{k/R}^\vee [1]$ has nontrivial topological structure. We will write $\LieAlgd_{\Delta, k/R}^{\pi, \blacksquare}$ for the $\infty$-category of partition Lie algebroids on $(\Solid_k)_{/L_{k/R}^\vee [1]}$.

\begin{theorem}
	The tangent complex (\cref{construction-tangentcomplex}) refines to an equivalence of $\infty$-categories
	\begin{align*}
		T_{k/-} [1] : \Moduli_{R/k}^\blacksquare \xrightarrow{\simeq} \LieAlgd_{\Delta, k/R}^{\pi, \blacksquare}
	\end{align*}
\end{theorem}

\begin{rem}
	We will (informally) justify the connection between formal moduli problems and partition Lie algebroids. The square-zero extension in \cref{example-sqz} gives a functor $\sqz: (\Mod_{k, \geq 0})_{L_{k/R}[-1]/} \to \CAlg_{R/ /k}^{\an, \surj}$. This admits a left adjoint $\cot : \CAlg_{R/ /k}^{\an, \surj} \to (\Mod_{k, \geq 0})_{L_{k/R}[-1] /}$ given by
	\begin{align*}
		A \mapsto (L_{k/R}[-1] \to L_{k/A}[-1])
	\end{align*}
	We can describe $\CAlg_{R/ /k}^{\an, \bNil} \subset \CAlg_{R/ /k}^{\an}$ as the smallest subcategory that contains the essential image of $\sqz : (\Mod_{k, \geq 0}^{< \infty})_{L_{k/R}[-1]/} \to \CAlg_{R/ /k}^{\an, \surj}$ and is closed under $\pi_0$-surjective pullbacks. Given that formal moduli problems commute with $\pi_0$-surjective pullbacks, all values of a formal moduli problem can be computed by knowing the values at square-zero extensions and the maps between them. By the adjunction,
	\begin{align*}
		\Map_{\CAlg_{R/ /k}^{\an}} (k \oplus_\alpha V , k \oplus_\beta W) \simeq \Map_{(\Mod_k)_{L_{k/R}[-1]/}} (\cot(k \oplus_\alpha V), W)
	\end{align*}
	so we will be interested in the comonad $\cot \circ \sqz $ on $(\Mod_{k, \geq 0})_{L_{k/R}[-1]/}$. Dualising, this gives a monad on $(\Solid_{k, \leq 0}^{\Pro})_{/L_{k/R}^\vee [1]}$, where $\Solid_{k, \leq 0}^{\Pro} \subset \Solid_k$ is the full subcategory of coconnective chain complexes with profinite homology. We will show that this uniquely extends to a sifted-colimit-preserving monad on $(\Solid_k)_{/L_{k/R}^\vee [1]}$, which is precisely the partition Lie algebroid monad.
\end{rem}

Given $A \in \CAlg_{R/ /k}^{\an, \surj}$, the arrow of chain complexes $\cot^\vee (A) \simeq (L_{k/A}^\vee[-1] \to L_{k/R}^\vee[-1])$ can be endowed with the structure of a partition Lie algebroid. This gives a functor $\KD : \CAlg_{R/ /k}^{\an, \surj} \to (\LieAlgd_{\Delta, k/R}^{\pi, \blacksquare})^{\op}$.

For our proof, we will use the formalism in \cite{DAGX}. A key ingredient is showing that the above functor restricts to a fully faithful embedding on $\CAlg_{R/ /k}^{\an, \bNil}$.

In \cite{foliations}, this is done by showing that the adjunction $\cot : \CAlg_{R/ /k}^{\an, \surj} \leftrightarrows (\Mod_{k, \geq 0})_{L_{k/R}[-1]/}: \sqz$ restricts to a comonadic adjunction between a class of complete objects and finite type chain complexes (\cite{foliations}, Theorem 1.10). In our setting, this will not be possible due to infinite objects not behaving well with completion. For example, the completion of the ring $k[x_1, x_2, \dots]$ at the maximal ideal $\mathfrak{m} := (x_1, x_2, \dots)$ is not $\mathfrak{m}$-adically complete (\cite{completions}, Example 1.8).

We will simplify this construction by fixing the order of nilpotence. That is, for a fixed $n$, we will consider maps $A \twoheadrightarrow k$ such that $\fib (A \to k)$ is $(n+1)$-nilpotent.

\begin{construction}[Animated $(n+1)$-nilpotent surjections]
	For any $n \in \NN \cup \{\infty\}$, we will construct an $\infty$-category $\CAlg_R^{\an, \surj, \leq n}$ satisfying the following properties:
	\begin{enumerate}
		\item Let $\catC$ be the category of surjections of discrete $R$-algebras $A \twoheadrightarrow B$ such that $\mathfrak{m}  = \ker (A \to B)$ satisfies $\mathfrak{m}^{n+1} = 0$. Then, there is a fully faithful embedding $\catC \subset \CAlg_{R}^{\an, \surj, \leq n}$, and the image of this embedding generates $\CAlg_R ^{\an, \surj, \leq n}$ under sifted colimits.
		\item $\CAlg_R^{\an, \surj, \leq \infty} \simeq \CAlg_R^{\an, \surj}$, the $\infty$-category of arrows of animated $R$-algebras that are $\pi_0$-surjective.
		\item There is a monadic forgetful functor $\iota_n^\infty: \CAlg_R^{\an, \surj, \leq n} \to \CAlg_R^{\an, \surj}$. An object $(A \to B)$ in its essential image satisfies that for any classes $x_1, \dots, x_{n+1} \in \pi_* (\fib (A \to B))$, we have $x_1 \dots x_{n+1} = 0$ (although this is not a sufficient condition). Thus, we can informally think of objects in $\CAlg_{R}^{\an, \surj, \leq n}$ as arrows of animated rings such that the fibre is $(n+1)$-nilpotent, with extra structure expressing the coherence conditions for this nilpotence.
		\item More generally, if $m < n$, there is a monadic right adjoint $\iota_m^n : \CAlg_{R/ /k}^{\an, \surj, \leq m} \to \CAlg_{R/ /k}^{\an, \surj, \leq n}$. Informally, an $(m+1)$-nilpotent map is $(n+1)$-nilpotent.
		\item For $n < \infty$, there is a comonadic adjunction $\cot_{\leq n} : \CAlg_{R/ /k}^{\an, \surj, \leq n} \leftrightarrows (\Mod_{k, \geq 0})_{L_{k/R}[-1]/}: \sqz $.
	\end{enumerate}
\end{construction}

We will use the adjunction of the last point above to construct \emph{dually truncated partition Lie algebroids}.

\begin{construction}[Dually truncated partition Lie algebroids]
	For each $n \in \NN \cup \{\infty\}$, we will construct a sifted-colimit-preserving monad $\Lie_{\Delta, k/R}^{\pi, \weirdleq n}$ on $(\Solid_k)_{/L_{k/R}^\vee [1]}$ satisfying the following properties:
	\begin{enumerate}
		\item The restriction of $\Lie_{\Delta, k/R}^{\pi, \weirdleq n}$ to $(\Solid_{k, \leq 0}^{\Pro})_{/L_{k/R}^\vee [1]}$ is the monad $\cot_{\leq n}  ^\vee \circ \sqz (-^\vee)$.
		\item If $L_{k/R}$ is of finite type, the monad $\Lie_{\Delta, k/R}^\pi := \Lie_{\Delta, k/R}^{\pi, \weirdleq \infty}$ restricted to $(\Mod_k)_{/L_{k/R}^\vee [1]} \subset (\Solid_k)_{/L_{k/R}^\vee [1]}$ coincides with the partition Lie algebroid monad as defined in \cite{partition} and \cite{foliations}.
		\item For $m < n$, there is a map of monads $\Lie_{\Delta, k/R}^{\pi, \weirdleq m} \to \Lie_{\Delta, k/R}^{\pi, \weirdleq n}$, which assembles into an equivalence of monads
			\begin{align*}
				\varinjlim_n \Lie_{\Delta, k/R}^{\pi, \weirdleq n} \simeq \Lie_{\Delta, k/R}^{\pi}
			\end{align*}
			and hence an equivalence of $\infty$-categories
			\begin{align*}
				\varprojlim_n \LieAlgd_{\Delta, k/R}^{\pi, \weirdleq n, \blacksquare} \simeq \LieAlgd_{\Delta, k/R}^{\pi, \blacksquare}
			\end{align*}
			where $\LieAlgd_{\Delta, k/R}^{\pi, \weirdleq n, \blacksquare}$ is the category of $\Lie_{\Delta, k/R}^{\pi, \weirdleq  n}$-algebras on $(\Solid_k)_{/L_{k/R}^\vee [1]}$.
	\end{enumerate}
\end{construction}

We will then show that any $A \in \CAlg_{R/ /k}^{\an, \bNil}$ is in the essential image of $\iota_n^\infty : \CAlg_{R/ /k}^{\an, \surj,\leq n} \to \CAlg_{R/ /k}^{\an, \surj}$ for some $n$. The decomposition of the partition Lie algebroid monad as the filtered colimit of the dually truncated monads will allow us to compute pushouts and obtain a fully faithful embedding 
\begin{align*}
	\KD : \CAlg_{R/ /k}^{\an, \bNil} \to \left( \LieAlgd_{\Delta, k/R}^{\pi, \blacksquare} \right) ^{\op}
\end{align*}
Our methods have some unexpected consequences. Remember that if $(A \twoheadrightarrow B)$ is a map in $\CAlg_R^{\an}$ that is surjective in $\pi_0$, then we can consider the filtered animated ring $\adic (A \twoheadrightarrow B)$. Even if this map has nilpotent fibre, it is not clear a priori that this filtration should be complete, since we are working with the derived adic filtration. 

For example, the associated graded of $\adic (k \oplus V \twoheadrightarrow k)$, where $V \in \Vect_k$, is the free graded $k$-algebra on $L_{k/k \oplus V}[-1]$. This cotangent complex is nontrivial in degree $1$, so the filtration $\adic (k \oplus V \twoheadrightarrow k)$ is not finite.

\begin{theorem}
	Let $(A \twoheadrightarrow B)$ be a map such that $\fib (A \twoheadrightarrow B)$ is eventually coconnective and $\ker (\pi_0 (A) \to \pi_0 (B))$ is nilpotent. Then, $\adic (A \twoheadrightarrow B)$ is complete.
\end{theorem}

This will be proved by using the bar resolution in $\CAlg_R^{\an, \surj, \leq n}$.

We have the following layout for this paper:

\begin{itemize}
	\item In \cref{section-animatednilpotent} we introduce animated nilpotent surjections and study some of their properties.
	\item In \cref{section-boundednilpotent} we study maps of animated rings with nilpotent fibre, and how they can be approximated with nilpotent surjections. We also show that the adic filtration of a map with nilpotent kernel is complete.
	\item In \cref{section-duallytruncatedLie} we introduce dually truncated partition Lie algebroids, and use these categories to understand partition Lie algebroids.
	\item In \cref{section-nilpotentfmp} we prove the equivalence between nilpotent formal moduli problems and partition Lie algebroids in ultrasolid modules.
\end{itemize}

\section*{Acknowledgements}

This is the last paper of my PhD and I would thus like to acknowledge everyone who has made this possible.

I want to thank Lukas Brantner for his supervision and guidance through the years, and for always encouraging me in the pursuit of mathematical research. I am also grateful to André Henriques for his mathematical education, and James Newton for his support as my college advisor. I am thankful to Jack Kelly, Juan Esteban Rodríguez Camargo and Dustin Clausen for interesting discussions that helped greatly in writing this thesis.

I would like to thank Adrià Marín Salvador, Colin McCulloch and Samuel Moore for the many conversations and all the support they provided. I am also grateful to the Mathematical Institute at the University of Oxford, for allowing me to pursue my research and further my education for the last eight years.

I am hugely indebted to all the spaces that fought for more inclusive environments in academia: Mathematrix at Oxford, and all the conferences promoting women and queer researchers in mathematics like WINGs, QuINGs and $\Spec (\overline{\QQ} (2 \pi i))$. 

Additionally, I am grateful to all the programmes and resources that have been part of my education: Fundación Talento Matemático y Científico, the Spanish Mathematical Olympiad and PROMYS Europe. Special thanks to Vicky Neale, who taught me how joyful research can be and is missed dearly.

This thesis would not have been possible without the support of countless people in my personal life. I want to thank Rafael, Lindsey, Rowena, Ella, Maryam, Nat, Rosemary, Katharine, Hannah, Debbie, Denise, Dani, Albert, Santiago, and many others for their support and encouragement, especially in the last few months of my writing. And Rachel, who made every day brighter and reminded me of the beautiful things. I am grateful to my parents and to Mariluz, for encouraging me to dream.

I am also thankful to all the wonderful people at Oxford Proud Voices. You have been my home, and being welcomed in this wonderful community has given me the courage to keep going. Finally, thank you to all the activists at Oxford for Trans Rights and many other groups, that keep our city safe and are always willing to stand up and fight back.

\section{Preliminaries}

\subsection{Filtered and graded objects}

We use this section as a reminder on filtered and graded objects, which we will later use in our study of nilpotent deformation problems. For the rest of this section, we fix a stable $\infty$-category $\catC$. In later applications, we will take $\catC = \Mod_k$.

\begin{defi}
	Let $N(\ZZ_{\geq 0})$ be the nerve of the poset $\ZZ_{\geq 0}$. The category of \emph{filtered objects in $\catC$} is defined as
	\begin{align*}
		\Fil (\catC) := \Fun (N(\ZZ_{\geq 0})^{\op}, \catC)
	\end{align*}
	Given a filtered object $X \in \Fil(\catC)$ and $i \in \ZZ_{\geq 0}$, we will often write $F^i X := X(i)$.
\end{defi}

\begin{rem}
	If $\catC$ is presentably symmetric monoidal, we can endow $\Fil(\catC)$ with a symmetric monoidal structure given by Day convolution, where we are using the additive monoidal structure on $\ZZ_{\geq 0}$.
\end{rem}

\begin{defi}
	Fix $n > 0$. We will write
	\begin{align*}
		\Fil^{\leq n} (\catC) := \Fun (\{0 \to \dots \to n \}^{\op}, \catC)
	\end{align*}
	Notice that we have a fully faithful embedding $\Fil^{\leq n} (\catC) \subset \Fil (\catC)$ whose essential image is given by those $X \in \Fil(\catC)$ such that $X(i) \simeq 0$ for $i > n$.
\end{defi}

\begin{rem}
	We can identify $\Fil^{\leq n} (\catC)$ with a localisation of $\Fil (\catC)$ whose kernel consists of $X \in \Fil(\catC)$ such that $\Gr^i (X) \simeq 0$ for $i\leq n$, so the left adjoint localisation functor inverts those maps inducing an equivalence in the first $n$ parts of the associated graded.
\end{rem}

\begin{rem}
	\label{rem-filmonoidal}
	We can endow $\Fil^{\leq n} (\catC)$ with the structure of a symmetric monoidal category. Write $(-)_{\leq n} : \Fil (\catC) \to \Fil^{\leq n} (\catC)$ for the localisation functor.

	The kernel of this localisation is an ideal, so by (\cite{HA}, Proposition 2.2.1.9), we can endow $\Fil^{\leq n} (\catC)$ with a symmetric monoidal structure $\otimes^{\leq n}$ given by
	\begin{align*}
		X \otimes^{\leq n} Y \simeq (X \otimes Y) _{\leq n}
	\end{align*}
	which makes the localisation functor symmetric monoidal.
\end{rem}

\begin{defi}
	Let $N(\ZZ_{\geq 0}^{\ds})$ be the category indexed by the non-negative integers with only identity morphisms. The category of \emph{graded objects in $\catC$} is defined as
	\begin{align*}
		\Gr (\catC) := \Fun (N(\ZZ_{\geq 0}^{\ds}), \catC)
	\end{align*}
	Given a graded object $X \in \Gr (\catC)$ and $i \in \ZZ_{\geq 0}$, we will often write $X_i := X(i)$.
\end{defi}

\begin{rem}
	Similar to filtered objects, if $\catC$ is symmetric monoidal, we can endow $\Gr (\catC)$ with a symmetric monoidal structure given by Day convolution, where we are using the additive monoidal structure on $\ZZ_{\geq 0}^{\ds}$. Given $X,Y \in \Gr (\catC)$,
	\begin{align*}
		(X \otimes Y)_i \simeq \bigoplus_{j + k = i} X_j \otimes Y_k
	\end{align*}
\end{rem}

\begin{defi}
	Fix $n > 0$. We will write
	\begin{align*}
		\Gr^{\leq n} (\catC) := \Fun (\{0, \dots, n \}, \catC)
	\end{align*}
	There is a fully faithful embedding $\Gr^{\leq n} (\catC) \to \Gr (\catC)$, whose essential image is given by those $X \in \Gr (\catC)$ such that $X_i \simeq 0$ for $i > n$.
\end{defi}

\begin{rem}
	Just as in \cref{rem-filmonoidal}, $\Gr^{\leq n} (\catC)$ can be endowed with a symmetric monoidal structure through localisation. This is essentially the symmetric monoidal product on $\Gr (\catC)$ where we discard anything in degree $> n$.
\end{rem}

\begin{defi}
	We say $X \in \Fil (\catC)$ is \emph{complete} if $\varprojlim F^i X \simeq 0$.

	Given $X \in \Fil (\catC)$, we have an \emph{associated graded} $\Gr (X)$ defined as
	\begin{align*}
		\Gr(X)_i := \cofib (F^{i+1} X \to F^i X)
	\end{align*}
	This gives a symmetric monoidal functor
	\begin{align*}
		\Gr : \Fil (\catC) \to \Gr (\catC)
	\end{align*}
	which restricts to a symmetric monoidal functor
	\begin{align*}
		\Gr: \Fil^{\leq n} (\catC) \to \Gr^{\leq n} (\catC)
	\end{align*}
\end{defi}

\begin{rem}
	If $X, Y \in \Fil (\catC)$ are complete, the map $X \to Y$ is an equivalence if and only if the map $\Gr (X) \to \Gr(Y)$ is an equivalence.

	Given $X,Y \in \Fil^{\leq n} (\catC)$, they are obviously complete, so equivalences in $\Fil^{\leq n} (\catC)$ are detected by the associated gradeds.
\end{rem}

\subsection{Poset topology}

We will use the bar resolution to compute the cotangent fibre and adic filtration of certain derived rings. We will be able to phrase these constructions in terms of poset topology.

\begin{defi}
	Let $P$ be a poset. Then, we can construct a simplicial set $\Delta (P)$ as follows: the $n$-simplices of $\Delta(P)$ are chains $x_0 \leq \dots \leq x_n$ in $P$. Face maps delete elements in the chain, and degeneracy maps insert repeated elements in the chain.

	We will not generally distinguish between $\Delta(P)$ and its geometric realisation.
\end{defi}

Notice that $\Delta(P)$ is naturally a simplicial complex, since the vertices of a simplex uniquely determine it.

\begin{defi}
	Let $P$ be a poset. We say that $P$ is \emph{atomised} if every element of $P$ lies above some minimal element and every finite set of minimal elements with an upper bound has a join.

	If $P$ is atomised, we call its minimal elements atoms and we will denote the set of atoms by $\mathcal{A} (P)$. We will say that a subset $S \subset  \mathcal{A}(P)$ \emph{generates $P$} if $P_{\geq S}$ is empty.
\end{defi}

\begin{prop}[\cite{atomised}, Theorem 2.3]
	\label{prop-atomised}
	Let $P$ be an atomised poset such that no $k$ atoms generate it. Then, $\Delta(P)$ is $(k-2)$-connected.
\end{prop}

One construction we will see often is the unreduced suspension of a poset.

\begin{defi}
	Let $X$ be a topological space. Then, its \emph{unreduced suspension} $X^\diamond$ is the topological space $(X\times \Delta^1)/_{\stackrel{(x,0)\sim (x',0)}{_{(x,1)\sim (x',1)}}}$.
\end{defi}

\begin{lemma}[\cite{partition}, Lemma 2.17]
	\label{prop-unreduced}
	Let $P$ be a poset, and let $\tilde{P} := P \cup \{\hat{1}\}$, where $\hat{1} \geq x$ for all $x \in P$. Then, $\Delta(P)^\diamond$ is homotopy equivalent to the simplicial set whose $n$-simplices are given by
	\begin{align*}
		\{x_0 \leq \dots \leq x_n = \hat{1} | x_i \in \tilde{P} \} \bigsqcup \{*\}
	\end{align*}
	The face maps preserve the basepoint and delete elements in the chain, sending invalid chains to the basepoint. The degeneracy maps insert repeated elements in the chain.
\end{lemma}

\section{Animated nilpotent surjections}
\label{section-animatednilpotent}

We will introduce a category of \emph{animated $n$-nilpotent surjections}. Informally, these are maps $A \twoheadrightarrow B$ of animated rings such that the fibre is $n$-nilpotent, together with coherence conditions expressing this nilpotence. This has the advantage of having simpler mapping spaces.

\subsection{Basic properties}
\label{subsection-basicproperties}

We will now fix a commutative discrete ring $R$.

\begin{defi}
	Let $n > 0$. We will write $\Poly_R^{\surj, \leq n}$ for the $1$-category of pairs of $R$-algebras of the form
	\begin{align*}
		\LSym_R^{\leq n} (V) \otimes_R \LSym_R^* (W) \xrightarrow{V \mapsto 0} \LSym_R^* (W)
	\end{align*}
	where $V,W$ are finite free $R$-modules, and $\LSym_R^{\leq n} (V)$ is the free $R$-algebra on $V$ quotiented by the ideal generated by $V^{\otimes n+1}$.

	The category of \emph{animated $(n+1)$-nilpotent surjections of $R$-algebras} is the $\infty$-category
	\begin{align*}
		\CAlg_R^{\an, \surj, \leq n} := \mathcal{P}_\Sigma (\Poly_R^{\surj, \leq n})
	\end{align*}
	In other words, this is the full subcategory of $\Fun ((\Poly_R^{\surj, \leq n})^{\op}, \Space)$ spanned by the finite-product-preserving functors.

	For $n = \infty$, we will write $\CAlg_R^{\an, \surj} := \CAlg_R^{\an, \surj, \leq \infty}$.
\end{defi}

\begin{rem}
	This definition is very much inspired by the category of \emph{animated surjections} as introduced in \cite{mao}. A map $R \to S$ of animated rings is an animated surjection if the map $\pi_0 (R) \to \pi_0 (S)$ is surjective. It turns out that this category is compactly generated by projectives of the form
	\begin{align*}
		\LSym_\ZZ^* (V) \otimes_\ZZ \LSym_\ZZ^* (W) \xrightarrow{V \mapsto 0} \LSym_\ZZ^* (W) 
	\end{align*}
	where $V$ and $W$ are finite free abelian groups (\cite{mao}, Theorem 1.12).

	Hence, our construction gives a category of surjections where the fibre is $(n+1)$-nilpotent, but there are coherence conditions that restrict these maps. For example, for $n = 1$, write $(R \oplus Rx \oplus Ry \twoheadrightarrow R)$ for the free $2$-nilpotent surjection on generators $x$ and $y$ in degree zero. Any map of $2$-nilpotent surjections from $(R \oplus Rx \oplus Ry \twoheadrightarrow R)$ to $(R \oplus Rz \twoheadrightarrow R)$ (where $z$ is taken in homological degree $i > 0$) is trivial. For $n = \infty$, this mapping space describes the free partition Lie algebra on two generators, or, more classically, the André-Quillen cohomology of $k[x,y]/(x^2, y^2, xy)$.

	It will be useful to consider this category not as $(n+1)$-nilpotent surjections $A \twoheadrightarrow B$, but rather as pairs $(A,I)$ where $I$ is the kernel ideal of the surjection. This allows us to compute limits, since a limit of surjections might not be surjective.
\end{rem}

\begin{construction}[Forgetful functor]
	\label{construction-forgetsurj}
	We will now construct a forgetful functor $F: \CAlg_R^{\an, \surj, \leq n} \to \Fun (\Delta^1, \Mod_{R, \geq 0})$. Informally, this will take a	surjection $A \twoheadrightarrow B$ to $\fib (A \to B) \to A$.

	Let $\Arr^{\st} (\Vect_R^\omega)$ be the category of arrows of the form $V \xrightarrow{i} V \oplus W$, where $i$ is the canonical inclusion map. We will often call these \emph{standard arrows}. By (\cite{mao}, Corollary 2.13), $\Fun(\Delta^1, \Mod_{R, \geq 0}) \simeq \mathcal{P}_\Sigma (\Arr^{\st} (\Vect_R^\omega))$.

	We have a functor $\Arr^{\st} (\Vect_R^\omega) \to \Poly_R^{\surj, \leq n}$ given by
	\begin{align*}
		(V \mapsto V \oplus W) \mapsto (\LSym^{\leq n}_R (V) \otimes_R \LSym^* _R (W) \to \LSym^* _R(W) )
	\end{align*}
	Since it preserves coproducts, it induces a functor $F: \CAlg_R^{\an, \surj, \leq n} \to \Fun (\Delta^1, \Mod_{R, \geq 0})$ on these subcategories of presheaves by composition.
\end{construction}

\begin{prop}
	\label{prop-surjmonadic}
	The functor $F: \CAlg_R^{\an, \surj, \leq n} \to \Fun(\Delta^1,  \Mod_{R, \geq 0})$ from \cref{construction-forgetsurj} preserves sifted colimits and is monadic. It is given on $\Poly_R^{\surj, \leq n}$ by
	\begin{align*}
		\left( \LSym^{\leq n}_R (V) \otimes_R \LSym^* _R (W) \to \LSym_R^* (W) \right) & \mapsto \\ 
																						 & \left(\bigoplus_{0 < i \leq n} \LSym^i (V) \right) \otimes_R \LSym^* (W) \to \LSym_R^{\leq n} (V) \otimes_R \LSym_R^* (W) 
	\end{align*}
\end{prop}

\begin{proof}
	$F$ commutes with all limits and sifted colimits, since these are computed pointwise in $\mathcal{P}_\Sigma (\catC)$ for a small category $\catC$ (where we consider $\mathcal{P}_\Sigma(\catC)$ as a subcategory of $\Fun (\catC^{\op}, \Space)$). Additionally, the functor $\Arr^{\st} (\Vect_R^\omega) \to \Poly_R^{\surj, \leq n}$ is essentially surjective, so $F$ must be conservative.	The functor is then monadic by the crude version of the Barr-Beck-Lurie Theorem (\cite{HA}, Theorem 4.7.0.3).

	One can easily check that the functor is as given on compact projectives.
\end{proof}

\begin{construction}[Target ring]
	\label{construction-targetring}
	We have a sifted-colimit-preserving functor $\CAlg_R^{\an, \surj, \leq n} \to \CAlg_R^{\an}$ given on compact projectives by
	\begin{align*}
		\left( \LSym_R^{\leq n} (V) \otimes_R \LSym_R^* (W) \to \LSym_R^* (W) \right) \mapsto \LSym_R^* (W)
	\end{align*}
	This also commutes with coproducts of compact projectives, so it commutes with all colimits. Informally, this is the functor $(A \twoheadrightarrow B) \mapsto B$. We will often refer to this as the \emph{target ring}.
\end{construction}

\begin{rem}
	\label{rem-pairs}
	By \cref{prop-surjmonadic}, we can think of an object in $\CAlg_R^{\an, \surj, \leq n}$ as an arrow $I \to A$ in $\Mod_{R, \geq 0}$ with the structure of a monad algebra. We will often write this as a pair $(A,I)$ (so the map $I \to A$ is omitted) or as an arrow $A \twoheadrightarrow B$, where $B := \cofib (I \to A)$. Notice that $B$ has the structure of an animated ring (\cref{construction-targetring}).

	We will later see that $A$ also has the structure of an animated ring, plus extra structure accounting for the nilpotence of the kernel ideal (\cref{prop-iotamonadic}).
\end{rem}

It turns out the compact projectives used to build $\CAlg_R^{\an, \surj, \leq n}$ are precisely the free algebras on the standard arrows.

\begin{lemma}
	\label{lemma-projfree}
	Let $V, W \in \Vect_R^\omega$. Then, the object $\LSym^{\leq n}_R (V) \otimes_R \LSym^* (W) \to \LSym^* (W)$ in $\CAlg_R^{\an, \surj,\leq n}$ can be identified with the free algebra on $V \to V \oplus W$ with respect to the monad in \cref{prop-surjmonadic}.
\end{lemma}

\begin{proof}
	Let $(A \twoheadrightarrow B) \in \CAlg_R^{\an, \surj,\leq n}$. Then, we just need to establish an equivalence
	\begin{align*}
		\Map_{\CAlg_R^{\an, \surj, \leq n}} (\LSym^{\leq n} (V) \otimes_R \LSym^* (W) \to \LSym^* (W), A \to B) \simeq \Map_{\Fun (\Delta^1, \Mod_{R,	\geq 0})} (V \to V \oplus W, I \to A)
	\end{align*}
	where $I := \fib (A \to B)$.

	Notice that both sides commute with sifted colimits in the second variable due to the relevant objects being compact projective, so we only need to show this for $(A \twoheadrightarrow B) \in \Poly_R^{\surj,\leq n}$, and then it is clear.
\end{proof}

\begin{notation}
	We will abuse notation and write $\LSym^{\leq n}_R : \Fun(\Delta^1, \Mod_{R, \geq 0}) \to \Fun (\Delta^1, \Mod_{R, \geq 0})$ for the monad coming from \cref{prop-surjmonadic}, and we will write the free algebra on an arrow $V \to W$ as $\LSym^{\leq n}_R (V \to W)$. Notice that with this notation, for a standard arrow $V \to V \oplus W$ in $\Vect_R^\omega$,
	\begin{align*}
		\LSym^{\leq n}_R (V \to V \oplus W) \simeq (\LSym_R^{\leq n} (V) \otimes_R \LSym_R^* (W) \xrightarrow{V \mapsto 0} \LSym_R^* (W))
	\end{align*}
\end{notation}

\begin{lemma}
	\label{lemma-animatedsubcat}
	There is a fully faithful embedding $\CAlg^{\an}_R \to \CAlg_R^{\an, \surj, \leq n}$ given by the sifted-colimit-preserving extension of
	\begin{align*}
		\LSym_R^* (V) \mapsto \LSym^{\leq n}_R (0 \to V)
	\end{align*}
	for $V \in \Vect_R^\omega$. Its essential image are those $(A \twoheadrightarrow B)$ such that the map $A \to B$ is an equivalence of chain complexes.

	Additionally, for any $A \twoheadrightarrow B \in \CAlg_R^{\an, \surj, \leq n}$ and $C \xrightarrow{\simeq} C \in \CAlg_R^{\an} \subset \CAlg_R^{\an, \surj, \leq n}$ there is an equivalence of arrows of chain complexes
	\begin{align*}
		(A \twoheadrightarrow B) \sqcup (C \xrightarrow{\simeq} C) \simeq (A \otimes_R C \twoheadrightarrow B \otimes_R C)
	\end{align*}
\end{lemma}

\begin{proof}
	The functor being fully faithful is a direct consequence of (\cite{HTT}, Proposition 5.5.8.22). Now suppose $(A \xrightarrow{\simeq} A) \in  \CAlg_R^{\an, \surj, \leq n}$. By its bar resolution, we can write it as a geometric realisation of objects in $\CAlg_R ^{\an} \subset \CAlg_R^{\an, \surj, \leq n}$.

	For the part about coproducts, since the forgetful functor commutes with sifted colimits, it is enough to show it for $C \simeq \LSym^*_R (U)$ and $(A \twoheadrightarrow B) \simeq \LSym^{\leq n}_R (V \to V \oplus W)$ for $U,V,W \in \Vect_R^\omega$. This follows from the fact that $\LSym^{\leq n}_R : \Fun (\Delta^1, \Mod_{R, \geq 0} ) \to \CAlg_R^{\an, \surj, \leq n}$ commutes with coproducts.
\end{proof}

\begin{prop}
	\label{prop-surjheart}
	Let $(\CAlg_R^{\an, \surj, \leq n}) ^{\heartsuit} \subset \CAlg_R^{\an, \surj, \leq n}$ be the full subcategory consisting of those $A \twoheadrightarrow B$ such that $A,B \in \Mod_R^\heartsuit$. Then, $(\CAlg_R^{\an, \surj, \leq n})^\heartsuit$ is equivalent to the full subcategory of $\Fun (\Delta^1, \CAlg_R^\heartsuit)$ consisting of those arrows $A \to B$ that are surjective in $\pi_0$ and the kernel ideal is $(n+1)$-nilpotent.
\end{prop}

\begin{proof}
	Let $\catC \subset \Fun(\Delta^1, \CAlg_R^\heartsuit)$ be the full subcategory consisting of arrows $A \to B$ that are surjective and the kernel ideal is $(n+1)$-nilpotent.

	Then, the restriction of the Yoneda embedding gives us a functor $h: \catC \to \CAlg_R^{\an, \surj, \leq n}$. That is, $(A \to B) \mapsto \Hom_{\Fun (\Delta^1, \CAlg_R^\heartsuit)} (-, A \to B)$. We first show that it is faithful. Consider two maps $f,g: (A \to B) \to (C \to D)$ in $\catC$, which we can think of as maps $f,g: A \to C$ preserving the kernel ideal. If these two maps are not the same, there is some $a \in A$ such that $f(a) \neq g(a)$.

	We have a natural identification $h(A \to B) (R[x] \to R[x]) \cong A$ and $h(C \to D) (R[x] \to R[x]) \cong C$. The natural transformations $h(f)$ and $h(g)$ at the element $(R[x] \to R[x])$ simply reduce to the maps $f,g$, so it follows that $h(f) \neq h(g)$. Hence, $h$ is faithful.

	We will now show that it is full. Suppose we have a natural transformation $F: h(A \to B) \to h(C \to D)$. As before, $F$ at the value $(R[x] \to R[x])$ induces a map $f: A \to C$. Compatibility with the map $(R[x] \to R[x]) \to (R[x]/x^{n+1} \to R)$ ensures that this map preserves the kernel ideal. Additionally, compatibility with the maps $(R[x] \to R[x]) \xrightarrow{x \mapsto y +z} (R[y,z] \to R[y,z])$ and $(R[x] \to R[x]) \xrightarrow{x \mapsto yz} (R[y,z] \to R[y,z])$ implies that this is a ring map. It easily follows that $h(f) = F$.

	Finally, we prove that $h$ is essentially surjective on the category of objects $A \twoheadrightarrow B$ such that $A$ and $B$ are discrete. We first show that the object $A \twoheadrightarrow B$ in $\CAlg_R^{\an, \surj, \leq n}$ is $0$-truncated. By the bar resolution, it suffices to show that $\Map_{\Poly_R^{\surj, \leq n}} (\LSym^{\leq n} (V \to W), (A \to B))$ is $0$-truncated, since $0$-truncated spaces are closed under limits. But this is clear since $I \to A$ is $0$-truncated in $\Fun(\Delta^1, \Mod_{R, \geq 0})$, where $I = \ker (A \to B)$.

	Now consider $(A \to B)$ as a functor $F: (\Poly_R^{\surj, \leq n})^{\op} \to \Space$. Notice that for a standard arrow $V \to V \oplus W$, we have a natural identification $F(\LSym^{\leq n} (V \to V \oplus W)) \simeq \Map_{\CAlg_R^{\an, \surj, \leq n}} (\LSym^{\leq n} (V \to V \oplus W), (A \to B))$. Since $(A \to B)$ is $0$-truncated, $F$ is a functor whose image is in $\Set \subset \Space$.

	By \cref{construction-forgetsurj}, we can also identify $A \cong F(R[x] \to R[x])$ and $I \cong F(R[x]/x^{n+1} \to R)$, where $I = \fib (A \to B)$. Compatibility with the maps in $\Poly_R^{\surj, \leq n}$ endows $A$ with the structure of an $R$-algebra and $I$ with the structure of an $(n+1)$-nilpotent ideal. One can then easily check that this is precisely the Yoneda embedding of the constructed ring map.
\end{proof}

\begin{defi}
	Let $S \in \CAlg_R^{\an}$. We define the category $\CAlg_{R/ /S}^{\an, \surj, \leq n}$ by the pullback diagram
	\[
		\begin{tikzcd}
			\CAlg_{R/ /S}^{\an, \surj, \leq n} \arrow{r} \arrow{d} & * \arrow{d}{* \mapsto S} \\
			\CAlg_R^{\an, \surj, \leq n} \arrow{r}{\target} & \CAlg_R^{\an}
		\end{tikzcd}
	\]
	Informally, it is the subcategory of $\CAlg_R^{\an, \surj, \leq n}$ consisting of arrows of the form $A \twoheadrightarrow S$ and maps that reduce to the identity on the target ring.
\end{defi}

As a special case, $S = R$ will give a useful subcategory.

\begin{prop}
	\label{prop-nonunital}
	The functor $\CAlg_{R/ / R}^{\an, \surj, \leq n} \to \Mod_{R, \geq 0}$ given by
	\begin{align*}
		(A \twoheadrightarrow R) \mapsto \fib (A \to R)
	\end{align*}
	is monadic. The relevant monad is given by the sifted-colimit-preserving monad given on $\Vect_R^\omega$ by
	\begin{align*}
		V \mapsto \bigoplus_{0 < i \leq n} V^{\otimes i}_{\Sigma _i}
	\end{align*}
\end{prop}

\begin{proof}
	The functor is clearly conservative and preserves sifted colimits, so it is monadic. For the last part, we can use \cref{lemma-projfree} to identify $\LSym^{\leq n}_R (V \to V)$ with the free algebra on $V$.
\end{proof}

\begin{notation}
	We will write $\LSym_R^{\nu, \leq n}: \Mod_{R, \geq 0} \to \Mod_{R, \geq 0}$ for the resulting monad. Sometimes, we will abuse notation and write $\LSym_R^{\nu, \leq n} (V)$ for the image of the free algebra on $V$ in the category $\CAlg_R^{\an, \surj, \leq n}$. In other words, $\LSym_R^{\nu, \leq n} (V) \simeq \LSym_R^{\leq n} (V \to V)$.
\end{notation}

\subsection{Polynomial truncation}

\begin{construction}[Polynomial truncation functors]
	\label{construction-polyomialtruncation}
	Let $m \leq n$. Then, there is a functor $t_{\leq m} ^n : \Poly_R^{\surj, \leq n} \to \Poly_R^{\surj, \leq m}$ given by
	\begin{align*}
		\LSym^{\leq n}_R (V \to V \oplus W) \mapsto \LSym^{\leq m} _R (V \to V \oplus W)
	\end{align*}
	Informally, for an $(n+1)$-nilpotent surjection $A \twoheadrightarrow B$ with kernel ideal $I$, $t_{\leq m} ^n (A \twoheadrightarrow B) = (A/I^{m+1} \to B)$.

	Notice that by left Kan extending we get a colimit-preserving functor $t_{\leq m} ^n : \CAlg_R^{\an, \surj, \leq n} \to \CAlg_R^{\an, \surj, \leq m}$.

	We will write its right adjoint as $\iota_m ^n : \CAlg_R^{\an, \surj, \leq m} \to \CAlg_R^{\an, \surj, \leq n}$. More explicitly, if $X \in \CAlg_R^{\an, \surj, \leq m}$, we can write it as a finite-product-preserving functor $X: (\Poly_R^{\surj, \leq m})^{\op} \to \Space$. Then, $\iota_m^n(X) : (\Poly_R^{\surj, \leq n })^{\op} \to \Space$ is the composition of $X$ with $t_{\leq m}^n : \Poly_R^{\surj, \leq n} \to \Poly_R^{\surj, \leq m}$.

	Notice that sifted colimits of product-preserving functors are computed pointwise, so the functor $\iota_m^n$ commutes with sifted colimits.
\end{construction}

\begin{rem}
	\label{rem-iotapreservesmodule}
	Consider the functor $\Arr^{\st} (\Vect_R^\omega) \to \CAlg_R^{\an, \surj, \leq n}$ from \cref{construction-forgetsurj}. Then, there is a commutative diagram
	\[
		\begin{tikzcd}
			\Arr^{\st} (\Vect_R^\omega) \arrow{dr} \arrow{r} & \CAlg_R^{\an, \surj, \leq n} \arrow{d}{t_{\leq m}^n} \\
																			 & \CAlg_R^{\an, \surj, \leq m}
		\end{tikzcd}
	\]
	By the definition of these functors, this implies that there is a commutative diagram
	\[
		\begin{tikzcd}
			\CAlg_R^{\an, \surj, \leq m} \arrow{r}{\iota_m ^n} \arrow{dr}[swap]{\forget} & \CAlg_R^{\an, \surj, \leq n} \arrow{d}{\forget} \\
																									& \Fun (\Delta^1, \Mod_{R, \geq 0})
		\end{tikzcd}
	\]
	So the functor $\iota_m^n$ preserves the underlying arrow of modules. Due to this, when it is clear from context, we will omit the functor $\iota_m^n$. That is, given $(A \twoheadrightarrow B) \in \CAlg_R^{\an, \surj, \leq m}$, we will write $(A \twoheadrightarrow B) \in \CAlg_R^{\an, \surj, \leq n}$ for $\iota_m^n (A \twoheadrightarrow B)$.
\end{rem}

\begin{prop}
	\label{prop-iotamonadic}
	Let $m \leq n$. The functor $\iota_m^n$ is monadic.
\end{prop}

\begin{proof}
	By \cref{rem-iotapreservesmodule}, $\iota_m^n$ preserves the underlying arrow in $\Mod_{R, \geq 0}$. Since the forgetful functor is conservative, $\iota_m^n$ is conservative. It also commutes with all sifted colimits, so by the crude version of the Barr-Beck-Lurie Theorem (\cite{HA}, Theorem 4.7.0.3) we are done.
\end{proof}

\begin{rem}
	In particular, setting $n = \infty$, we have a monadic functor $\iota_m^\infty: \CAlg_R^{\an, \surj, \leq m} \to \CAlg_R^{\an, \surj, \leq \infty}$. The latter is just the category of animated surjections, supporting our earlier claim that an object $A \twoheadrightarrow B$ in $\CAlg_R^{\an, \surj, \leq m}$ is just an animated surjection with $(m+1)$-nilpotent kernel ideal and extra structure accounting for the coherence conditions of the nilpotence.
\end{rem}

\begin{rem}
	The functors $\iota_m^n: \CAlg_R^{\an, \surj, \leq m} \to \CAlg_R^{\an, \surj, \leq n}$ preserve the underlying modules, so by \cref{prop-surjheart} they restrict to functors $\iota_m^n: (\CAlg_R^{\an, \surj, \leq m})^\heartsuit \to (\CAlg_R^{\an, \surj, \leq n})^\heartsuit$. 

	This corresponds to the functor sending a surjective map of rings $(A \twoheadrightarrow B)$ with $(m+1)$-nilpotent kernel ideal to itself; that is, an $(m+1)$-nilpotent ideal is necessarily $(n+1)$-nilpotent.
\end{rem}

\subsection{Filtered and graded animated nilpotent surjections}

Fix a discrete ring $R$. We will construct appropriate filtered and graded analogues of $\CAlg_R^{\an, \surj, \leq n}$.

\begin{rem}
	Let $\Fil_{\free} \Mod_{R, \geq 0} \subset \Fil \Mod_{R, \geq 0}$ be the full subcategory generated under coproducts by objects of the form $R^{\leq i}$ where $F^j R^{\leq i} \simeq 0$ for $j > i$ and $F^j R^{\leq i} \simeq R$ for $j \leq i$. In other words, $\Fil_{\free} \Mod_{R, \geq 0}$ is the subcategory of finitely generated free filtered $R$-modules where each generator is of a certain filtered weight. By (\cite{mao}, Corollary 2.13), $\Fil \Mod_{R, \geq 0} \simeq \mathcal{P}_\Sigma (\Fil_{\free} \Mod_{R, \geq 0})$.
\end{rem}

\begin{notation}
	We will write $\Fil_{\free}^0 \Mod_{R, \geq 0} \subset \Fil_{\free} \Mod_{R, \geq 0}$ for the full subcategory of finite free modules in filtered weight $0$.

	We will write $\Fil_{\free}^{\geq 1} \Mod_{R, \geq 0} \subset \Fil_{\free} \Mod_{R, \geq 0}$ for the full subcategory of finite free modules all of whose generators are in filtered weight $ > 0$. Note that this is spanned under coproducts by objects of the form $R^{\leq i}$ where $i > 0$.
\end{notation}

\begin{defi}
	Let $n > 0$. We will write $\Fil \Poly_R^{\leq n}$ for the category of filtered rings of the form $\LSym_R^{\leq n} (V) \otimes_R \LSym_R^* (W)$, where $V \in \Fil_{\free} ^{\geq 1} \Mod_{R, \geq 0}$ and $W \in \Fil_{\free}^0 \Mod_{R, \geq 0}$.

	The category of \emph{animated filtered $(n+1)$-nilpotent surjections} is
	\begin{align*}
		\Fil\CAlg_R^{\an, \surj, \leq n} := \mathcal{P}_\Sigma (\Fil\Poly_R^{\leq n})
	\end{align*}
\end{defi}

We do not include a proof of the following since they are exactly the same as in \cref{subsection-basicproperties}.

\begin{prop}
	Let $n > 0$.
	\begin{enumerate}
		\item The sifted-colimit-preserving extension of the forgetful functor $\Fil\Poly_R^{\leq n} \to \Fil \Mod_{R, \geq 0}$ is monadic. We will write $\LSym_R^{\Fil,\leq n}$ for the resulting monad.
		\item Given $V \in \Fil_{\free}^{\geq 1} \Mod_{R, \geq 0}$ and $W \in \Fil_{\free}^0 \Mod_{R, \geq 0}$, we can identify $\LSym_R^{\leq n} (V) \otimes_R \LSym_R^* (W)$ with the free $\LSym_R^{\Fil, \leq n}$-algebra on $V \oplus W$.
	\end{enumerate}
\end{prop}

\begin{construction}[Adic filtration]
	Notice that there is a fully faithful embedding $\Poly_R^{\surj, \leq n} \subset \Fil\Poly_R^{\leq n}$, where for $V, W \in \Vect_R^\omega$, we identify $\LSym_R^{\leq n} (V \to V \oplus W)$ with $\LSym_R^{\Fil, \leq n} (V \oplus W)$, where we consider $V$ in filtered weight $1$ and $W$ in filtered weight $0$. Notice that this embedding commutes with finite coproducts.

	By (\cite{HTT}, Proposition 5.5.8.22 and Proposition 5.5.8.15), there is a fully faithful colimit-preserving embedding $\adic_{\leq n} : \CAlg_R^{\an, \surj, \leq n} \to \Fil\CAlg_R^{\an, \surj, \leq n}$, which we will call the \emph{adic filtration}. Notice that $F^j \adic_{\leq n}(-) \simeq 0$ for $j > n$.

	This admits a right adjoint $\aug: \Fil \CAlg_R^{\an, \surj, \leq n} \to \CAlg_R^{\an, \surj, \leq n}$.
\end{construction}

\begin{lemma}
	Let $R \in \Fil\CAlg_R^{\an, \surj, \leq n}$. Then, there is an isomorphism of arrows of chain complexes
	\begin{align*}
		\aug (R) \simeq (F^0 R, F^1 R)
	\end{align*}
	where we use the notation of \cref{rem-pairs}.
\end{lemma}

\begin{proof}
	For a standard arrow $V \to V \oplus W$, we have natural equivalences
	\begin{align*}
		\Map_{\Fun (\Delta^1, \Mod_{R, \geq 0})} (V \to V \oplus W, \aug(R)) & \simeq \Map_{\CAlg_R^{\an, \surj, \leq n}} (\LSym^{\leq n}_R (V \to V \oplus W), \aug (R)) \\
		& \simeq \Map _{\Fil \CAlg_R^{\an, \surj, \leq n}} (\LSym^{\Fil^{\leq n}}_R (V \to V \oplus W), R) \\
		& \simeq \Map_{\Fil^{\leq n} \Mod_{R, \geq 0}} (V \to V \oplus W, R) \\
		& \simeq \Map_{\Fun (\Delta^1, \Mod_{R, \geq 0})} (V \to V \oplus W, F^1 R \to F^0 R)
	\end{align*}
	Since all of the mapping spaces commute with sifted colimits in the first variable, this is true for all $V \to W$ in $\Fun (\Delta^1, \Mod_{R, \geq 0})$. We are now done by the Yoneda embedding.
\end{proof}

We will now construct a graded analogue.

\begin{defi}
	We will write $\Gr_{\free} \Mod_{R, \geq 0} \subset \Gr \Mod_{R, \geq 0}$ for the full subcategory generated under coproducts by graded modules of the form $R^i$, where $R^i_j \simeq 0$ for $i \neq j$ and $R^i_j \simeq R$ for $i = j$. In other words, these are the finitely generated free graded modules where each generator is homogeneous of a fixed degree.

	We will write $\Gr_{\free}^0 \Mod_{R, \geq 0}$ for the category of finitely generated free graded modules in degree $0$. Similarly, we will write $\Gr_{\free}^{\geq  1} \Mod_{R, \geq 0}$ for the finitely generated free graded modules such that every generator is homogeneous of degree $> 0$.
\end{defi}

\begin{defi}
	Let $n > 0$. We will write $\Gr\Poly_R^{\leq n}$ for the category of graded algebras of the form $\LSym_R^{\leq n} (V) \otimes_R \LSym_R^* (W)$, where $V \in \Gr_{\free}^{\geq 1} \Mod_{R, \geq 0}$ and $W \in \Gr_{\free}^0 \Mod_{R, \geq 0}$.

	The category of \emph{animated graded $(n+1)$-nilpotent surjections} is
	\begin{align*}
		\Gr\CAlg_R^{\an, \surj, \leq n} := \mathcal{P}_\Sigma (\Gr \Poly_R^{\leq n})
	\end{align*}
\end{defi}

We do not include a proof of the following as it is identical to \cref{subsection-basicproperties}.

\begin{prop}
	Let $n > 0$.
	\begin{enumerate}
		\item The sifted-colimit-preserving extension of the forgetful functor $\Gr\Poly_R^{\leq n} \to \Gr \Mod_{R, \geq 0}$ is monadic. We will write $\LSym_R^{\Gr,\leq n}$ for the resulting monad.
		\item Given $V \in \Gr_{\free}^{\geq 1} \Mod_{R, \geq 0}$ and $W \in \Gr_{\free}^0 \Mod_{R, \geq 0}$, we can identify $\LSym_R^{\leq n} (V) \otimes_R \LSym_R^* (W)$ with the free $\LSym_R^{\Gr, \leq n}$-algebra on $V \oplus W$.
	\end{enumerate}
\end{prop}

\begin{defi}
	We will write $\Gr : \Fil \CAlg_R^{\an, \surj, \leq n} \to \Gr \CAlg_R^{\an, \surj, \leq n}$ for the sifted-colimit-preserving extension of the functor $\Fil \Poly_R^{\leq n} \to \Gr \Poly_R^{\leq n}$ that takes the associated graded.
\end{defi}

\begin{defi}
	We will write $\Fil\Mod$ for the $\infty$-category of pairs $(A,M)$ where $A \in \CAlg_R^{\an}$ and $M \in \Fil \Mod_A$. We can construct this as the source of the biCartesian fibration that classifies the functor $A \mapsto \Fil \Mod_A$

	Similarly, let $\Gr\Mod$ be the $\infty$-category of pairs $(A,M)$ where $A \in \CAlg_R^{\an}$ and $M \in \Gr \Mod_A$.
\end{defi}

\begin{defi}[Filtered and graded tangent fibre]
	We construct a functor $\cot_{\leq n} : \Fil \CAlg_R^{\an, \surj, \leq n} \to \Fil \Mod$ and $\cot_{\leq n} : \Gr \CAlg_R^{\an, \surj, \leq n} \to \Gr\Mod$. Suppose $ A \in \Fil \Poly_R^{\leq n}$ or $\Gr \Poly_R^{\leq n}$. Then define
	\begin{align*}
		\cot_{\leq n} (A) := (\Gr^0 (A), I/I^2)
	\end{align*}
	where $I:= \fib (A \to \Gr^0 (A))$. If $V \in \Fil^{> 0}_{\free} \Mod_{R, \geq 0}$ and $W \in \Fil^0 _{\free} \Mod_{R, \geq 0}$, this gives $\cot_{\leq n} (\LSym_R^{\Fil, \leq n} (V \oplus W)) \simeq (\LSym_R^* (W), \LSym_R^* (W) \otimes_R V)$, and similar for the graded analogue.

	This clearly commutes with coproducts, so by (\cite{HTT}, Proposition 5.5.8.15) we get a colimit-preserving functor $\cot_{\leq n} : \Fil \CAlg_R^{\an, \surj, \leq n} \to \Fil \Mod$ and $\cot_{\leq n} : \Gr \CAlg_R^{\an, \surj, \leq n} \to \Gr \Mod$. Notice that in both cases we get $\Gr^0 \cot_{\leq n} (-) \simeq 0$.
\end{defi}

\begin{rem}
	By checking on compact projectives one has that for any $A' \in \Fil \CAlg_R^{\an, \surj, \leq n}$,
	\begin{align*}
		\Gr (\cot_{\leq n} (A')) & \simeq \cot_{\leq n} (\Gr (A')) \\
		\aug (\cot_{\leq n} (A')) & \simeq \cot_{\leq n} (\aug (A')) \\
	\end{align*}
\end{rem}

\begin{construction}[Polynomial truncation]
	Just as in the non-filtered case, for $m \leq n$ we have polynomial truncation adjunctions $t_{\leq m}^n: \Fil \CAlg_R^{\an, \surj, \leq n} \leftrightarrows \Fil\CAlg_R^{\an, \surj, \leq m} : \iota_m^n$ and $t_{\leq m}^n : \Gr \CAlg_R^{\an, \surj, \leq n} \leftrightarrows \Gr\CAlg_R^{\an, \surj, \leq m} : \iota_m^n$. The functors $\iota_m^n$ preserve the underlying filtered/graded module, and for $V \in \Fil \Mod_{R, \geq 0}$ and $W \in \Gr \Mod_{R, \geq 0}$,
	\begin{align*}
		t_{\leq m} ^n (\LSym_R^{\Fil, \leq n} (V)) & \simeq \LSym_R^{\Fil, \leq m} (V) \\
		t_{\leq m}^n (\LSym_R^{\Gr, \leq n} (W)) & \simeq \LSym_R^{\Gr, \leq m} (W)
	\end{align*}
\end{construction}

\subsection{The cotangent fibre}

\begin{defi}
	We will write $\Mod$ for the $\infty$-category of pairs $(A,M)$ where $A$ is an animated $R$-algebra and $M$ is an $A$-module. A map $(A,M) \to (B,N)$ is given by a map of animated $R$-algebras $A \to B$ and a map of modules $B \otimes_A M \to N$.

	Let $\Mod_{\geq 0} \subset \Mod$ be the full subcategory spanned by the pairs $(A,M)$ where $M$ is connective. Then, we can write $\Mod_{\geq 0} := \mathcal{P}_\Sigma (\PolyMod)$, where $\PolyMod$ is the subcategory of pairs $(\LSym_R^* (V), \LSym_R^* (V) \otimes_R W)$, where $V,W \in \Vect_R^\omega$.
\end{defi}

\begin{construction}[Trivial square-zero extension]
	We construct a functor 
	\begin{align*}
		\triv: \Mod_{\geq 0} \to \CAlg_R^{\an, \surj, \leq n}
	\end{align*}
	We define it as the sifted-colimit-preserving extension of the functor $\PolyMod \to (\CAlg_R^{\an, \surj, \leq n})^\heartsuit \subset \CAlg_R^{\an, \surj, \leq n}$ sending a pair $(A,M)$ to the surjection $A \oplus M \twoheadrightarrow A$, where $A \oplus M$ has trivial multiplicative structure on $M$. For $m \leq n$, we get a commutative diagram of sifted-colimit-preserving functors
	\[
		\begin{tikzcd}
			\Mod \arrow{dr}[swap]{\triv} \arrow{r}{\triv} & \CAlg_R^{\surj, \leq m} \arrow{d}{\iota_m^n} \\
																						 & \CAlg_R^{\surj, \leq n}
		\end{tikzcd}
	\]
	The forgetful functor $\CAlg_R^{\surj, \leq n} \to \Fun (\Delta^1, \Mod_{R, \geq 0})$ commutes with sifted colimits, so by definition we have that the underlying pair of $\triv (A,M)$ is $(M, A \oplus M)$.

	This implies that $\triv$ commutes with limits, since limits are computed in the underlying module. Hence, it has a left adjoint which we will write as $L^{\leq n}: \CAlg_R^{\surj, \leq n} \to \Mod$.
\end{construction}

\begin{lemma}
	Let $V,W \in \Mod_{R, \geq 0}$. Then,
	\begin{align*}
		L^{\leq n}(\LSym^{\leq n}_R (V \to V \oplus W)) \simeq (\LSym_R^* (W), (V \oplus W) \otimes_R \LSym_R^* (W) )
	\end{align*}
\end{lemma}

\begin{proof}
	Given $(B,M) \in \Mod$, we have a natural equivalence
	\begin{align*}
		\Map _{\Mod} (L^{\leq n}(\LSym^{\leq n}_R (V \to V \oplus W)), (B,M)) & \simeq \Map_{\CAlg_R^{\an , \surj, \leq n}} (\LSym^{\leq n}_R (V \to V \oplus W), B\oplus M \to B) \\
		& \simeq \Map_{\Fun (\Delta^1, \Mod_R)} (V \to V \oplus W, M \to B \oplus M) \\
		& \simeq \Map_{\Mod_R} (V, M) \times \Map _{\Mod_R} (W, B \oplus M) \\
		& \simeq \Map_{\Mod} ((\LSym_R^* (W), (V \oplus W) \otimes_R \LSym_R^* (W)), (B,M))
	\end{align*}
	so we are done by the Yoneda embedding.
\end{proof}

\begin{rem}
	For $n = \infty$, the functor $L^{\leq n}: \CAlg_R^{\an, \surj} \to \Mod$ is the functor $(A \twoheadrightarrow B) \mapsto (B, L_{A/R} \otimes_A B)$.
\end{rem}

\begin{defi}
	The $\infty$-category of \emph{derivations} is the $\infty$-category defined by the pullback diagram
	\[
		\begin{tikzcd}
			\Der_{/R} \arrow{r} \arrow{d} & \Fun(\Delta^1, \Mod) \times _{\Fun (\Delta^1, \CAlg_R^{\an})} \CAlg_R^{\an} \arrow{d}{\ev_0} \\
			\CAlg_R^{\an} \arrow{r}{L_{-/R}[-1]} & \Mod
		\end{tikzcd}
	\]
	In other words, this is the category of triples $(B, M,\alpha)$, where $B$ is an animated $R$-algebra, $M$ is a $B$-module and $\alpha: L_{B/R}[-1] \to M$ is a map of $B$-modules.

	Given $B \in \CAlg_R^{\an}$, we will write $\Der_{B/R}$ for the $\infty$-category defined by the following pullback square.
	\[
		\begin{tikzcd}
			\Der_{B/R} \arrow{r} \arrow{d} & \Der_{/R} \arrow{d} \\
			* \arrow{r}{* \mapsto B} & \CAlg_R^{\an}
		\end{tikzcd}
	\]
	Note that we have an equivalence $\Der_{B/R} \simeq (\Mod_B)_{L_{B/R}[-1]/}$.

	The map $\Der_{/R} \to \CAlg_R^{\an}$ is a biCartesian fibration, classifying the functor $A \mapsto \Der_{A/R}$. Given a map $f: A \to B$ in $\CAlg_R^{\an}$, this induces a right adjoint $f_*: \Der_{B/R} \to \Der_{A/R}$ by restriction of scalars. That is, given a derivation $\beta : L_{B/R}[-1] \to M$, we have a  map of $A$-modules $L_{A/R}[-1] \to L_{B/R}[-1] \to M$. Its left adjoint is given on a derivation $\alpha: L_{A/R}[-1] \to M$ by the pushout of the following diagram of $B$-modules
	\begin{align*}
		L_{B/R}[-1]  \leftarrow B \otimes_A L_{A/R}[-1] \xrightarrow{\alpha} B \otimes_A M
	\end{align*}
\end{defi}

\begin{construction}[Square-zero extension]
	Let $\Der_{/R, \geq -1} \subset \Der_{/R}$ be the full subcategory spanned by pairs $(A,L_{A/R}[-1] \to M)$ where $M$ is a $(-1)$-connective $A$-module. We will abuse notation and write $\triv[1]: \Der_{/R, \geq -1} \to \CAlg_R^{\an, \surj, \leq n}$ given by $(B,\alpha: L_{B/R}[-1] \to M) \mapsto \triv (B, M[1]) \simeq (B \oplus M[1] \twoheadrightarrow B)$. This factors through the forgetful  map $\Der_{/R} \to \Mod$, whose left adjoint maps a pair $(B,M)$ to the derivation $(B, L_{B/R}[-1] \to L_{B/R}[-1] \oplus M)$. Hence, we get a left adjoint 
	\begin{align*}
		(A \twoheadrightarrow B) \mapsto (B, \alpha: L_{B/R}[-1] \to L_{B/R}[-1] \oplus L^{\leq n}(A \twoheadrightarrow B)[-1])
	\end{align*}
	The forgetful map $\const: \Der_{/R} \to \CAlg_R^{\an, \surj, \leq n}$ given by $(A, \alpha: L_{A/R}[-1] \to M) \mapsto (A \xrightarrow{\id} A)$, has a left adjoint given by $(A \twoheadrightarrow B) \mapsto (B, L_{B/R}[-1] \xrightarrow{\id} L_{B/R}[-1])$.

	Note that
	\begin{align*}
		\Map_{\CAlg_R^{\an, \surj, \leq n}} (B \to B, B \oplus M[1] \to B) \simeq \Map_{\Mod_B} (L_{B/R}, M[1]) 
	\end{align*}
	Hence, given a derivation $(B, \alpha: L_{B/R}[-1] \to M)$, there are two canonical maps $(B \to B) \to (B \oplus M[1] \to B)$, one given by $\alpha$ and the other one given by zero. This gives two natural transformations $0, \alpha: \const \to \triv[1]$. We will write $\sqz: \Der_{/R, \leq -1} \to \CAlg_R^{\an, \surj, \leq n}$ for the functor defined by the pullback diagram
	\[
		\begin{tikzcd}
			\sqz \arrow{r} \arrow{d} & \const \arrow{d}{\alpha} \\
			\const \arrow{r}{0} & \triv[1]
		\end{tikzcd}
	\]
	Since $\sqz$ is a pullback of right adjoints, it is a right adjoint. We will write $\cot_{\leq n}: \CAlg_R^{\an, \surj, \leq n} \to \Der_{/R, \geq -1}$ for its left adjoint, which can be written as a pushout of the left adjoints.

	We will sometimes abuse notation and, given $(A \twoheadrightarrow B) \in \CAlg_R^{\an, \surj, \leq n}$, we will still write $\cot_{\leq n} (A \twoheadrightarrow B)$ for the underlying module after forgetting the derivation.
\end{construction}

\begin{rem}
	For $n = \infty$, this is the functor $(A \twoheadrightarrow B) \mapsto (L_{B/R}[-1] \to L_{B/A}[-1])$.
\end{rem}

\begin{rem}
	Let $m \leq n$. Since $\sqz \simeq \iota_m^n \circ \sqz$, we can deduce an analogous statement for the left adjoints. That is,
	\begin{align*}
		\cotn \simeq \cot_{\leq m} \circ t_{\leq m} ^n
	\end{align*}
\end{rem}

\begin{lemma}
	Let $V,W \in \Vect_R$ and $n > 0$. Then,
	\begin{align*}
		\cot_{\leq n} (\LSym^{\leq n}_R (V \to V \oplus W)) \simeq (\LSym_R^* (W), W \otimes_R \LSym_R^* (W)[-1]  \xrightarrow{0} V \otimes_R \LSym_R^* (W))
	\end{align*}
\end{lemma}

\begin{proof}
	There is a pushout diagram in $\Der_{\LSym_R^* (W) /R}$
	\[
		\begin{tikzcd}
			(W \otimes_R \LSym_R^* (W)[-1] \xrightarrow{0} (V \oplus W) \otimes_R \LSym_R^* (W)[-1] ) \arrow{r} \arrow{d} & ( W \otimes_R \LSym_R^* (W)[-1] \xrightarrow{\id} W \otimes_R \LSym_R^* (W)[-1]) \arrow{d} \\
			( W \otimes_R \LSym_R^* (W)[-1] \xrightarrow{\id} W \otimes_R \LSym_R^* (W)[-1]) \arrow{r} & \cot_{\leq n} (\LSym^{\leq n}_R (V \to V \oplus W))
		\end{tikzcd}
	\]
	from which the result follows.
\end{proof}

\begin{rem}
		In particular, we see that $\cot_{\leq n }$ always lands in connective modules, since connective modules are closed under sifted colimits. That is, it lands in $\Der_{/R, \geq 0} \subset \Der_{/R, \geq -1}$, the full subcategory spanned by those derivations $(B, \alpha: L_{B/R}[-1] \to M)$ where $M$ is connective.
\end{rem}

\begin{rem}
	By checking on compact projectives, we can see that for any $(A \twoheadrightarrow B) \in \CAlg_R^{\an, \surj, \leq n}$ there is an equivalence	
	\begin{align*}
		F^{* \leq 1} \cot_{\leq n} (A \twoheadrightarrow B) & \simeq \cot_{\leq n} (\adic_{\leq n} (A \twoheadrightarrow B))
	\end{align*}
	where $F^{* \leq 1} \cot_{\leq n} (A \twoheadrightarrow B)$ is the module $\cot_{\leq n} (A \twoheadrightarrow B)$ put in filtration weight $1$. In the above, we are forgetting the map from $L_{B/R}[-1]$ and considering $\cot_{\leq n} (A \twoheadrightarrow B)$ as an object in $\Mod_{B, \geq 0}$.
\end{rem}

We will later need the following lemma.  

\begin{lemma} 
	\label{lemma-gradic}
	There is a natural equivalence of functors $\CAlg_R^{\an, \surj, \leq n} \to \Gr \CAlg_R ^{\an, \surj, \leq n}$ 
	\begin{align*}
		\Gr \circ \adic_{\leq n} (A \twoheadrightarrow B) \simeq \LSym_B^{\Gr, \leq n} (\cot^{\leq n} (A \twoheadrightarrow B) (1))
	\end{align*}
	where $\cot^{\leq n} (A \twoheadrightarrow B) (1)$ refers to the graded chain complex with $\cot^{\leq n} (A \twoheadrightarrow B)$ concentrated in degree $1$.
\end{lemma}

\begin{proof}
	Both functors commute with sifted colimits, so it easily follows by computing both sides on compact projectives.
\end{proof}

\begin{prop}
	\label{prop-cotsplit}
	Let $S \in \CAlg_R^{\an}$. Let $X^\bullet$ be an augmented cosimplicial diagram in $\CAlg_{R/ /S}^{\an, \surj, \leq n}$ that is $\cot_{\leq n}$-split. Then, $\Tot(X^\bullet)$, computed in $\Mod_R$ is in $\Mod_{R, \geq 0}$ and the natural map
	\begin{align*}
		\cot_{\leq n} (\Tot(X^\bullet)) \xrightarrow{\simeq} \Tot (\cot_{\leq n} (X^\bullet))
	\end{align*}
	is an equivalence, where the totalisation is computed in $\Mod_S$.
\end{prop}

\begin{proof}
	$\Gr\adicn (X^\bullet) \simeq \LSym_S^{\Gr, \leq n} (\cotn(X^\bullet)(1))$ is also split and the totalisation, computed in $\Mod_S$, is in $\Mod_{S, \geq 0}$. The forgetful functor $\Mod_S \to \Mod_R$ commutes with limits, so this totalisation is also computed in $\Mod_R$. This gives an equivalence
	\begin{align*}
		\cotn (\Tot (\Gr \adicn (X^\bullet))) \xrightarrow{\simeq} \Tot (\cotn (\Gr \adicn (X^\bullet)))
	\end{align*}
	where the totalisation is computed in $\Gr^{\leq n} \Mod_R$.

	$\Tot(\adicn (X^\bullet))$, computed in $\Fil^{\leq n} \Mod_R$, has graded $\Tot(\Gr \adic_{\leq n} (X^\bullet))$. Additionally, for each $i < n$ there is an exact sequence
	\begin{align*}
		F^{i+1}\Tot (\adicn (X^\bullet)) \to F^i \Tot (\adicn (X^\bullet)) \to \Gr^i \Tot(\adicn (X^\bullet))
	\end{align*}
	and $F^n \Tot(\adicn (X^\bullet)) \simeq \Gr^n \Tot (\adicn (X^\bullet))$. By induction, we see that each filtered piece, computed in $\Mod_R$, is in $\Mod_{R, \geq 0}$. Additionally, there is a natural map
	\begin{align*}
		\cotn (\Tot (\adicn (X^\bullet))) \to \Tot (\cotn (\adicn (X^\bullet))) 	
	\end{align*}
	The filtrations are finite and hence complete, and this is an equivalence after passing to gradeds, so it is an equivalence.

	We can now just take $\aug$ to conclude that $\Tot(X^\bullet)$, computed in $\Mod_R$, is in $\Mod_{R, \geq 0}$ and that the map
	\begin{align*}
		\cotn (\Tot(X^\bullet)) \xrightarrow{\simeq} \Tot(\cotn (X^\bullet))
	\end{align*}
	is an equivalence.
\end{proof}

\begin{prop}
	\label{prop-cotcomonadic}
	Let $S \in \CAlg_R^{\an}$. For $n < \infty$, the adjunction
	\begin{align*}
		\cotn: \CAlg_{R/ /S}^{\an, \surj, \leq n} \leftrightarrows (\Mod_{S, \geq 0})_{L_{S/R}[-1]/}: \sqz 
	\end{align*}
	is comonadic.
\end{prop}

\begin{proof}
	First we show that $\cotn$ is conservative. Suppose $A \to B$ is a map in $\CAlg_{R/ /S}^{\an, \surj, \leq n}$ inducing an equivalence on $\cotn$. Then, $\adicn (A\twoheadrightarrow S) \to \adicn (B \twoheadrightarrow S)$ is an equivalence on associated gradeds. By completeness, since these filtrations are finite, this is an equivalence before taking gradeds. The functor $\adicn$ is fully faithful, so $A \simeq B$.

	By \cref{prop-cotsplit}, $\cotn$-split cosimplicial diagrams commute with $\cotn$, so we are done by the Barr-Beck-Lurie Theorem (\cite{HA}, Theorem 4.4.3.5).
\end{proof}

We now compute $\cot_{\leq n}$ of some algebras. We first need some technicalities of genuine equivariant chain complexes.

\begin{defi}[Genuine equivariant chain complexes]
	Let $R$ be a discrete ring and $G$ a finite group. We will write $R[\mathcal{O}_G]$ for the $1$-category of discrete $R[G]$-modules of the form $R[S]$, where $S$ is a finite $G$-set. Note that every object in this category can be written as a coproduct of objects of the form $R[G/H]$, where $H \subset G$ is a normal subgroup.

	We define the category of \emph{genuine equivariant chain complexes} as 
	\begin{align*}
		\Mod_{\underline{R}} ^G := \Mod_{R[\mathcal{O}_G]}
	\end{align*}
	in the sense of (\cite{explicit}, Definition 2.1); that is, finite-product-preserving functors $(R[\mathcal{O}_G])^{\op} \to \Sp$.
\end{defi}

\begin{rem}
	The $t$-structure on spectra induces a $t$-structure on $\Mod_{\underline{R}}^G$. Since $R[\mathcal{O}_G]$ is additive, we can identify
	\begin{align*}
		\Mod_{\underline{R}, \geq 0} ^G \simeq \mathcal{P}_\Sigma (R[\mathcal{O}_G])
	\end{align*}
	(\cite{Gepner_2015}, Corollary 2.10).
\end{rem}

\begin{rem}
	We have deliberately avoided using Mackey functors since this will not be necessary for our purposes. More classically, one would use Mackey functors to define genuine equivariant $G$-spectra (\cite{barwick_spectra}, \cite{Guillou}), and define genuine equivariant chain complexes as the $R$-module objects in this category. However, this simplifies to the above definition (\cite{explicit}, Example 2.15).
\end{rem}

\begin{defi}[Strict fixed points and orbits]
	\label{defi-strict}
	There is a limit and colimit preserving functor $\triv_G: \Mod_R \to \Mod_{\underline{R}}^G$, given by sending a chain complex to the same underlying chain complex with trivial $G$-action. This admits a right and a left adjoint, respectively written as
	\begin{align*}
		(-)^G : \Mod_{\underline{R}}^G \to \Mod_R, \; \; \; \; \; \; \; \;  (-)_G: \Mod_{\underline{R}}^G \to \Mod_R
	\end{align*}
	which we will call \emph{strict fixed points} and \emph{strict orbits}.
\end{defi}

\begin{rem}
	On objects of the form $R[G/H]$, we have that
	\begin{align*}
		R[G/H]^G \simeq R[G/H]_G \simeq R
	\end{align*}
\end{rem}

\begin{example}
	\label{example-strictquotient}
	Let $S$ be a simplicial $G$-set. We will write $R[S]$ for the resulting genuine equivariant chain complex. That is, each simplicial level of the simplicial $R[G]$-modules $R[S]$ can be written as $R[S_i]$, where $S_i$ is a $G$-set. Then, $R[S]$, as an object in $\Mod_{\underline{R}}^G$, can be abstractly defined as the geometric realisation of the $R[S_i]$.

	Then, since strict orbits commute with colimits, we see that
	\begin{align*}
		R[S]_G \simeq R[S/G]
	\end{align*}
	where $S/G$ is the $1$-categorical simplicial quotient. We will also often write $\tilde{C} _*(S, R) := R[S]$, since the underlying chain complex is precisely the $R$-valued chains on the geometric realisation of $S$.
\end{example}

\begin{example}
	Consider the functor $\LSym^n_R : \Mod_{R, \geq 0} \to \Mod_{R, \geq 0}$. Previously, we defined it as the sifted-colimit preserving functor of $V \mapsto V^{\otimes n}_{\Sigma_n}$ for $V \in \Vect_R^\omega$. Considering $V^{\otimes n}_{\Sigma_n}$ as a genuine $\Sigma_n$-chain complex, we see that this functor can always be written as $V \mapsto V^{\otimes n}_{\Sigma_n}$.
\end{example}

\begin{defi}
	For $n > 0$ and $i \geq 0$, we will say that a chain $\hat{0} = \sigma _0 \leq \dots \leq \sigma_j = \hat{1}$ of partitions of $[n]$ has \emph{type $\leq i$} if in each of the refinements $\sigma_k \leq \sigma_{k+1}$, each part of $\sigma_{k+1}$ is refined into at most $i$ parts.
\end{defi}

\begin{prop}
	\label{prop-cotsym}
	Let $m \leq n$. For $i > 0$, define the simplicial set $T_{\leq n} (m,i)$ by
	\begin{align*}
		T_{\leq n}  (m,i)_j := \left\{ \sigma_0 \leq \dots \leq \sigma_j = \hat{1} \ \middle\vert \begin{array}{l} \sigma_k \text{ is a partition of } [i] \text{, the chain has type } \leq n \\ \text{and each part of } \sigma_0 \text{ has size } \leq m \end{array}  \right\} \bigsqcup \{*\}
	\end{align*}
	The face maps delete partitions from the chain, sending invalid chains to the basepoint. Degeneracy maps insert repeated partitions.

	Then, for $V \in \Vect_R$,
	\begin{align*}
		\cot_{\leq n} (\LSym_R^{\leq m} (V \to V)) \simeq \bigoplus_{i > 0} \left(V^{\otimes i} \otimes \tilde{C}_* (T_{\leq n} (m,i), R)\right) _{\Sigma_i}
	\end{align*}
	where $(-)_{\Sigma_i}$ refers to strict orbits.
\end{prop}

\begin{rem}
	Notice that for $n = \infty$ and $m = 1$, we get $\cot (\sqz (V))$ and we recover (\cite{partition}, Proposition 5.35).
\end{rem}

\begin{proof}
	Given $(A \twoheadrightarrow R) \in \CAlg_R^{\an, \surj, \leq n}$, we can write it as its bar resolution in $\CAlg_{R/ /R}^{\an, \surj, \leq n}$ as a $\LSym_R^{\nu, \leq n}$-algebra (\cref{prop-nonunital}). This means that we can write $(A \twoheadrightarrow R)$ as a geometric realisation of objects of the form $\LSym^{\leq n}_R (V \to V) = (\LSym_R^{\leq n} (V) \xrightarrow{V \mapsto 0} R)$. Notice that this is not the same as its bar resolution in $\CAlg_R^{\an, \surj, \leq n}$ as a $\LSym_R^{\leq n}$-algebra.

	$\LSym_R^{\nu, \leq m} (V)$ is a $\LSym_R^{\nu, \leq n}$-algebra by trivialising anything in degree higher than $m$. Remember that $\cot_{\leq n} (\LSym_R^{\leq n} (V \to V)) \simeq (0 \to V)$. Hence, omitting the zero (since $L_{R/R} \simeq 0$), we get
	\begin{align*}
		\cot_{\leq n} (\LSym^{\nu, \leq m}_R (V)) & \simeq | \Bar_\bullet (\id, \LSym^{\nu, \leq n}_R, \LSym^{\nu, \leq m}_R (V)) | \\
													 & \simeq \bigg|  \ \ \ldots \ \ \substack{\longrightarrow \vspace{-4pt} \\ \leftarrow   \vspace{-4pt} \ \\ \longrightarrow \vspace{-4pt} \\ \leftarrow   \vspace{-4pt} \ \\ \longrightarrow  \vspace{-4pt} \\ \leftarrow  \vspace{-4pt} \ \\ \longrightarrow} \ \ \bigoplus_{c \leq n} (\bigoplus_{b \leq n}   ( \bigoplus_{a\leq m}  V^{\otimes a}_{\Sigma_a} )^{\otimes b}_{\Sigma_b})^{\otimes c}_{ \Sigma_c}\ \ \substack{\longrightarrow \vspace{-4pt} \\ \leftarrow   \vspace{-4pt} \ \\ \longrightarrow  \vspace{-4pt} \\ \leftarrow  \vspace{-4pt} \ \\ \longrightarrow} \ \ \bigoplus_{b \leq n} (\bigoplus_{a \leq m}  V^{\otimes a}_{\Sigma_a})^{\otimes b}_{ \Sigma_b} \ \ \substack{  \longrightarrow  \vspace{-4pt} \\ \leftarrow  \vspace{-4pt} \ \\ \longrightarrow} \ \ \bigoplus_{a \leq m} V^{\otimes a}_{ \Sigma_{a}} \ \ \bigg|\
	\end{align*}
	We can now apply a well-known combinatorial observation (c.f. \cite{Ching}, \S 7) and expand these exponents binomially. We see that an exponent adding up to $i$ at the $j$th simplicial level of this diagram is equivalent to a chain $\sigma_0 \leq \dots \leq \sigma_j = \hat{1}$ of partitions of $[i]$. Additionally, each part of $\sigma_0$ has size $\leq n$, since the most refined part of this partition corresponds to the exponents coming from $\bigoplus_{a \leq m} V^{\otimes a}_{\Sigma_a}$. Additionally, due to the bounds on the other sums, this chain has type $\leq n$. Hence,
	\begin{align*}
		\cot_{\leq n} (\LSym_R^{\nu, \leq m} (V)) & \simeq  \bigg|  \ \ \ldots \ \ \substack{\longrightarrow \vspace{-4pt} \\ \leftarrow   \vspace{-4pt} \ \\ \longrightarrow  \vspace{-4pt} \\ \leftarrow  \vspace{-4pt} \ \\ \longrightarrow} \ \ \bigoplus_{i > 0}  (V^{\otimes i} \otimes R[T_{\leq n} (m,i)_1] )_{\Sigma_i} \ \ \substack{  \longrightarrow  \vspace{-4pt} \\ \leftarrow  \vspace{-4pt} \ \\ \longrightarrow} \ \ \bigoplus_{i > 0} (V^{\otimes i} \otimes  R[T_{\leq n} (m,i)_0])_{ \Sigma_i} \ \ \bigg| \\
													 & \simeq \bigoplus_{i > 0} \left(V^{\otimes i} \otimes \tilde{C}(T_{\leq n} (m,i),R)\right)_{ \Sigma_i}
	\end{align*}
	where $R[T_{\leq n} (m,i)_j]$ refers to the free $R$-module generated by the set $T_{\leq n} (m,i)_j$ quotiented by the basepoint.
\end{proof}

We will now see that the above sum is finite in each homological degree.

\begin{lemma}
	\label{lemma-cotconnectivity}
	Let $l,m,n > 0$ be fixed. Then, $(C^* (T_{\leq n} (m,i), R) \otimes_R V^{\otimes i})_{\Sigma_i}$ is $l$-connective for $i \gg 0$ for any $V \in \Mod_{R, \geq 0}$. The bounds on $i$ do not depend on $V$.
\end{lemma}

\begin{proof}
	We will write $X := C^* (T_{\leq n} (m,i), R)$ to simplify notation. Let $F: \Mod_{S, \geq 0} \to \Mod_{S, \geq 0}$ be the functor $V \mapsto (X \otimes_R V^{\otimes i})_{\Sigma_i}$. We may only prove this for $V$ finite and free and then left Kan extend. Then, using the binomial formula, we see that if $V = R^j$ then
	\begin{align*}
		F(V) \simeq \bigoplus_{i_1 + \dots + i_j = i}  \left( \Ind _{\Sigma_{i_1} \times \dots \times \Sigma_{i_j}}^{\Sigma_i} X \right)_{\Sigma_i} \simeq \bigoplus_{i_1 + \dots + i_j = i} X_{\Sigma_{i_1} \times \dots \times \Sigma_{i_j}}
	\end{align*}
	Thus, it suffices to show that $T_{\leq n} (m,i)/_{\Sigma_{i_1} \times \dots \times \Sigma_{i_j}}$ is $l$-connected for $i \gg 0$.

	A \emph{partition of $[i]$ with $i_1, \dots, i_j$ repetitions} is a partition of a set with $i_1$ undistinguishable elements, $i_2$ undistinguishable elements, and so on. We can identify the set of partitions of $[i]$ with $i_1 , \dots , i_j$ repetitions with the $(\Sigma_{i_1} \times \dots \times \Sigma_{i_j})$-orbits of the set of partitions of $[i]$. Notice that we can identify $T_{\leq n} (m,i)/_{\Sigma_{i_1} \times \dots \times \Sigma_{i_j}}$ with the simplicial set whose $r$-simplices are chains $\sigma_0 \leq \dots \leq \sigma_r = \hat{1}$ of partitions of $i$ with $i_1, \dots , i_j$ repetitions (plus a basepoint), where $\sigma_0$ has all parts of size $\leq m$, and the chain is of type $\leq n$ (\cref{example-strictquotient}).

	If $n < \infty$, the result is now clear because for $i$ large enough we cannot have any such chains of length $r < l$. Hence, assume $n = \infty$.

	There is a map $T_{\leq \infty} (\infty,i)/_{\Sigma_{i_1} \times \dots \times \Sigma_{i_j}} \to T_{\leq \infty} (m,i)/_{\Sigma_{i_1} \times \dots \times \Sigma_{i_j}}$ given by collapsing invalid chains to the basepoint. Let $\tilde{T}$ be the set-theoretic fibre of this map, which is a cofibration since it is injective at every simplicial level. Hence, $\tilde{T}$ is the simplicial set whose $r$-simplices are chains $\sigma_0 \leq \dots \leq \sigma_r = \hat{1}$ of partitions of $[i]$ with $i_1 , \dots, i_j$ repetitions (plus a basepoint), where $\sigma_0$ has at least one part of size $> m$. 

	In any chain of $\tilde{T}$, each partition is a coarsening of $\sigma_0$, so all of the partitions have a part with size $> m$. Hence, by \cref{prop-unreduced}, we see that $\tilde{T}$ is the unreduced suspension of the poset of proper partitions of $[i]$ with $i_1, \dots, i_j$ repetitions and one part of size $> m$.

	This poset is atomised, and the atoms are partitions with exactly one part of size $m+1$ and every other part a singleton. Let $\sigma_1, \dots, \sigma_k$ be $k$ such atoms, with parts $P_1, \dots, P_k$ of size $m+1$. Their join is given by the partition with one part $P_1 \cup \dots \cup P_k$ and every other part a singleton. Hence, we see that the least number of atoms that generate this poset is $\lceil \frac{i}{m+1} \rceil$. By \cref{prop-atomised}, we see that this poset is at least $(\lceil \frac{i}{m+1} \rceil -2)$-connected.

	Hence, since $\tilde{T}$ is the unreduced suspension of this poset, it is $(\lceil \frac{i}{m+1} \rceil -1)$-connected.

	By a similar argument, $T_{\leq \infty} (\infty, i)/_{\Sigma_{i_1} \times \dots \times \Sigma_{i_j}}$ is the unreduced suspension of the poset of partitions of $[i]$ with $i_1 , \dots, i_j$ repetitions that are not the indiscrete partition (i.e. not a single part). But this poset has an infimum, so it is contractible. 

	Additionally, $T_{\leq \infty} (m,i)/_{\Sigma_{i_1} \times \dots \times \Sigma_{i_j}}$ is the set-theoretic quotient of the map $\tilde{T} \to T_{\leq \infty} (\infty, i)/_{\Sigma_{i_1} \times \dots \times \Sigma_{i_j}}$, and thus the homotopy cofibre since the map is a cofibration. Hence, $T_{\leq \infty} (m,i)/_{\Sigma_{i_1} \times \dots \times \Sigma_{i_j}}$ is the suspension of $\tilde{T}$ so it is $\lceil \frac{i}{m+1} \rceil$-connected. The result follows.
\end{proof}

\section{Bounded and nilpotent algebras}
\label{section-boundednilpotent}

We will now study maps of animated rings $(A \twoheadrightarrow B)$ that are \emph{bounded and nilpotent}. We will see that we can understand these maps through the categories of nilpotent surjections.

\begin{defi}
	We say that $(A \twoheadrightarrow B) \in \CAlg_R^{\an, \surj}$ is \emph{nilpotent} if $\mathfrak{m} := \ker (\pi_0 (A) \to \pi_0 (B))$ is nilpotent; that is, $\mathfrak{m}^n \cong 0$ for $n \gg 0$. We will write $\CAlg_R^{\an, \Nil} \subset \CAlg_R^{\an, \surj}$ for the full subcategory spanned by the nilpotent surjections. For $S \in \CAlg_R^{\an}$, we will write $\CAlg_{R/ /S}^{\an, \Nil} \subset \CAlg_{R/ /S}^{\an, \surj}$ for the category of nilpotent arrows with target $S$.

	We say that $(A \twoheadrightarrow B)$ is \emph{bounded} if $\pi_i \fib (A \to B) \cong 0$ for $i\gg 0$.

	We will write $\CAlg_R^{\an, \bNil}$ and $\CAlg_{R/ /S}^{\an, \bNil}$ for the corresponding categories of bounded and nilpotent arrows.
\end{defi}

\begin{construction}[Approximating $\cot$]
	\label{construction-approxcot}
	Let $(A \twoheadrightarrow B) \in \CAlg_R^{\an, \surj, \leq n}$ and let $m \geq n$. Then, we can consider $(A \twoheadrightarrow B)$ as an object in $\CAlg_R^{\an, \surj, \leq m}$ by the right adjoint to polynomial truncation. Hence, there is a counit map $t_{\leq n}^m (A \twoheadrightarrow B) \to (A \twoheadrightarrow B)$. Applying $\cot_{\leq n}$, we get a map
	\begin{align*}
		\cot_{\leq m} (A \twoheadrightarrow B) \to \cot_{\leq n} (A \twoheadrightarrow B)
	\end{align*}
	The polynomial truncation maps are compatible with each other so we get a tower
	\begin{align*}
		\cot (A \twoheadrightarrow B) \to \dots \to \cot_{\leq n+1} (A \twoheadrightarrow B) \to \cot_{\leq n} (A \twoheadrightarrow B)
	\end{align*}
	and hence a natural map
	\begin{align*}
		\cot (A \twoheadrightarrow B) \to \varprojlim_{m \geq n} \cot_{\leq m} (A \twoheadrightarrow B)
	\end{align*}
\end{construction}

\begin{construction}[Filtrations on $\cot_{\leq n}$]
	\label{construction-cotfiltrations}
	Let $m \leq n$. For $(A \twoheadrightarrow B) \in \CAlg_R^{\an, \surj, \leq m}$, we can consider it as an object of $\Fil\CAlg_R^{\an, \surj, \leq m}$ through the adic filtration. Applying $\cot_{\leq n} \circ \iota_m^n$, we get a natural filtration on $\cot_{\leq n} (A \twoheadrightarrow B)$. All of the functors we used commute with sifted colimits, so this construction commutes with sifted colimits.

	Notice that the map $\cot_{\leq n+1} (A \twoheadrightarrow B) \to \cot_{\leq n} (A \twoheadrightarrow B)$ then refines to a map of filtered modules.
\end{construction}

\begin{lemma}
	\label{lemma-cotcomplete}
	Let $m \leq n$ and let $(A \twoheadrightarrow B) \in \CAlg_R^{\an, \surj, \leq m}$. Then the filtration on $\cot_{\leq n} (A \twoheadrightarrow B)$ is complete.
\end{lemma}

\begin{proof}
	Complete objects are closed under geometric realisations (\cite{partition}, Proposition 2.11), so we may just show this for $(A \twoheadrightarrow B) = \LSym_R^{\leq m} (V \to V \oplus W)$ for $V,W \in \Vect_R$. By \cref{prop-cotsym}, we have an equivalence
	\begin{align*}
		\cot_{\leq n} (\LSym_R^{\leq m} (V \to V \oplus W)) \simeq \bigoplus_{i > 0} (V^{\otimes i} \otimes \tilde{C}_* (T_{\leq n} (m,i) , R))_{\Sigma_i} \otimes \LSym_R^* (W)
	\end{align*}
	Following the filtered weights in the proof of \cref{prop-cotsym}, we see that
	\begin{align*}
		F^j \cot_{\leq n} (\LSym_R^{\leq m} (V \to V \oplus W)) \simeq \bigoplus_{i \geq j} (V^{\otimes i} \otimes \tilde{C}_* (T_{\leq n} (m,i) , R))_{\Sigma_i} \otimes \LSym_R^* (W)
	\end{align*}
	By \cref{lemma-cotconnectivity}, for fixed $l$, the $j$th filtered piece is $l$-connective for $j \gg 0$. This directly implies completeness by the Milnor sequence.
\end{proof}

\begin{lemma}
	\label{lemma-cotpieces}
	Let $m > 0$.
	\begin{enumerate}
		\item Let $l \geq 0$. Then, there exists a $j$ such that for any $(A \twoheadrightarrow B) \in \CAlg_R^{\an, \surj, \leq m}$ the composition
		\begin{align*}
			\cot (A \twoheadrightarrow B) \to \cot_{\leq n} (A \twoheadrightarrow B) \to \cot_{\leq n} (A \twoheadrightarrow B)/F^j \cot_{\leq n} (A \twoheadrightarrow B)
		\end{align*}
		is $l$-connected for $n \gg m$. The bounds on $n$ do not depend on $(A \twoheadrightarrow B)$.
	\item Let $i$ be fixed. For any $(A \twoheadrightarrow B) \in \CAlg_R^{\an, \surj, \leq m}$ the map
		\begin{align*}
			\Gr^i \cot (A \twoheadrightarrow B) \to \Gr^i \cot_{\leq n} (A \twoheadrightarrow 	B)
		\end{align*}
		is an equivalence for $n \gg m$. The bounds on $n$ do not depend on $(A \twoheadrightarrow B)$.
	\end{enumerate}
\end{lemma}

\begin{proof}
	We may just check this on compact projectives. Choose $j$ so that for all $j' > j$ and all $V \in \Vect_R$, $(V^{\otimes j'} \otimes \tilde{C}_* (T_{\leq \infty} (m,i), R))_{\Sigma_{j'}}$ is $l$-connective (\cref{lemma-cotconnectivity}). Thus, for $V,W \in \Vect_R$ and $(A \twoheadrightarrow B) = \LSym_R^{\leq m} (V \to V \oplus W)$, the above map becomes 
	\begin{align*}
		\bigoplus_{i > 0} (V^{\otimes i} \otimes \tilde{C}_* (T_{\leq \infty} (m,i), R))_{\Sigma_i} \otimes \LSym_R^* (W) \to \bigoplus_{0 \leq i \leq j} (V^{\otimes i} \otimes \tilde{C} (T_{\leq n} (m,i), R))_{\Sigma_i} \otimes \LSym_R^* (W)
	\end{align*}
	Every term on the left summand with $i > j$ is $l$-connective by construction, so the result follows from the fact that $T_{\leq \infty} (m,i) \simeq T_{\leq n} (m,i)$ for $n \gg m$.

	For the second part, we again reduce it to the case of compact projectives. This ends up reducing to the fact that $T_{\infty} (m,i) \simeq T_{\leq n} (m,i)$ for $n \gg m$.
\end{proof}

\begin{prop}
	\label{prop-cotcolim}
	Let $(A \twoheadrightarrow B) \in \CAlg_R^{\surj, \leq m}$. Then, the natural map
	\begin{align*}
	\cot (A \twoheadrightarrow B) \to  \varprojlim_{n \geq m} \cot_{\leq n}  (A \twoheadrightarrow B)
	\end{align*}
	is an equivalence, where this limit is computed in $\Mod_B$.
\end{prop}

\begin{proof}
	The filtration on $\cot_{\leq n} (A \twoheadrightarrow B)$ is complete for all $n$ (\cref{lemma-cotcomplete}). An inverse limit of complete filtered modules is complete, so we just need to check that this is an equivalence on gradeds. This follows from \cref{lemma-cotpieces}.
\end{proof}

\begin{defi}
	Let $S \in \CAlg_R^{\an}$ and let $f: A \to B$ be a morphism in $\CAlg_{R/ /S}^{\an, \surj}$. It is \emph{elementary} if it fits in a pullback diagram
	\[
		\begin{tikzcd}
			A \arrow{r} \arrow{d} & B \arrow{d} \\
			S \arrow{r} & S \oplus M[n]
		\end{tikzcd}
	\]
	where $M \in \Mod_S^\heartsuit$ and $n > 0$.

	$f: A \to B$ is \emph{small} if it is a composition of elementary morphisms.

	An object $A \in \CAlg_{R/ /S}^{\an}$ is \emph{small} if the map $A \to S$ is small.
\end{defi}

We have the following infinite-dimensional version of (\cite{DAGX}, Lemma 1.1.20).

\begin{prop}
	\label{prop-smallsurjections}
	An object $A \in \CAlg_{R/ /S}^{\an, \surj}$ is small if and only if it is bounded and nilpotent.

	Additionally, a morphism $f: A \to B$ in $\CAlg_{R/ /S}^{\an, \bNil}$ is small if and only if it is a surjection in $\pi_0$.
\end{prop}

\begin{proof}
	Suppose $A \in \CAlg_{R/ /S}^{\an, \surj}$ is small. Then, it is clearly bounded and nilpotent. Hence, it suffices to prove that a morphism in $\CAlg_{R/ /S}^{\an, \bNil}$ is small if and only if it is a surjection in $\pi_0$.

	Any elementary morphism is surjective on $\pi_0$, so the same is true of any small morphism.

	Let $f: A \to B$ be surjective on $\pi_0$, and let $C := \cofib (A \to B)$. Let $n$ be the smallest non-negative integer such that $\pi_n (C) \neq 0$. By assumption, $n > 0$. There is an $(n+1)$-connective map $C \otimes_A B \to L_{B/A}$ (\cite{HA}, Theorem 7.4.3.1 and \cite{SAG}, Proposition 25.3.5.1), and hence an $(n+1)$-connective map $C \otimes_A S \to L_{B/A} \otimes_B S$. We have the fibre sequence
	\begin{align*}
		L_{A/R} \otimes_A B \to L_{B/R} \to L_{B/A}
	\end{align*}
	We can tensor by $S$ as $B$-modules to get a fibre sequence
	\begin{align*}
		L_{A/R} \otimes_A S \to L_{B/R} \otimes_B S \to L_{B/A} \otimes_B S
	\end{align*}
	Let $D := \pi_n (L_{B/A}\otimes_B S) \cong \pi_n (C \otimes_A S)$. Notice that $D \in \Mod_S^\heartsuit$. We have the trivial map $A \to S \oplus D[n] $. Then,
	\begin{align*}
		\Map_{A/ /S} (B, S \oplus D[n]) & \simeq \Map_{\Mod_S} (L_{B/A} \otimes_B S, D[n]) \\
											  & \simeq \Map_{\Mod_S} ( \tau _{\leq n} L_{B/A} \otimes_B S, D[n]) \\
											  & \simeq \Map_{\Mod_S} ( \tau _{\leq n} C \otimes_A S, D[n]) \\
											  & \simeq \Map _{\Mod_S} (C \otimes_A S, D[n])
	\end{align*}
	The truncation map gives us a canonical choice of a map $C \otimes_A S \to D[n]$, which induces a map $B \to S \oplus D[n]$. Let $B' := B \times_{S \oplus D[n]} S$, so we have a diagram
	\[
		\begin{tikzcd}
			A \arrow{dr}{f'} \arrow{ddr} \arrow{drr}{f} & & \\
																 & B' \arrow{r} \arrow{d} & B \arrow{d} \\
																 & S \arrow{r} & S \oplus D[n]
		\end{tikzcd}
	\]
	We have a map of cofibre sequences
	\[
		\begin{tikzcd}
			A \arrow{r} \arrow{d} & A \arrow{r} \arrow{d} & 0 \arrow{d} \\
			B' \arrow{r} & B \arrow{r} & D[n]
		\end{tikzcd}
	\]
	Hence, $C' := \cofib (A \to B') \simeq \fib (C \to D[n])$. The map $A \to S$ induces a map $C \to C \otimes_A S$, and the map $C \to D[n]$ is induced by the composition of this with the truncation $C \otimes_A S \to  \tau _{\leq n} C \otimes_A S \simeq D[n]$.

	Let $\mathfrak{m} := \ker (\pi_0 (A) \to \pi_0 (S))$. Then, $\pi_n (C') \cong \mathfrak{m} D$. We can repeat this inductively to obtain a map $A \to B'^{(m)}$, where $C'^{(n)} := \cofib (A \to B'^{(n)}) \in \Mod_{R, \geq n}$, $\pi_n (C'^{(n)}) \cong \mathfrak{m}^m D$ and $\pi_l (C'^{(n)}) \cong \pi_l (C)$ for all other $l$. Since $A$ is nilpotent, this process must finish, and we must eventually get an $(n+1)$-connective map $A \to B'^{(m)}$. Then we can also induct on the homotopy groups of $C$ to obtain $f$ as a composition of elementary morphisms.
\end{proof}

We will now construct a refinement of Postnikov towers for nilpotent algebras. It is well-known that each step in the Postnikov tower can be written as a square-zero extension. That is, if $A \in \CAlg^{\an}$, there is a pullback diagram
\[
	\begin{tikzcd}
		\tau_{\leq n+1} A \arrow{r} \arrow{d} & \tau_{\leq n} A \arrow{d} \\
		\tau_{\leq n} A \arrow{r} & \tau_{\leq n} A \oplus \pi_{n+1} (A) [n+2]
	\end{tikzcd}
\]
We will refine this into a tower where each step is a square-zero extension by $S$-modules.

\begin{construction}[Refined Postnikov towers]
	\label{construction-refinedPostnikov}
	Let $A \in \CAlg_{R/ /S}^{\Nil}$. Consider the map $\tau_{\leq n+1} A \to \tau_{\leq n} A$. By the proof of \cref{prop-smallsurjections}, we can refine this to a tower of the form
	\begin{align*}
		\tau_{\leq n+1} A \to \dots \to \tau_{\leq n+1}^{i} A \to \dots \to \tau_{\leq n+1} ^{1} A \to \tau_{\leq n+1}^0 A \simeq \tau_{\leq n} A
	\end{align*}
	where
	\begin{align*}
		\pi_j (\tau_{\leq n+1} ^i A ) \cong \begin{cases} \pi_j (A) \; & \text{if $j \leq n$}	 \\ \mathfrak{m}^{i-1}\pi_{n+1} (A)/\mathfrak{m}^i \; & \text{if $j = n+1$} \\ 0 \; & \text{otherwise}	\end{cases}
	\end{align*}
	where $\mathfrak{m}$ is the kernel ideal of $\pi_0 (A) \to \pi_0 (S)$. Since $\mathfrak{m}^ i = 0$ for $i \gg 0$, $\tau_{\leq n+1}^i A \simeq \tau_{\leq n+1} A$ for $i \gg 0$. Each of the steps in this tower comes from a pullback diagram of the form
	\[
		\begin{tikzcd}
			\tau_{\leq n+1} ^{i+1} A \arrow{r} \arrow{d} & \tau_{\leq n+1}^i A \arrow{d} \\
			\tau_{\leq n+1}^i  A \arrow{r} & \tau_{\leq n+1}^i A \oplus \mathfrak{m}^i \pi_{n+1} (A)/\mathfrak{m}^{i+1} [n+2]
		\end{tikzcd}
	\]
	The module $\mathfrak{m}^i \pi_{n+1} (A) /\mathfrak{m}^{i+1}$ is a $\pi_0 (A)/\mathfrak{m} \cong \pi_0 (S)$-module.

	Additionally, the construction in \cref{prop-smallsurjections} makes this tower functorial, meaning that a map $A \to B$ in $\CAlg_{R/ /S}^{\Nil}$ induces a map between these refined towers, as well as between the square-zero extension diagrams.
\end{construction}

\begin{prop}
	\label{prop-Nilimage}
	Let $S \in \CAlg_R^{\an}$ and let $A \in \CAlg_{R/ /S}^{\an, \bNil}$. Then, $A$ is in the image of $\iota_n^\infty: \CAlg_{R/ /S}^{\an, \surj, \leq n} \to \CAlg_{R/ /S}^{\an, \surj}$ for some $n$.

	Additionally, if $A \to B$ is a map in $\CAlg_{R/ /S}^{\an, \bNil}$, it is also in the image of $\iota_n^\infty$ for some $n$.
\end{prop}

\begin{proof}
	By \cref{prop-smallsurjections}, any object $A \in \CAlg_{R/ /S}^{\an, \bNil}$ can be written as a finite pullback of trivial square-zero extensions. Also, for $V \in \Mod_{S, \geq 0}$, $S \oplus V$ is in the image of $\iota_1^\infty$ as a free algebra. Hence, since $\iota_n^\infty$ commutes with limits, for the first part it suffices to show that if $A \in \CAlg_{R/ /S}^{\surj, \leq n}$ and $V \in \Mod_{S, \geq 0}$ is eventually coconnective, the natural map
	\begin{align*}
		\varinjlim_{m \geq n} \Map_{\CAlg_{R/ /S}^{\surj, \leq m}} (A, S \oplus V) \to \Map_{\CAlg_{R/ /S}^{\an}} (A, S \oplus V) 
	\end{align*}
	is an equivalence. Using the adjunction, it suffices to show that
	\begin{align*}
		\varinjlim _{m \geq n} \Map_{\Mod_S} (\cot_{\leq m} (A \twoheadrightarrow S), V) \to  \Map_{\Mod_S} (\cot (A \twoheadrightarrow S), V)
	\end{align*} 
	is an equivalence. We will show that this is an equivalence on homotopy groups. For $\pi_0$, we need to show that any map $\cot (A \twoheadrightarrow S) \to V$ factors through $\cot_{\leq m} (A \twoheadrightarrow S)$ for some $m$. Since $V$ is truncated, we may just show that any map $\tau_{\leq i} \cot (A \twoheadrightarrow S) \to V$ factors through $\tau_{\leq i} \cot_{\leq m} (A \twoheadrightarrow S)$ for some $i$. This follows from \cref{lemma-cotpieces}, since the map $\tau_{\leq i} \cot (A \twoheadrightarrow S) \to \tau_{\leq i} \cot_{\leq m} (A \twoheadrightarrow S)$ admits a section for $m \gg n$.

	Now take $l > 0$. Then, $\pi_l \Map_{\Mod_S} (\cot (A \twoheadrightarrow S), V)  \cong \pi_0 \Map_{\Mod_S} (\cot (A \twoheadrightarrow S), \tau_{\geq 0} V[-l])$ and similar for $\cot_{\leq m} (A \twoheadrightarrow S)$. $\tau_{\geq 0} V[-l]$ is bounded so the claim follows just as in the case $l = 0$.

	For the last part, let $A,B \in \CAlg_{R/ /S}^{\an, \bNil}$. We must show that any map $A \to B$ is in the image of $\iota_n^\infty$ for some $n$. We just proved that $A$ is in the image of $\iota_n^\infty$ for some $n$. We can again build $B$ by a sequence of square-zero extensions. If $B = S \oplus V$, where $V \in \Mod_{S, \geq 0}$ is eventually coconnective, we just proved the result, since
	\begin{align*}
		\Map_{\CAlg_{R/ /S}^{\an}} (A, S \oplus V) \simeq \Map _{\Mod_S} (\cot (A \twoheadrightarrow S), V) \simeq \varinjlim _{m \geq n} \Map_{\Mod_S} (\cot_{\leq m} (A \twoheadrightarrow S), V)
	\end{align*}
	so any map $A \to S \oplus V$ comes from $\CAlg_{R/ /S}^{\an, \surj, \leq m}$ for some $m$.

	We will now induct on the square-zero extension diagrams. Write $B' = B \times_{S \oplus V} S$, where $V \in \Mod_{S, \geq 0}$ is eventually coconnective, $B$ is in the image of $\iota_n ^\infty$ and $\Map_{\CAlg_{R/ /S}^{\an}}(A, B') \simeq \varinjlim_{m \geq n} \Map_{\CAlg_{R/ /S}^{\surj, \leq m}} (A, B')$. By increasing $n$ if necessary, we may assume that the map $B \to S \oplus V$ is in the image of $\iota_n^\infty$. Since the functors $\iota_n^m$ commute with limits, $B'$ is in the image of $\iota_n^\infty$. Then,
	\begin{align*}
		\Map_{\CAlg_{R/ /S}^{\an}} (A, B') & \simeq \fib \left(\Map_{\CAlg_{R/ /S}^{\an}} (A, B) \to  \Map_{\CAlg_{R/ /S}^{\an}} (A,  S \oplus V)\right) \\
		& \simeq \fib \left( \varinjlim _{m \geq n} \Map_{\CAlg_{R/ /S}^{\an, \surj, \leq m}} (A, B) \to \varinjlim _{m \geq n} \Map _{\CAlg_{R/ /S}^{\an, \surj, \leq m}} (A,  S \oplus V) \right)  \\
		& \simeq \varinjlim_{m \geq n} \fib \left( \Map_{\CAlg_{R/ /S}^{\an, \surj, \leq m}} (A, B) \to \Map_{\CAlg_{R/ /S}^{\an, \surj, \leq m}} (A, S \oplus V) \right) \\
		& \simeq \varinjlim_{m \geq n} \Map_{\CAlg_{R/ /S}^{\an, \surj, \leq m}} (A, B')
	\end{align*}
\end{proof}

We will now find a rich source of animated rings whose adic filtration is complete.

\begin{lemma}
	\label{lemma-symncomplete}
	Let $V,W \in \Vect_R$. Then, there is an equivalence
	\begin{align*}
		F^j \adic (\LSym_R^{\leq n} (V \to V \oplus W)) \simeq \bigoplus_{i > 0} (V^{\otimes i} \otimes \tilde{C}_* (S^{\geq j} (n,i), R))_{\Sigma_i} \otimes_R \LSym_R^* (W)
	\end{align*}
	Where $S^{\geq j} (n,i)$ is the simplicial set whose $l$-simplices are given by
	\begin{align*}
		S^{\geq j} (n,i)_l := \left\{ \sigma_0 \leq \dots \leq \sigma_i \ \middle\vert \begin{array}{l} \sigma_r \text{ is a partition of } [i] \text{ with $\geq j$ parts}  \\ \text{and each part of } \sigma_0 \text{ has size } \leq n \end{array}  \right\} \bigsqcup \{*\}
	\end{align*}
	Additionally, if $n < \infty$, this filtration is complete.
\end{lemma}

\begin{proof}
	Note that $F^j \adic(\LSym_R^* (V \to V \oplus W)) \simeq \bigoplus_{i \geq j} \LSym_R^{\geq j} (V) \otimes_R \LSym_R^* (W)$, where $\LSym_R^{\geq j} (V) = \bigoplus_{i \geq j} V^{\otimes i}_{\Sigma_i}$. We can use the bar construction as in \cref{prop-cotsym}, to obtain the result.

	Now we show completeness. First we will show that for fixed $j$, $l$ and $n$, $(V^{\otimes i} \otimes \tilde{C}_* (S^{\geq j} (n,i)))_{\Sigma_i}$ is $l$-connected for $i \gg 0$ for any $V \in \Mod_{R, \geq 0}$, and the bounds on $i$ do not depend on $V$. We proceed exactly as in \cref{lemma-cotconnectivity}.

	Firstly, we just need to show that $S^{\geq j} (n,i) /_{\Sigma_{i_1} \times \dots \times \Sigma_{i_r}}$ is $l$-connected for $i \gg 0$. Let $\tilde{S}$ be the set-theoretic fibre of the map $S^{\geq j} (\infty, i)/_{\Sigma_{i_1} \times \dots \times \Sigma_{i_r}} \to S^{\geq j} (n,i)/_{\Sigma_{i_1} \times \dots \times \Sigma_{i_r}}$ that collapses invalid chains. 

	The map $\tilde{S} \to S^{\geq j} (\infty, i) /_{\Sigma_{i_1} \times \dots \times \Sigma_{i_r}}$ is a cofibration, so we get a cofibre sequence $\tilde{S} \to S^{\geq j} (\infty, \i) /_{\Sigma_{i_1} \times \dots \times \Sigma_{i_r}} \to S^{\geq j} (n,i)/_{\Sigma_{i_1} \times \dots \times \Sigma_{i_r}}$.

	We can identify $\tilde{S}$ with the poset of partitions of $[i]$ with $i_1, \dots, i_r$ repetitions with at least one part of size $> n$ and $\geq j$ parts (\cref{example-strictquotient}). This poset is atomised and the atoms are the partitions with exactly one part of size $(n+1)$ and the rest singletons.

	Let $\sigma_1, \dots , \sigma_t$ be $t$ such atoms, with parts $P_1, \dots, P_t$ of size $n+1$. Their join is given by the partition with one part $P_1 \cup \dots \cup P_t$ and every other part a singleton, provided that this has at least $j$ parts. Hence, the least number of atoms that generate this poset is $\lceil \frac{i-j}{n+1} \rceil$. By \cref{prop-atomised}, this simplicial set is $\lceil \left( \frac{i-j}{n+1} \rceil -2\right)$-connected.

	The simplicial set $S^{\geq j} (\infty, i)/_{\Sigma_{i_1} \times \dots \times \Sigma_{i_r}}$ is the poset of partitions of $[i]$ with $i_1, \dots, i_r$ repetitions with at least $j$ parts. But this has an infimum so it is contractible. Hence, $S^{\geq j} (n,i)$ is the suspension of $\tilde{S}$ and it is $\left( \lceil \frac{i-j}{n+1} \rceil \right)$-connected, and this quantity gets arbitrarily large.

	In particular, this shows that	
	\begin{align*}
		F^j \adic (\LSym_R^{\leq n} (V \to V \oplus W)) \simeq \prod_{i > 0} (V^{\otimes i} \otimes \tilde{C}_* (S^{\geq j} (n,i), R))_{\Sigma_i} \otimes_R \LSym_R^* (W)
	\end{align*}
	This allows us to write $\adic (\LSym_R^{\leq n} (V \to V \oplus W))$ as a product of filtered modules $A_i$ where
	\begin{align*}
		F^j A_i \simeq (V^{\otimes i} \otimes \tilde{C}_* (S^{\geq j} (n,i), R))_{\Sigma_i} \otimes_R \LSym_R^* (W)
	\end{align*}
	We may just show that this is complete, but this is clear because there are no partitions of $i$ with at least $j$ parts for $j > i$, so $S^{\geq j} (n,i) \simeq *$ for $j > i$.
\end{proof}

\begin{theorem}
	Let $(A \twoheadrightarrow B)$ be a map in $\CAlg_R^{\an}$ that is $\pi_0$-surjective such that the kernel ideal of $\pi_0 (A)$ is nilpotent. Then, $\adic (A \twoheadrightarrow B)$ is complete.
\end{theorem}

\begin{proof}
	Note that if $(A \twoheadrightarrow B) \to (C \twoheadrightarrow D)$ is $n$-connected, so is $\adic (A \twoheadrightarrow B) \to \adic (C \twoheadrightarrow D)$. By the Milnor sequence, it is enough then to show that $\adic (\tau_{\leq i} A \twoheadrightarrow \tau_{\leq i} B)$ is complete for all $i$, so we may assume that both $A$ and $B$ are eventually coconnective.

	The map between truncations is nilpotent and bounded, so it is in the image of $\iota_n^\infty : \CAlg_R^{\an, \surj, \leq n} \to \CAlg_R^{\an, \surj}$ for some $n$ (\cref{prop-Nilimage}). The functor $\iota_n^\infty$ commutes with sifted colimits, so using the bar resolution in $\CAlg_R^{\an, \surj, \leq n}$, we can write $(A \twoheadrightarrow B)$ as a geometric realisations of objects of the form $\LSym_R^{\leq n} (V \to V \oplus W)$, where $V,W \in \Vect_R$.

	We are now done because complete objects are closed under geometric realisations and $\adic (\LSym_R^{\leq n} (V \to V \oplus W))$ is complete (\cref{lemma-symncomplete}).
\end{proof}

\section{Dually truncated partition Lie algebroids}
\label{section-duallytruncatedLie}

In this section, we will introduce \emph{dually truncated partition Lie algebroids}. They will be Lie algebraic models naturally arising from the category of nilpotent surjections. For the rest of this section we fix a field $k$ that is an $R$-algebra.

We first prove the following refinement of \cref{prop-cotcomonadic}.

\begin{lemma}
	Let $n < \infty$. Then, the adjunction
	\begin{align*}
		\cotn^\vee : \CAlg_{R/ /k}^{\an, \surj, \leq n} \leftrightarrows (\Solid_{k, \leq 0}^{\Pro})_{/L_{k/R}^\vee[1]}^{\op}: \sqz (-^\vee)
	\end{align*}
	is comonadic.
\end{lemma}

\begin{proof}
	This follows exactly the same proof as \cref{prop-cotcomonadic}, but we can't just dualise that result because duality might not take totalisations to geometric realisations. 

	Since duality gives an equivalence $\Mod_{k, \geq 0} \simeq (\Solid_{k, \leq 0}^{\Pro})^{\op}$, we see that $X^\bullet$ is $\cotn^\vee$-split. Since $\cotn^\vee (\Gr \adicn (X^\bullet)) \simeq \cotn^\vee (X^\bullet)(1)$, we get an equivalence
	\begin{align*}
		| \cotn^\vee (\Gr	\adicn (X^\bullet)) | \xrightarrow{\simeq} \cotn^\vee (\Tot (\Gr \adicn (X^\bullet)))
	\end{align*}
	Then, the map
	\begin{align*}
		| \cotn ^\vee (\adicn (X^\bullet)) | \xrightarrow{\simeq} \cotn^\vee (\Tot(\adicn (X^\bullet)))
	\end{align*}
	is an equivalence on gradeds, so it is an equivalence. The result follows from applying $\aug$.
\end{proof}

We also need the following dual of \cref{prop-cotcolim}.

\begin{lemma}
	\label{lemma-truecotcolim}
	If $A \in \CAlg_{R/ /k}^{\an, \surj, \leq m}$, there is a natural equivalence
	\begin{align*}
		\varinjlim_{n \geq m} \cot_{\leq n}^\vee (A) \simeq \cot^\vee (A)
	\end{align*}
\end{lemma}

\begin{proof}
	There is a natural filtration on $\cot_{\leq m}^\vee (A)$ that is complete (\cref{lemma-cotcomplete}). Dualising, we get an increasing filtration that	is cocomplete, meaning that the filtered colimit of the filtered pieces vanishes. Thus, it suffices to check an equivalence on gradeds. This is true by dualising \cref{lemma-cotpieces}.
\end{proof}

This gives us a monad on $(\Solid_{k, \leq 0}^{\Pro})_{/L_{k/R}^\vee[1]}$. The goal will be to show that we can extend this to a sifted-colimit-preserving monad on $(\Solid_k)_{/L_{k/R}^\vee [1]}$.

\begin{defi}
	We will say that a tower $X_0 \to \dots \to  X_i \to \dots$ in $\Solid_{k, \leq 0}^{\Pro}$ is \emph{almost eventually constant} if for each $j$, the map $\tau_{\geq j} X_n\to \tau_{\geq j} X_{n+1}$ is an equivalence for $n \gg 0$.

	Let $F: \Solid_{k, \leq 0}^{\Pro} \to \catC$ be a functor where $\catC$ is an $\infty$-category with sifted colimits. $F$ is \emph{regular} if it preserves almost eventually constant towers.

	We will say that $F$ preserves \emph{finite stable geometric realisations} if it preserves colimits of diagrams $\Delta^{\op} \to \Solid_{k, \leq 0}^{\Pro}$ that are left Kan extended from $(\Delta^{\leq n})^{\op}$ for some $n$ such that the geometric realisation, when computed in $\Solid_k$, is in $\Solid_{k, \leq 0}^{\Pro}$.
\end{defi}

We first need the following technical lemma, following (\cite{explicit}).

\begin{lemma}
	\label{lemma-completion}
	Let $\catC$ be an $\infty$-category that admits sifted colimits. Then, restriction gives an equivalence
	\begin{align*}
		\Fun _{\Sigma} (\Solid_k, \catC) \xrightarrow{\simeq} \Fun_{\sigma}' (\Solid_{k, \leq 0}^{\Pro}, \catC)
	\end{align*}
	where $\Fun_{\Sigma} ( \Solid_k, \catC) \subset \Fun (\Solid_k, \catC) $ refers to the full subcategory of sifted-colimit-preserving functors, and $\Fun_{\sigma}' (\Solid_{k, \leq 0}^{\Pro}, \catC) \subset \Fun (\Solid_{k, \leq 0}^{\Pro}, \catC)$ refers to the full subcategory of regular functors that preserve finite stable geometric realisations.
\end{lemma}

\begin{proof}
	Let $\kappa$ be an uncountable strong limit cardinal. Notice that since $\Solid_k^\kappa$ is closed under colimits, it suffices to show the above result for $\Solid_k^\kappa$. By (\cite{explicit}, Proposition 2.39), it suffices to show that $\Solid_k^\kappa$ is equivalent to its category of pro-coherent modules and that its dually almost perfect modules are equivalent to $\Solid_{k, \leq 0}^{\Pro, \kappa}$.

	In the language of (\cite{explicit}, \S 2), $\Solid_k^\kappa$ can be identified with the category of modules over the small category $\ProVect_{< \kappa}$. The opposite of this category is just the category of vector spaces of dimension $< \kappa$. Notice that every object is projective, so it is clear that almost perfect objects in the stable $\infty$-category $\Mod_{\ProVect_{< \kappa}^{\op}}$ are equivalent to the category of bounded below chain complexes of vector spaces of dimension $< \kappa$. Hence, the coherent ones are precisely the bounded chain complexes of dimension $< \kappa$, which are dually equivalent to the bounded chain complexes in $\ProVect_{< \kappa}$. By (\cite{ultrasolid}, Proposition 2.25), we see that pro-coherent modules in $\ProVect_{< \kappa}$ are just $\Solid_k^\kappa$.

	From this one easily deduces that the dually almost perfect objects are precisely $\Solid_{k, \leq 0}^{\Pro} \cap \Solid_k^\kappa$. 
\end{proof}

\begin{prop}
	The monad $\cotn^\vee \circ \sqz (-^\vee)$ on $(\Solid_{k, \leq 0}^{\Pro})_{/L_{k/R}^\vee [1]}$ admits a unique extension to a sifted-colimit-preserving monad on $(\Solid_k)_{/L_{k/R}^\vee [1]}$.
\end{prop}

\begin{proof}
	This follows by exactly the same proof as (\cite{foliations}, Proposition 3.8). By (\cite{explicit}, Proposition 4.20), we just need to show that $\cotn^\vee \circ \sqz (-^\vee)$ preserves the following kind of colimits in $(\Solid_{k, \leq 0}^{\Pro})_{/L_{k/R}^\vee [1]}$:
	\begin{enumerate}
		\item Towers  $M_1 \to M_2 \to \dots$  such that for all $i$, the tower $\pi_i (M_1) \to \pi_i (M_2) \to \dots$ is eventually constant.
		\item Simplicial diagrams $X^\bullet$ in $(\Solid_{k, \leq 0}^{\Pro})_{/L_{k/R}^\vee [1]}$ that are left Kan extended from $(\Delta^{\leq m})^{\op}$ for some $m$, such that $|X^\bullet|$, when computed in $(\Solid_k)_{/L_{k/R}^\vee [1]}$, is in $(\Solid_{k, \leq 0}^{\Pro})_{/L_{k/R}^\vee [1]}$.
	\end{enumerate}
	The first part is clear because if $A \to B$ is a map in $\CAlg_{R/ /k}^{\an, \surj}$ that is $i$-connective, then $\cot (A) \to \cot(B)$ is $i$-connective.

	For the second one, we will use the filtrations on $\cot$ from \cref{construction-cotfiltrations}. That is, $\sqz ((X^\vee)^\bullet)$ factors through the category $\CAlg_{R / /k}^{\an, \surj, \leq 1}$, and this gives us a complete filtration on $\cot \circ \sqz ((X^\vee)^\bullet)$. Dualising, we get an increasing filtration that is \emph{cocomplete}; that is, the colimit of the filtered pieces vanishes. Equivalences are detected on gradeds, and filtered colimits of cocomplete increasingly filtered objects are closed under geometric realisations, so it suffices to show the equivalence after taking gradeds.

	Given a derivation $\alpha: L_{k/R}[-1] \to V$, the filtration on $\cot \circ \sqz (V)$ is obtained by putting $V$ in filtered degree $1$ and $L_{k/R}[-1]$ in filtered degree $0$. The associated graded is the trivial derivation, so we may assume that all of the derivations are trivial.

	In this case, we use \cref{prop-cotsym}, since $\sqz (L_{k/R}[-1] \xrightarrow{0} V) \simeq \LSym_k^{\leq 1} (V \to V)$. We just need to show that the functor $F: \Solid_{k, \leq 0}^{\Pro} \to \Solid_{k, \leq 0}^{\Pro}$ given by
	\begin{align*}
		V \mapsto (V^{\otimes n} \otimes \tilde{C}^* (T_{\leq \infty} (1, n), k))^{\Sigma_n}
	\end{align*}
	preserves the relevant colimits. Let $X^\bullet$ be a simplicial diagram in $\Solid_{k, \leq 0}^{\Pro}$ that is left Kan extended from $(\Delta^{\leq m})^{\op}$ such that the geometric realisation, when computed in $\Solid_k$, is in $\Solid_{k, \leq 0}^{\Pro}$. Then, tensoring with itself $n$ times gives an $n$-simplicial diagram, to which the diagonal is cofinal. We will write $(X^\bullet)^{\otimes n}$ for the diagonal simplicial diagram, which is left Kan extended from $(\Delta^{\leq mn})^{\op}$ (\cite{polynomialfunctors}, Proposition 2.10). Since we are in a stable category, finite colimits commute with limits, and $(V^{\otimes n} \otimes \tilde{C}^* (T_{\leq \infty} (1,n), k))^{\Sigma_n}$ is left Kan extended from $(\Delta^{\leq nm})^{\op}$. Then, using that finite colimits commute with limits in $\Solid_k$,
	\begin{align*}
		|((X^\bullet)^{\otimes n} \otimes \tilde{C}^* (T_{\leq \infty} (1, n), k))^{\Sigma_n}|  & \simeq \colim_{(\Delta^{\leq mn})^{\op}} ((X^\bullet)^{\otimes n} \otimes \tilde{C} ^* (T_{\leq \infty} (1,n), k))^{\Sigma_n} \\
		& \simeq  ((\colim_{(\Delta^{\leq mn})^{\op}} X^\bullet)^{\otimes n} \otimes \tilde{C} ^* (T_{\leq \infty} (1,n), k))^{\Sigma_n} \\
		& \simeq (|X^\bullet|^{\otimes n} \otimes \tilde{C}^* (T_{\leq \infty} (1,n), k))^{\Sigma_n}
	\end{align*}
	as required.
\end{proof}

\begin{defi}
	Let $n > 0$. The \emph{dually $n$-truncated partition Lie algebroid monad}, written as $\Lie^{\pi, \weirdleq n}_{\Delta, k/R}$, is the unique sifted-colimit-preserving extension of the monad $\cotn^\vee \circ \sqz (-^\vee)$ on $(\Solid_{k, \leq 0}^{\Pro})_{/L_{k/R}^\vee [1]}$.

	We will write
	\begin{align*}
		\LieAlgd^{\pi, \weirdleq n, \blacksquare}_{\Delta, k/R} := \Alg_{\Lie^{\pi, \weirdleq n}_{\Delta, k/R}} ((\Solid_k)_{/L_{k/R}^\vee[1]})
	\end{align*}
	for the $\infty$-category of algebras over this monad.

	For $n = \infty$, we will write $\Lie_{\Delta, k/R}^{\pi} := \Lie_{\Delta, k/R}^{\pi, \weirdleq \infty}$ for the \emph{partition Lie algebroid monad} and $\LieAlgd^{\pi, \blacksquare}_{\Delta, k/R} := \LieAlgd_{\Delta, k/R}^{\pi, \weirdleq \infty, \blacksquare}$ for its category of algebras.
\end{defi}

\begin{rem}
	For $m \leq n$, we have a factorisation of the $\cotn^\vee$ and $\sqz(-^\vee)$ adjunction
	\[
		\begin{tikzcd}
			\CAlg_R^{\an, \surj, \leq n} \arrow{r}{t_{\leq m}^n} \arrow{dr}{\cotn^\vee} & \CAlg_R^{\an, \surj, \leq m} \arrow{d}{\cot_{\leq m}^\vee} \\
																												 & \left( (\Solid_{k, \leq 0}^{\Pro})_{/L_{k/R}^\vee[1]}\right) ^{\op}
		\end{tikzcd}
	\]
	The counit map on the comonad $t_{\leq m}^n \circ \iota_m^n \to \id$ induces a map of monads $\Lie_{\Delta, k/R}^{\pi, \weirdleq m} \to \Lie_{\Delta, k/R}^{\pi, \weirdleq n}$ (\cite{explicit}, Lemma 3.10). Hence, we get a tower of monads
	\begin{align*}
		\Lie_{\Delta, k/R} ^{\pi, \weirdleq n} \to \Lie_{\Delta, k/R}^{\pi, \weirdleq n+1}\to \dots  \to \Lie_{\Delta, k/R}^{\pi} 
	\end{align*}
	Thus, for $m < n$, we have a right adjoint $\LieAlgd_{\Delta, k/R}^{\pi, \weirdleq n, \blacksquare} \to \LieAlgd_{\Delta, k/R}^{\pi, \weirdleq m, \blacksquare}$ that is the identity on the underlying chain complex. Informally, any $\Lie_{\Delta, k/R}^{\pi, \weirdleq n}$-algebra has the canonical structure of a $\Lie_{\Delta, k/R}^{\pi, \weirdleq m}$-algebra. We will often abuse notation, and for $\mathfrak{g} \in \LieAlgd_{\Delta, k/R}^{\pi, \weirdleq n, \blacksquare}$, we will still write $\mathfrak{g}$ for its image in $\LieAlgd_{\Delta, k/R}^{\pi, \weirdleq m, \blacksquare}$.
\end{rem}

\begin{construction}[Koszul duality functors]
	The adjunction gives us functors
	\begin{align*}
		\KD^{\leq n}: \CAlg_{R/ /k}^{\an, \surj, \leq n} \to (\LieAlgd_{\Delta, k/R}^{\pi, \weirdleq n, \blacksquare})^{\op}
	\end{align*}
	This admits a right adjoint
	\begin{align*}
		C^*_{\leq n} : \left( \LieAlgd_{\Delta, k/R}^{\pi, \weirdleq n, \blacksquare} \right) ^{\op} \to \CAlg_{R/ /k}^{\an, \surj, \leq n}
	\end{align*}
	By \cref{lemma-squarecommutes}, for $m < n$ the left adjoints fit in a commutative diagram
	\[
		\begin{tikzcd}
			\CAlg_{R/ /k}^{\an, \surj, \leq n} \arrow{d}{ t _{\leq m}^{n}} \arrow{r}{\KD^{\leq n }} & (\LieAlgd_{\Delta, k/R}^{\pi, \weirdleq n, \blacksquare})^{\op} \arrow{d} \\
			\CAlg_{R/ /k}^{\an, \surj, \leq m} \arrow{r}{\KD^{\leq m}} & (\LieAlgd_{\Delta, k/R}^{\pi, \weirdleq m, \blacksquare})^{\op}
		\end{tikzcd}
	\]
	We will write $\KD := \KD^{\leq \infty}$ for the Koszul duality functor $\KD: \CAlg_{R/ /k}^{\an, \surj} \to (\LieAlgd_{\Delta, k/R}^{\pi})^{\op}$.
\end{construction}

\begin{lemma}
	\label{lemma-squarecommutes}
	Let $m \leq n$. Then, the diagram
	\[
		\begin{tikzcd}
			\CAlg_{R/ /k}^{\an, \surj, \leq n} \arrow{d}{ t _{\leq m}^{n}} \arrow{r}{\KD^{\leq n }} & (\LieAlgd_{\Delta, k/R}^{\pi, \weirdleq n, \blacksquare})^{\op} \arrow{d} \\
			\CAlg_{R/ /k}^{\an, \surj, \leq m} \arrow{r}{\KD^{\leq m}} & (\LieAlgd_{\Delta, k/R}^{\pi, \weirdleq m, \blacksquare})^{\op}
		\end{tikzcd}
	\]
	is commutative.
\end{lemma}

\begin{proof}
	We will prove it abstractly. Let $F: \catC \leftrightarrows \catD : G$ be an adjunction between $\infty$-categories with an initial or final object. By (\cite{explicit}, Lemma 3.10), the functor $\End(\catD) \to \End(\catC)$ given by $T \mapsto GTF$ is lax monoidal. Additionally, we can identify $\catC$ as a full subcategory of $\End(\catC)$ given by the constant functor (here we are using the existence of an initial or final object, so that these constant functors are given by left/right Kan extension from the initial/final object). We will abuse notation and write $\LMod (\catC)$ for the category of modules in $\catC$ over $\End (\catC)$. That is, objects in $\LMod(\catC)$ consist of pairs $(T,M)$ where $T$ is a monad in $\catC$ and $M$ is a $T$-algebra. A map $(T,M) \to (T',M')$ is given by a map of monads $T \to T'$ and a map of $T$-algebras $M \to M'$. Define $\LMod(\catD)$ similarly.

	This fits in a commutative diagram
	\[
		\begin{tikzcd}
			\LMod(\catD) \arrow{r} \arrow{d} & \LMod(\catC) \arrow{d} \\
			\End(\catD) \arrow{r} & \End(\catC)
		\end{tikzcd}
	\]
	where the arrow $\LMod(\catD) \to \End(\catD)$ is the biCartesian fibration classifying $T \mapsto \Alg_T (\catD)$, and similar for $\catC$. Informally, for an object $(T,M)$ in $\LMod(\catD)$, the object $GM$ has the structure of a $GTF$-algebra. Clearly, the map $\LMod(\catD) \to \LMod(\catC)$ preserves Cartesian edges: the Cartesian edges are those $(T,M) \to (T',M')$ such that the map $M \to M'$ of $T$-algebras is an equivalence, and these are preserved due to the map $\LMod(\catD) \to \LMod(\catC)$ being compatible with the forgetful functors to $\catD$ and $\catC$.

	Consider the above situation for $\catC = (\CAlg_{R/ /k}^{\an, \surj, \leq m})^{\op}$, $\catD = (\Solid_{k, \leq 0}^{\Pro})_{/L_{k/R}^\vee [1]}$, $F = \sqz(-)^\vee$ and $G = \cot^\vee_{\leq n}$. Then, the functor $ t_{\leq m} ^n \circ \iota_m^n$ is a comonad and has a counit map $\iota_m^n \circ t_{\leq m}^n \to \id$. Since in the above diagram Cartesian edges are preserved, we get a commutative diagram
	\[
		\begin{tikzcd}
			\coAlg_{ t_{\leq m} ^n \circ \iota_m^n } (\CAlg_{R/ /k}^{\an, \surj, \leq m}) \arrow{r} \arrow{d} & (\LieAlgd_{\Delta, k/R}^{\pi, \weirdleq n, \blacksquare} )^{\op}\arrow{d} \\
			\CAlg_{R/ /k}^{\an, \surj, \leq m}  \arrow{r}{\KD^{\leq m}} &  (\LieAlgd_{\Delta, k/R}^{\pi, \weirdleq m, \blacksquare})^{\op}
		\end{tikzcd}
	\]
	Notice that there is a factorisation
	\[
		\begin{tikzcd}
			\CAlg_{R/ /k}^{\an, \surj, \leq n} \arrow{r} \arrow{dr}{t^n_{\leq m}} & \coAlg_{ t_{\leq m}^n \circ \iota_m^n} (\CAlg_{R/ /k}^{\an, \surj, \leq m}) \arrow{d} \\
			& \CAlg_{R/ /k}^{\an, \surj, \leq m}
		\end{tikzcd}
	\]
	(\cite{HA}, Proposition 4.7.3.3). So it suffices to show that the functor $\KD^{\leq n} : \CAlg_{R/ /k}^{\an, \surj, \leq n} \to (\LieAlgd_{\Delta, k/R}^{\pi, \weirdleq n, \blacksquare})^{\op}$ factorises through $\left( \coAlg_{ t_{\leq m}^n \circ \iota_m^n} (\CAlg_{R/ /k}^{\an, \surj, \leq n}) \right) ^{\op}$. We will prove this abstractly.

	Let $\catC, \catD$ and $\catE$ be $\infty$-categories with a final or initial object so that we can consider them as full subcategories of their endomorphism categories as constant functors. Consider adjunctions $F_1: \catC \leftrightarrows \catD : G_1$ and $F_2: \catD \leftrightarrows \catE: G_2$, and let $F = F_2 \circ F_1$, $G = G_1 \circ G_2$. As before, we have lax monoidal functors $\End(\catE) \to \End(\catD) \to \End(\catC)$.

	This induces natural functors $\LMod(\End(\catE)) \to \LMod(\End(\catD)) \to \LMod(\End(\catC))$. The natural map $\catE \to \Alg_{G \circ F} (\catC)$ can be described by the restriction of this composition to the full subcategory of $\LMod(\End(\catE))$ spanned by pairs $(\id, X)$, where $X$ is a constant functor. This trivially gives a factorisation 
	\[
		\begin{tikzcd}
			\catE \arrow{r} \arrow{dr} \arrow{r} & \Alg_{G_2 \circ F_2} (\catD) \arrow{d} \\
															 & \Alg_{G \circ F} (\catC)
		\end{tikzcd}
	\]
	The result now follows from letting $\catC = (\Solid_{k, \leq 0}^{\Pro})_{/L_{k/R}^\vee [1]}$, $\catD = \left( \CAlg_{R/ /k}^{\an, \surj, \leq m} \right) ^{\op}$, $\catE = \left( \CAlg_{R/ /k}^{\an, \surj, \leq n} \right)^{\op}$, $(F_1, G_1) = (\sqz(-^\vee), \cot_{\leq m}^\vee)$ and $(F_2, G_2) = (\iota_m^n, t_{\leq m}^n)$.
\end{proof}

\begin{rem}
	\label{rem-KDcolim}
	The previous commutative diagram refines \cref{construction-approxcot} to a map of $\Lie_{\Delta, k/R}^{\pi, \weirdleq m}$-algebras. That is, if $A \in \CAlg_{R/ /k}^{\an, \surj, \leq m}$ and $m \leq  n$, applying $\KD^{\leq m}$ to the counit map $ t _{\leq m}^n  A \to A$, gives a natural map of $\Lie_{\Delta, k/R}^{\pi, \weirdleq m}$-algebras
	\begin{align*}
		\KD ^{\leq m} (A) \to \KD^{\leq n} (A)	
	\end{align*}
	where we are considering $\KD^{\leq n}(A)$ as a $\Lie_{\Delta, k/R}^{\pi, \weirdleq m}$-algebra.

	By \cref{prop-cotcolim}, these maps assemble into an equivalence of chain complexes
	\begin{align*}
		\varinjlim _{m \leq n} \KD^{\leq n} (A) \xrightarrow{\simeq} \KD (A)
	\end{align*}
\end{rem}

The following is an immediate consequence of \cref{prop-cotcolim}.

\begin{prop}
	\label{prop-Lievalues}
	Let $V \in \Solid_{k, \leq 0}^{\Pro}$. Then,
	\begin{align*}
		\Lie_{\Delta, k/R}^{\pi, \weirdleq n} (  V \xrightarrow{0} L_{k/R}^\vee [1] ) \simeq \left( \left(  \bigoplus_{i \geq 0}  \left( \tilde{C}_* (T_{\leq n}(1,i), k)  \otimes (V^\vee) ^{\otimes i}\right)  _{\Sigma_i} \right)^\vee \xrightarrow{0} L_{k/R}^\vee [1] \right) 
	\end{align*}
	where the subscript $\Sigma_i$ refers to strict orbits (\cref{defi-strict}).
\end{prop}

\begin{prop}
	\label{prop-Liecolim}
	There is an equivalence
	\begin{align*}
		\Lie_{\Delta, k/R}^{\pi} \simeq \varinjlim_n \Lie_{\Delta, k/R}^{\pi, \weirdleq n}
	\end{align*}
	where the colimit is taken in the category of monads on $(\Solid_k)_{/L_{k/R}^\vee [1]}$.
\end{prop}

\begin{proof}
	Firstly, notice that since all the monads involved commute with filtered colimits, the colimit is computed pointwise. All of the monads commute with sifted colimits, so it is enough to check on $(\Solid_{k, \leq 0}^{\Pro})_{/L_{k/R}^\vee [1]}$. 

	For $V \in (\Solid_{k, \leq 0}^{\Pro})_{/L_{k/R}^\vee [1]}$, $\sqz (V^\vee)$ is in the image of $\iota_1^\infty$. The result follows by \cref{lemma-truecotcolim}.
\end{proof}

The following is an immediate consequence of \cref{prop-Liecolim} and (\cite{monads}, Corollary 6.9).

\begin{theorem}
	\label{thm-Liecolim}
	There is an equivalence of $\infty$-categories
	\begin{align*}
		\LieAlgd_{\Delta, k/R}^{\pi, \blacksquare} \simeq \varprojlim _n \LieAlgd_{\Delta, k/R}^{\pi, \weirdleq n, \blacksquare}
	\end{align*}
	where the transition maps $\LieAlgd_{\Delta, k/R}^{\pi, \weirdleq n+1, \blacksquare} \to \LieAlgd_{\Delta, k/R}^{\pi, \weirdleq n, \blacksquare}$ are the right adjoints induced by the maps of monads $\Lie_{\Delta, k/R}^{\pi, \weirdleq n} \to \Lie_{\Delta, k/R}^{\pi, \weirdleq n+1}$.
\end{theorem}

We can now obtain an algebraic description of coconnective partition Lie algebroids with profinite homology.

\begin{cor}
	There is an equivalence of $\infty$-categories
	\begin{align*}
		\LieAlgd_{\Delta, k/R}^{\pi, \blacksquare} \left( (\Solid_{k, \leq 0}^{\Pro})_{/L_{k/R}^\vee [1]} \right)  \simeq \left( \varprojlim \CAlg_{R/ /k}^{\an, \surj, \leq n} \right)^{\op}
	\end{align*}
	where the transition maps $\CAlg_{R/ /k}^{\an, \surj, \leq n+1} \to \CAlg_{R/ /k}^{\an, \surj, \leq n}$ are polynomial truncation.
\end{cor}

\begin{proof}
	For $n < \infty$, the adjunction
	\begin{align*}
		\cot^\vee : \CAlg_{R/ /k}^{\an, \surj, \leq n} \leftrightarrows (\Solid_{k, \leq 0}^{\Pro})^{\op}_{/L_{k/R}^\vee [1]} : \sqz (-^\vee)
	\end{align*}
	is comonadic (\cref{prop-cotcomonadic}). Additionally, there is a commutative square
	\[
		\begin{tikzcd}
			\CAlg_{R/ /k}^{\an, \surj, \leq n+1} \arrow{d}{ t _{\leq n}^{n+1}} \arrow{r}{\KD^{\leq n+1 }} & \LieAlgd_{\Delta, k/R}^{\pi, \weirdleq n+1, \blacksquare} \left( (\Solid_{k, \leq 0}^{\Pro})_{/L_{k/R}^\vee [1]} \right)   \arrow{d} \\
				\CAlg_{R/ /k}^{\an, \surj, \leq n} \arrow{r}{\KD^{\leq n}} & \LieAlgd_{\Delta, k/R}^{\pi, \weirdleq n, \blacksquare} \left( (\Solid_{k, \leq 0}^{\Pro})_{/L_{k/R}^\vee [1]} \right) 
		\end{tikzcd}
	\]
	(\cref{lemma-squarecommutes}). The horizontal arrows are equivalences, and combining this with \cref{thm-Liecolim} we get the desired result.
\end{proof}

We will now set up a Koszul duality adjunction. Due to set-theoretic issues, we will have to do this one cardinal at a time and show that this is consistent.

\begin{construction}[$\sigma$-small Koszul duality adjunction]
	\label{construction-smallKD}
	Let $\sigma$ be an uncountable strong limit cardinal. For $\sigma$ large enough, the monad $\Lie_{\Delta, k/R}^{\pi, \weirdleq n}$ preserves the full subcategory $(\Solid_k^\sigma)_{/L_{k/R}^\vee [1]}  \subset (\Solid_k)_{/L_{k/R}^\vee [1]}$. We will write $\LieAlgd_{\Delta, k/R}^{\pi, \weirdleq n, \blacksquare, \sigma} \subset \LieAlgd_{\Delta, k/R}^{\pi, \weirdleq n, \blacksquare}$ for the full subcategory of dually truncated partition Lie algebroids whose underlying chain complex is $\sigma$-ultrasolid. This subcategory is closed under all colimits. The inclusion admits a right adjoint $\LieAlgd_{\Delta, k/R}^{\pi, \weirdleq n, \blacksquare} \to \LieAlgd_{\Delta, k/R}^{\pi, \weirdleq n, \blacksquare, \sigma}$. We will write
	\begin{align*}
		\KD^{\leq n}_\sigma : \CAlg_{R/ /k}^{\an, \surj, \leq n} \to \left( \LieAlgd_{\Delta, k/R}^{\pi, \weirdleq n, \blacksquare, \sigma} \right) ^{\op}
	\end{align*}
	for the composition of $\KD^{\leq n}$ with this right (left in the opposite category) adjoint.

	This admits a right adjoint
	\begin{align*}
		C^*_{\leq n, \sigma} : \left( \LieAlgd_{\Delta, k/R}^{\pi, \weirdleq n, \blacksquare, \sigma} \right)^{\op} \to \CAlg_{R/ /k}^{\an, \surj, \leq n}
	\end{align*}
\end{construction}

We will now show a few results about the independence of these functors with respect to cardinality.

\begin{defi}
	Let $\sigma$ be an uncountable strong limit cardinal. We will say that $V \in \Mod_k$ is \emph{$\sigma$-small} if $\pi_i (V)$ has dimension $< \sigma$ for all $i$.

	We will write $\left( \CAlg_{R/ /k}^{\an, \surj, \leq n} \right)^\sigma \subset \CAlg_{R/ /k}^{\an, \surj, \leq n}$ for the full subcategory consisting of $A \in \CAlg_{R/ /k}^{\an, \surj, \leq n}$ such that $\cot(A)$ is $\sigma$-small.
\end{defi}

\begin{lemma}
	\label{lemma-smallKD}
	Let $\sigma$ and $\sigma'$ be uncountable strong limit cardinals with $ \sigma < \sigma'$. Then,
	\begin{align*}
		\KD^{\leq n}_\sigma |_{ \left( \CAlg_{R/ /k}^{\an, \surj, \leq n} \right) ^\sigma} \simeq \KD^{\leq n}_{\sigma'} | _{\left( \CAlg_{R/ /k}^{\an, \surj, \leq n} \right)^\sigma}
	\end{align*}
	Additionally,
	\begin{align*}
		C^* _{\leq n, \sigma'} |_{\LieAlgd_{\Delta, k/R}^{\pi, \weirdleq n, \blacksquare, \sigma} } \simeq C^* _{\leq n , \sigma}
	\end{align*}
\end{lemma}

\begin{proof}
	If $A \in \left( \CAlg_{R/ /k}^{\an, \surj, \leq n} \right) ^\sigma$, the dual of $\cot(A)$, computed in $\Solid_k$, is in $\Solid_k^\sigma \subset \Solid_k$. Hence, $\KD^{\leq n} (A) \in \LieAlgd_{\Delta, k/R}^{\pi, \weirdleq n, \blacksquare, \sigma} \subset \LieAlgd_{\Delta, k/R}^{\pi, \weirdleq n}$.

	The result now follows from the fact that the left adjoint
	\begin{align*}
		\LieAlgd_{\Delta, k/R}^{\pi, \weirdleq n, \blacksquare, \sigma} \to \LieAlgd_{\Delta, k/R}^{\pi, \weirdleq n, \blacksquare, \sigma'}
	\end{align*}
	is fully faithful, so the unit is an equivalence.
\end{proof}

In order to prove an ultrasolid variant of the Lurie-Pridham theorem, we need some technical results regarding pushouts of Lie algebras.

We follow the strategy of (\cite{partition}, Theorem 4.20) to prove the following.

\begin{prop}
	\label{prop-truncatedpushout}
	Let $n < \infty$ and let $A \to C$ and $B \to C$ be maps in $\CAlg_{R/ /k}^{\an, \surj, \leq n}$ that are surjective on $\pi_0$. Then, the natural map of $\Lie_{\Delta, k/R}^{\pi, \weirdleq n}$-algebras
	\begin{align*}
		\KD^{\leq n} (A) \sqcup_{\KD^{\leq n} (C)} \KD^{\leq n} (B) \xrightarrow{\simeq} \KD^{\leq n} (A \times_C B)
	\end{align*}
	is an equivalence.
\end{prop}

\begin{proof}
	Firstly, note that if $U,V,W \in (\Mod_{k, \geq 0})_{L_{k/R}[-1]/}$ and we have $\pi_0$-surjective maps $U \to W$, $V \to W$ in $(\Mod_{k, \geq 0})_{L_{k/R}[-1]/}$, the natural map
	\begin{align*}
		\KD^{\leq n} (\sqz (U)) \sqcup_{\KD^{\leq n} (\sqz(W))} \KD^{\leq n} (\sqz(V)) \to \KD^{\leq n} (\sqz(U) \times_{\sqz (W)} \sqz (V)) \simeq \KD^{\leq n} (\sqz(U \times_W V))
	\end{align*}
	is an equivalence. Indeed, $\KD^{\leq n} (\sqz(U)), \KD^{\leq n} (\sqz(V))$ and $\KD^{\leq n} (\sqz(W))$ are the free $\Lie_{\Delta, k/R}^{\pi, \weirdleq n}$-algebras on $U^\vee, V^\vee$ and $W^\vee$ respectively. Similarly, the right hand side is the free $\Lie_{\Delta, k/R}^{\pi, \weirdleq n}$-algebra on $U^\vee \sqcup_{W^\vee} V^\vee$. Hence, this equivalence amounts to the fact that the free algebra functor commutes with colimits.

	We will use the comonadicity from \cref{prop-cotcomonadic}. Let $A^\bullet$, $B^\bullet$ and $C^\bullet$ be the bar resolutions of $A, B$ and $C$ respectively, so that $A^i := (\sqz \circ \cot_{\leq n})^{\circ (i+1)} (A)$ and similar for $B$ and $C$. Let $X^\bullet := A^\bullet \times_{C^\bullet} B^\bullet$, so that $A \times_C B \simeq \Tot (X^\bullet)$.

	By our discussion of the free case, we have that the natural map
	\begin{align*}
		\KD^{\leq n} (A^i) \sqcup_{\KD^{\leq n} (C^i)} \KD^{\leq n} (B^i) \xrightarrow{\simeq} \KD^{\leq n} (X^i)
	\end{align*}
	is an equivalence.

	Then, $\KD^{\leq n} (A) \sqcup_{\KD^{\leq n} (C)} \KD^{\leq n} (B) \simeq | \KD^{\leq n} (A^\bullet) \sqcup_{\KD^{\leq n} (C^\bullet)} \KD^{\leq n} (B^\bullet)| \simeq |\KD^{\leq n} (X^\bullet)|$. Hence, it suffices to show that $|\KD^{\leq n} (X^\bullet)| \simeq \KD^{\leq n} (A \times_C B)$.

	For this, consider the cosimplicial object $\tilde{X}^\bullet := \adic_{\leq n} (A^\bullet) \times_{\adic_{\leq n} (C^\bullet)} \adic_{\leq n} (B^\bullet)$ in $\Fil \CAlg_{R/ /k}^{\an, \surj, \leq n}$, so that $\aug( \tilde{X}^\bullet )\simeq X^\bullet$. By construction, $\cot_{\leq n} (A^\bullet)$, $\cot_{\leq n} (B^\bullet)$ and $\cot_{\leq n} (C^\bullet)$ are all split. Additionally, $\Gr \adic_{\leq n} (-) \simeq \LSym_k^{\Gr, \leq n}  (\cot_{\leq n} (-)(1))$, so 
	\begin{align*}
		\Gr (\tilde{X}^\bullet) \simeq \LSym_k^{\Gr, \leq n} (\cot_{\leq n} (A^\bullet)(1)) \times_{\LSym_k^{\Gr, \leq n} (\cot_{\leq n} (C^\bullet)(1))} \LSym_k^{\Gr, \leq n}  (\cot_{\leq n} (B^\bullet)(1))
	\end{align*}
	We can take fibre products to turn the splittings on $\cot_{\leq n} (A^\bullet)$, $\cot_{\leq n} (B^\bullet)$ and $\cot_{\leq n} (C^\bullet)$ into a splitting of $\Gr(\tilde{X}^\bullet)$. This implies that the natural map
	\begin{align*}
		\cot_{\leq n}(\Tot(\Gr(\tilde{X}^\bullet))) \xrightarrow{\simeq} \Tot(\cot_{\leq n} (\Gr(\tilde{X}^\bullet)))
	\end{align*}
	is an equivalence. Lifting to filtered objects, since everything is complete, implies that the natural map
	\begin{align*}
		\cot_{\leq n} (\Tot (\tilde{X}^\bullet)) \xrightarrow{\simeq} \Tot (\cot_{\leq n} (\tilde{X}^\bullet))
	\end{align*}
	is an equivalence. Applying the functor $\aug$, we get the result.
\end{proof}

\begin{construction}
	\label{construction-pushoutdiagram}
	Let $A \to C$ and $B \to C$ be maps in $\CAlg_{R/ /k}^{\an, \surj, \leq m}$. For $m \leq n$, the natural transformation $\KD^{\leq m} (-) \to \KD^{\leq n} (-)$ from \cref{rem-KDcolim} gives a commutative diagram
	\[
		\begin{tikzcd}
			\KD ^{\leq m} (A) \bigsqcup^{\weirdleq m} _{\KD^{\leq m} (C)} \KD^{\leq m} (B) \arrow{r} \arrow{d} & \KD^{\leq m} (A \times_C B) \arrow{d} \\
			\KD^{\leq n} (A) \bigsqcup ^{\weirdleq m} _{\KD^{\leq n} (C)} \KD^{\leq n} (B) \arrow{r} & \KD^{\leq n} (A \times_C B)
		\end{tikzcd}
	\]
	where the superscript $\weirdleq m$ in the coproduct symbol means that the colimit is taken in the category of $\Alg_{\Delta, k/R}^{\pi, \weirdleq m}$-algebras.

	The map $\KD^{\leq n} (A) \bigsqcup^{\weirdleq m}_{\KD^{\leq n} (C)} \KD^{\leq n} (B)\to \KD^{\leq n} (A \times_C B)$ is induced by the diagram
	\[
		\begin{tikzcd}
			\KD^{\leq n} (C) \arrow{d} \arrow{r} & \KD^{\leq n} (A) \arrow{d} \\
			\KD^{\leq n} (B) \arrow{r} & \KD^{\leq n} (A \times_C B)
		\end{tikzcd}
	\]
	in the category of $\Lie_{\Delta, k/R}^{\pi, \weirdleq m}$-algebras. Note that this diagram lifts to $\Lie_{\Delta, k/R}^{\pi, \weirdleq n}$-algebras as follows
	\[
		\begin{tikzcd}
			\KD^{\leq n} (C) \arrow{d} \arrow{r} & \KD^{\leq n} (A) \arrow{d} \arrow{ddr} & \\
			\KD^{\leq n} (B) \arrow{r} \arrow{drr} & \KD^{\leq n} (A) \bigsqcup_{\KD^{\leq n} (C)}^{\weirdleq n} \KD^{\leq n} (B)\arrow{dr} & \\
												& & \KD^{\leq n} (A \times_C B)
		\end{tikzcd}
	\]
	This implies that there is a factorisation
	\begin{align*}
		\KD^{\leq n} (A) \sqcup^{\weirdleq m}_{\KD^{\leq n} (C)} \KD^{\leq n} (B) \to \KD^{\leq n} (A) \sqcup^{\weirdleq n}_{\KD^{\leq n} (C)} \KD^{\leq n} (B) \to \KD^{\leq n} (A \times_C B)
	\end{align*}
	So we obtain the following variation on the first commutative diagram.
	\[
		\begin{tikzcd}
			\KD ^{\leq m} (A) \bigsqcup^{\weirdleq m} _{\KD^{\leq m} (C)} \KD^{\leq m} (B) \arrow{r} \arrow{d} & \KD^{\leq m} (A \times_C B) \arrow{d} \\
			\KD^{\leq n} (A) \bigsqcup ^{\weirdleq n} _{\KD^{\leq n} (C)} \KD^{\leq n} (B) \arrow{r} & \KD^{\leq n} (A \times_C B)
		\end{tikzcd}
	\]
	Notice that this is also valid for $n = \infty$.
\end{construction}

\begin{prop}
	\label{prop-pushoutascolimit}
	Let $A \to C$ and $B \to C$ be maps in $\CAlg_{R/ /k}^{\an, \surj, \leq m}$. Then, the natural map of chain complexes
	\begin{align*}
		\varinjlim_{n \geq m} \; \KD^{\leq n} (A) \sqcup_{\KD^{\leq n} (C)} ^{\weirdleq n} \KD^{\leq n} (B) \xrightarrow{\simeq} \KD (A) \sqcup _{\KD(C)} \KD(B)
	\end{align*}
	is an equivalence.
\end{prop}

\begin{proof}
	Remember that by the bar resolution we can write
	\begin{align*}
		\KD^{\leq n} (A) \simeq | (\Lie_{\Delta, k/R}^{\pi, \weirdleq n})^{\circ (\bullet + 1)} (\KD^{\leq n} (A))|
	\end{align*}
	and similar for $\KD (A)$. The map $\KD^{\leq n} (A) \to \KD(A)$ can be written simplicially. We have a map 
	\begin{align*}
		(\Lie_{\Delta, k/R}^{\pi, \weirdleq n})^{\circ (\bullet + 1)} (\KD^{\leq n} (A)) \to (\Lie_{\Delta, k/R}^\pi)^{\circ (\bullet + 1)} (\KD(A))
	\end{align*}
	given by the composition of the map $\KD^{\leq n} (A) \to \KD (A)$ and the map of monads $\Lie_{\Delta, k/R}^{\pi, \weirdleq n} \to \Lie_{\Delta, k/R}^{\pi}$. Similarly for $B$ and $C$. Taking the pushout of these simplicial diagrams, we see that
	\begin{align*}
		\KD^{\leq n} (A) \sqcup^{\weirdleq n} _{\KD^{\leq n} (C)} \KD^{\leq n} (B) \simeq \big\lvert \Lie_{\Delta, k/R}^{\pi, \weirdleq n} \left(  (\Lie_{\Delta, k/R}^{\pi, \weirdleq n})^{\circ (\bullet )} (\KD^{\leq n} (A)) \sqcup_{(\Lie_{\Delta, k/R}^{\pi, \weirdleq n})^{\circ (\bullet)} (\KD^{\leq n} (C))}^{\weirdleq 1} (\Lie_{\Delta, k/R}^{\pi, \weirdleq n})^{\circ (\bullet)} (\KD^{\leq n} (B)) \right) \big\rvert
	\end{align*}
	where the superscript $\weirdleq 1$ in the pushout means we are taking the pushout as chain complexes over $L_{k/R}^\vee [1]$.

	We have the same resolution for $\KD(A) \sqcup_{\KD (C)} \KD(B)$ by setting $n = \infty$. Once again, the map between these pushouts can be written simplicially. That is, we have a map
	\begin{align*}
		\Lie_{\Delta, k/R}^{\pi, \weirdleq n} \left(  (\Lie_{\Delta, k/R}^{\pi, \weirdleq n})^{\circ (\bullet )} (\KD^{\leq n} (A)) \sqcup_{(\Lie_{\Delta, k/R}^{\pi, \weirdleq n})^{\circ (\bullet)} (\KD^{\leq n} (C))}^{\weirdleq 1} (\Lie_{\Delta, k/R}^{\pi, \weirdleq n})^{\circ (\bullet)} (\KD^{\leq n} (B)) \right) \\
		\to  \Lie_{\Delta, k/R}^\pi \left(  (\Lie_{\Delta, k/R}^\pi)^{\circ (\bullet )} (\KD (A)) \sqcup_{(\Lie_{\Delta, k/R}^\pi)^{\circ (\bullet)} (\KD (C))}^{\weirdleq 1} (\Lie_{\Delta, k/R}^\pi)^{\circ (\bullet)} (\KD (B)) \right)
	\end{align*}
	which is induced by the maps $\KD^{\leq n} (A) \to \KD(A)$, $\KD^{\leq n} (B)\to \KD (B)$, $\KD^{\leq n} (C) \to \KD (C)$, and the map of monads $\Lie_{\Delta, k/R}^{\pi, \weirdleq n} \to \Lie_{\Delta, k/R}^{\pi}$.

	Since colimits commute with colimits, it suffices to check that the map at each simplicial level is an equivalence after taking filtered colimits. But this is true because the natural map $\varinjlim_n \KD^{\leq n} (A) \xrightarrow{\simeq} \KD(A)$ is an equivalence (\cref{rem-KDcolim}), similarly for $B$ and $C$, $\Lie_{\Delta, k/R}^\pi \simeq \varinjlim_n \Lie_{\Delta, k/R}^{\pi, \weirdleq n} $ (\cref{prop-Liecolim}), and all of these monads commute with filtered colimits.
\end{proof}

\begin{prop}
	\label{prop-KDpushout}
	Let $A \to C$ and $B \to C$ be maps in $\CAlg_{R/ /k}^{\an, \bNil}$ that are surjective on $\pi_0$. Then, the natural map
	\begin{align*}
		\KD(A) \sqcup_{\KD(C)} \KD(B) \xrightarrow{\simeq} \KD(A \times_C B)
	\end{align*}
	is an equivalence.
\end{prop}

\begin{proof}
	Notice that it is enough to check that this is an equivalence at the level of chain complexes.

	By \cref{prop-Nilimage}, we may choose $m$ large enough so that the maps $A \to C$ and $B \to C$ lift to $\CAlg_{R/ /k}^{\an, \surj, \leq m}$. By \cref{construction-pushoutdiagram}, for any $n \geq m$ we have a commutative diagram
	\[
		\begin{tikzcd}
			\KD ^{\leq n} (A) \sqcup^{\weirdleq n} _{\KD^{\leq n} (C)} \KD^{\leq n} (B) \arrow{r} \arrow{d} & \KD^{\leq n} (A \times_C B) \arrow{d} \\
			\KD(A) \sqcup _{\KD (C)} \KD (B) \arrow{r} & \KD (A \times_C B)
		\end{tikzcd}
	\]
	and these maps are compatible with each other as $n$ varies. By \cref{prop-pushoutascolimit}, the left vertical maps assemble into an equivalence 
	\begin{align*}
		\varinjlim_{n \geq m} \; \KD^{\leq n} (A) \sqcup_{\KD^{\leq n} (C)} ^{\weirdleq n} \KD^{\leq n} (B) \xrightarrow{\simeq} \KD (A) \sqcup _{\KD(C)} \KD(B)
	\end{align*}
	and by \cref{rem-KDcolim}, the right vertical maps assemble into an equivalence $\varinjlim_{n \geq m} \KD^{\leq n} (A \times_C B) \xrightarrow{\simeq} \KD(A \times_C B)$. Hence, the map $\KD(A) \sqcup _{\KD (C)} \KD (B) \to \KD (A \times_C B)$ can be written as a filtered colimit of the maps $\KD ^{\leq n} (A) \sqcup^{\weirdleq n} _{\KD^{\leq n} (C)} \KD^{\leq n} (B) \to  \KD^{\leq n} (A \times_C B)$. But each of these maps is an equivalence by \cref{prop-truncatedpushout}, so we are done.
\end{proof}

\begin{prop}
	\label{prop-KDfullyfaithful}
	The functor
	\begin{align*}
		\KD: \CAlg_{R/ /k}^{\an, \bNil} \to \left( \LieAlgd_{\Delta, k/R}^{\pi, \blacksquare} \right)  ^{\op}
	\end{align*}
	is fully faithful and for any $A \in \CAlg_{R/ /k}^{\an, \bNil}$, the unit map $A \to C^* ( \KD (A))$ is an equivalence.
\end{prop}

\begin{proof}
	It suffices to show that the unit map is an equivalence, since then the functor being fully faithful follows.

	We first prove that the functor is fully faithful on square-zero extensions. Let $V, W \in (\Mod_{k, \geq 0})_{L_{k/R}[-1]/ }$. By construction, $\KD(\sqz(V)) \simeq \Lie_{\Delta, k/R}^\pi (V^\vee)$ and similarly for $W$, where we write $\Lie_{\Delta, k/R}^\pi (V^\vee)$ for the free $\Lie_{\Delta, k/R}^\pi$-algebra on $V^\vee$. Then we have a natural map
	\begin{align*}
		\Map_{\CAlg_{R/ /k}^{\an, \surj}} (\sqz (V), \sqz(W))  \xrightarrow{\KD} &  \Map_{\LieAlgd_{\Delta, k/R}^{\pi, \blacksquare}} (\Lie_{\Delta, k/R}^\pi (W^\vee), \Lie_{\Delta, k/R}^\pi (V^\vee)) \\
		\xrightarrow{\forget} & \Map_{(\Solid_k)_{/L_{k/R}^\vee [1] }} (\cot^\vee (\sqz(W)), \cot^\vee( \sqz(V))) \\ 
		\xrightarrow{\text{unit}_{\Lie_{\Delta, k/R}^{\pi}}}  & \Map_{(\Solid_k)_{/L_{k/R}^\vee [1]}} (W^\vee, \cot^\vee (\sqz(V))) 
	\end{align*}
	The composition of the last two maps is an equivalence due to the free forgetful adjunction on partition Lie algebroids. By construction, the unit on $\Lie_{\Delta, k/R}^\pi$ coincides with the unit on the adjunction $\cot^\vee : \CAlg_{R/ /k}^{\an, \surj} \leftrightarrows (\Solid_{k, \leq 0})_{/L_{k/R}^\vee [1]}^{\op} : \sqz(-)^\vee$. Hence, the composition factors as
	\begin{align*}
		\Map_{\CAlg_{R/ /k}^{\an}} (\sqz(V), \sqz (W)) \xrightarrow{\cot^\vee} & \Map_{(\Solid_k)_{/L_{k/R}^\vee [1]}} (\cot^\vee(\sqz(W)), \cot^\vee(\sqz(V))) \\
		\xrightarrow{\text{unit}_{\Lie_{\Delta, k/R}^\pi}}&  \Map_{(\Solid_k)_{/L_{k/R}^\vee [1]}}(W^\vee, \cot^\vee(\sqz(V)))
	\end{align*}
	and it is an equivalence by the $\cot^\vee$ and $\sqz(-^\vee)$ adjunction. Hence, the map $\Map_{\CAlg_{R/ /k}^{\an}} (\sqz (V), \sqz (W)) \xrightarrow{\KD} \Map_{\LieAlgd_{\Delta, k/R}^{\pi, \blacksquare}} (\KD(\sqz(W)), \KD(\sqz(V)))$ is an equivalence.

	We will now show that if $V \in (\Mod_{k,\geq 0})_{L_{k/R}[-1]/}$ is eventually coconnective, the unit map $\sqz (V) \to C^* (\KD(\sqz (V)))$ is an equivalence. Let $\catC \subset \CAlg_{R/ /k}^{\an} $ be the full subcategory spanned by objects of the form $\sqz (V)$, where $V \in (\Mod_{k, \geq 0 })_{L_{k/R}[-1]/}$ is eventually coconnective, and let $\catD \subset \LieAlgd_{\Delta, k/R}^{\pi, \blacksquare}$ be the full subcategory spanned by objects of the form $\Lie_{\Delta, k/R}^{\pi} (V) $ where $V \in (\Mod_{k, \leq 0})_{L_{k/R}[-1] /}$ is eventually coconnective. We will show that the Koszul duality adjunction restricts to an adjunction between $\catC$ and $\catD^{\op}$. For this, we only need to prove that $C^* (\KD(\sqz(V))) \simeq \sqz (V)$. This is true by applying the Yoneda embedding to the equivalence
	\begin{align*}
		\Map_{\CAlg_{R/ /k}^{\an}} (-, C^* (\KD(\sqz(V)))) & \simeq \Map_{\LieAlgd_{\Delta, k/R}^{\pi, \blacksquare}} (\Lie_{\Delta, k/R}^\pi (V^\vee), \KD(-))  \\
																			& \simeq \Map_{(\Solid_k)_{/L_{k/R}^\vee [1]}} (V^\vee, \cot^\vee(-)) \\
																			& \simeq \Map_{\CAlg_{R/ /k}^{\an}} (-, \sqz(V))
	\end{align*}
	We just proved that the left adjoint on $\catC$ is fully faithful, and it is obviously essentially surjective, so the adjunction $\KD: \catC \leftrightarrows \catD^{\op}: C^*$ is an equivalence. Hence, for any $X \in \catD$, the unit $X \to \KD(C^* (X))$ is an equivalence.

	We will now prove the more general statement. Let $A \in \CAlg_{R/ /k}^{\an, \bNil}$. Then, $A$ can be obtained as a sequence of square-zero extensions, all of which are taken to pushouts by $\KD$ (\cref{prop-KDpushout}). The functor $C^*$ preserves pushouts by the adjunction. Hence, the functor $C^* \circ \KD$ preserves all of the pullbacks in the refined Postnikov tower. This reduces it to the case $A = \sqz (V)$ for $V \in (\Mod_{k, \geq 0})_{L_{k/R}[-1]/}$, which we have already proved.
\end{proof}

We will finish this section by showing that partition Lie algebras in $\Solid_{k, \leq 0}^{\Pro}$ are \emph{not} the profinite completion of nilpotent partition Lie algebras. To simplify this example, we will set $R = k$.

\begin{defi}
	Let $V \in \Vect_k^\omega$ and $n < 0$. The \emph{abelian partition Lie algebra on $V[n]$} is the partition Lie algebra $\ab (V[n]) := \KD (\LSym_k^*  (V^\vee[-n]))$.

	The category of \emph{nilpotent partition Lie algebras} is the smallest full subcategory of $\LieAlgd_{\Delta, k/ k}^{\pi, \blacksquare}$ that is closed under pullbacks and contains $\ab(V[n])$ for every $n \leq 0$ and $V \in \Vect_k^\omega$. We will write $\LieAlgd_{\Delta, k/k}^{\pi, \Nil}$ for this category.
\end{defi}

\begin{prop}
	Let $F$ be the functor
	\begin{align*}
		F: \Pro(\LieAlgd_{\Delta, k/ k}^{\pi, \Nil}) \to \LieAlgd_{\Delta, k/k}^{\pi, \blacksquare}
	\end{align*}
	given by the inverse-limit-preserving extension of the fully faithful embedding $\LieAlgd_{\Delta, k/k}^{\pi, \Nil} \subset \LieAlgd_{\Delta, k/k}^{\pi, \blacksquare}$. Then, $F$ is not fully faithful.
\end{prop}

\begin{proof}
	Let $V = \varprojlim V_i$ be a profinite vector space. Then, $F$ sends the formal inverse limit $\varprojlim \ab (V_i[n])$ to the partition Lie algebra $\ab (V[n]) \simeq \KD (\LSym^* (V^\vee [-n]))$. Then, we can compute the following mapping space by applying \cref{thm-Liecolim}.
	\begin{align*}
		\Map_{\LieAlgd_{\Delta, k/ k}^{\pi, \blacksquare}} (\ab (V[n]), \ab (k[0])) \simeq \varprojlim_m \Map _{\LieAlgd_{\Delta, k/ k}^{\pi, \weirdleq m, \blacksquare} } (\ab (V[n]), \ab (k[0]))
	\end{align*}
	There is an equivalence of $\Lie_{\Delta, k/ k}^{\pi, \weirdleq m}$-algebras $\ab (V[n]) \simeq \KD^{\leq m} (\LSym_k^{\leq m} (V^\vee [-n] ))$. By \cref{prop-cotcomonadic}, the functor $\KD^{\leq m} : \CAlg_{k/ /k}^{\an, \surj, \leq m} \to \LieAlgd_{\Delta, k/ k}^{\pi, \weirdleq m, \blacksquare}$ is fully faithful. Hence,
	\begin{align*}
		\Map_{\LieAlgd_{\Delta, k/ k}^{\pi}} (\ab (V[n]), \ab (k[0])) & \simeq \varprojlim_m \Map _{\CAlg_{k/ /k}^{\an, \surj, \leq m}} (\LSym_k^{\leq m} (k[0]), \LSym_k^{\leq m} (V^\vee [-n])) \\
																						 & \simeq \varprojlim_m \LSym_k^{\leq m} (V^\vee [-n])
	\end{align*}
	Whereas, since $\ab (k[0])$ is cocompact in $\Pro(\LieAlgd_{\Delta, k/ k}^{\pi, \Nil})$,
	\begin{align*}
		\Map_{\Pro(\LieAlgd_{\Delta, k/ k}^{\pi, \Nil})} (\ab (V[n]), \ab (k[0])) & \simeq \varinjlim_i \Map_{\LieAlgd_{\Delta, k/k}^{\pi, \Nil}} (\ab (V_i[n]), \ab (k[0]))  \\
																												  & \simeq \varinjlim _i \varprojlim_m \LSym^{\leq m} (V_i^\vee [-n])
	\end{align*}
	The natural map between these mapping spaces is not an isomorphism. Indeed, if $\{x_j\}_{j \in J}$ is a dual basis of $V$, the first mapping space consists of formal power series on these variables that are finite in each degree. On the other hand, the second mapping space consists of formal power series on these variables that only use finitely many variables.
\end{proof}

\begin{rem}
	However, an analogous equivalence does hold for dually truncated partition Lie algebras. This is essentially given by the equivalence $\LieAlgd_{\Delta, k/k}^{\pi, \weirdleq n, \blacksquare} (\Solid_{k, \leq 0}^{\Pro} ) \simeq (\CAlg_{k/ /k}^{\an, \surj, \leq n})^{\op}$. The category $\CAlg_{k/ /k}^{\an, \surj, \leq n}$ is the Ind-completion of finite colimits of finite free $\LSym_k^{\leq n}$-algebras. By applying Koszul duality, we can write $\LieAlgd_{\Delta, k/ k}^{\pi, \weirdleq n} (\Solid_{k, \leq 0}^{\Pro})$ as the Pro-completion of the category of finite pullbacks of finite abelian Lie algebras.
\end{rem}

\section{Nilpotent formal moduli problems}
\label{section-nilpotentfmp}

Fix a ring $R$ and a map $R \to k$ where $k$ is a field. We will now study formal deformations over nilpotent (but not necessarily Artinian) bases.

\begin{defi}
	An \emph{inaccessible nilpotent formal moduli problem} is a functor $F: \CAlg_{R/ /k}^{\an, \bNil} \to \Space$ such that
	\begin{enumerate}
		\item $F(k) \simeq *$
		\item If $A \to C$ and $B \to C$ are $\pi_0$-surjective maps in $\CAlg_{R/ /k}^{\an, \bNil}$, the diagram
			\[
				\begin{tikzcd}
					F(A \times_C B) \arrow{r} \arrow{d} & F(A) \arrow{d} \\
					F(B) \arrow{r} & F(C)
				\end{tikzcd}
			\]
		is a pullback square.
	\end{enumerate}
\end{defi}

\begin{construction}[Tangent complex]
	Let $F: \CAlg_{R/ /k}^{\an, \bNil} \to \Space$ be an inaccessible nilpotent formal moduli problem. Then, we have a functor $(\Mod_{k, \geq 0}^{< \infty})_{L_{k/R}[-1]} \to \Space$ given by
	\begin{align*}
		V \mapsto F(k \oplus V[n])
	\end{align*}
	This functor preserves $\pi_0$-surjective pullbacks by the assumption that $F$ is an inaccessible nilpotent formal moduli problem. We will call $T_F$ the \emph{tangent complex of $F$}.
\end{construction}

We will later need the following technical lemma.

\begin{lemma}
	\label{lemma-weirdKan}
	Let $\kappa$ be an uncountable cardinal, and let $\Mod_{k, \geq 0}^{< \infty} \subset \Mod_{k, \geq 0}$ be the full subcategory of eventually coconnective chain complexes. Then, there is an equivalence of $\infty$-categories 
	\begin{align*}
		\Fun _{\kappa, \pi_0} (\Mod_{k, \geq 0}^{< \infty}, \Space) \simeq \Solid_k^\kappa
	\end{align*}
	where $\Fun _{\kappa, \pi_0} (\Mod_{k, \geq 0}^{< \infty}, \Space) \subset \Fun (\Mod_{k, \geq 0}^{< \infty}, \Space)$ is the full subcategory of functors that commute with $\kappa$-filtered colimits and preserve $\pi_0$-surjective pullbacks and the terminal object.
\end{lemma}

\begin{proof}
	We will start by writing $\Mod_{k, \geq 0}^{< \infty}$ as an accessible category. Let $\Mod_k^b \subset \Mod_k$ be the full subcategory of bounded chain complexes, and let $\Mod_k^{b, \kappa} \subset \Mod_k^b$ be the full subcategory of bounded chain complexes whose homology has dimension $< \kappa$. We will show that
	\begin{align*}
		\Ind _\kappa (\Mod_k^{b, \kappa}) \simeq \Mod_k^b
	\end{align*}
	Firstly, $\Mod_k^b \subset \Mod_k$ is closed under $\kappa$-filtered colimits. Suppose $F: I \to \Mod_k^b$ is a $\kappa$-filtered diagram. Let $(x_n)_{n \in \ZZ}$ be a sequence with $x_n \in \pi_n (\colim_I F)$. Lift each $x_n$ to an element $y_n \in \pi_n (F(i_n))$ for some $i_n \in I$. Since $I$ is $\kappa$-filtered and $\kappa$ is uncountable, the diagram consisting of the $i_n$ admits a cocone $i$. So there is a factorisation $F(i_n) \to F(i) \to \colim_I F$. Since $F(i)$ is bounded, $x_n = 0$ for all but finitely many $n$. This shows $\colim_I F \in \Mod_k^b$.

	The objects in $\Mod_k^{b, \kappa}$ are clearly $\kappa$-compact. By (\cite{HTT}, Proposition 5.3.5.11), we just need to show that the objects in $\Mod_k^{b, \kappa}$ generate all of $\Mod_k^b$ under $\kappa$-filtered colimits. But this is clear since every chain complex is split.

	By a similar argument, $\Mod_{k, \geq 0}^{< \infty} \simeq \Ind_\kappa (\Mod_{k, \geq 0}^{< \infty, \kappa})$, where $\Mod_{k, \geq 0}^{< \infty, \kappa} \subset \Mod_{k, \geq 0}^{< \infty}$ is the full subcategory of chain complexes whose homology has dimension $< \kappa$.

	By (\cite{ultrasolid}, Proposition 2.25), for an uncountable strong limit cardinal $\kappa$, we can describe $\Solid_k^\kappa$ as the $\infty$-category of exact functors $(\Solid_{k, \geq 0}^{\perf, \kappa} )^{\op} \to \Space$, or equivalently, exact functors $\Mod_{k, \geq 0}^{< \infty, \kappa} \to \Space$, where $\Mod_{k, \geq 0}^{< \infty, \kappa} \subset \Mod_{k, \geq 0}^{< \infty}$.

	Restriction gives an equivalence
	\begin{align*}
		\Solid_k^\kappa \simeq \Fun_{\exact} (\Mod_k^{b, \kappa}, \Space) \xrightarrow{\simeq} \Fun _{\pi_0} (\Mod_{k, \geq 0}^{< \infty, \kappa}, \Space)
	\end{align*}
	where the subscript $\pi_0$ refers to those functors preserving $\pi_0$-surjective pullbacks and the terminal object. Its inverse is given by right Kan extension. 

	The result now follows, since any functor $F \in \Fun_{\mathcal{A}, \pi_0} (\Mod_{k, \geq 0}^{< \infty}, \Space)$ is left Kan extended from $\Mod_{k, \geq 0}^{< \infty, \kappa}$ for some cardinal $\kappa$, and the left Kan extension of any functor $\Mod_{k, \geq 0}^{< \infty, \kappa} \to \Space$ that preserves $\pi_0$-surjective pullbacks will also preserve $\pi_0$-surjective pullbacks..
\end{proof}

\begin{lemma}
	\label{lemma-cotcomplexasfunctor}
	Let $V \in \Mod_k$. Then, there is an equivalence of $\infty$-categories
	\begin{align*}
		\Fun_{\mathcal{A}, \pi_0} ((\Mod_{k, \geq 0}^{< \infty})_{V/}, \Space) \simeq (\Solid_k)_{/V^\vee}
	\end{align*}
\end{lemma}

\begin{proof}
	This immediately follows from \cref{lemma-weirdKan} and (\cite{foliations}, Proposition 3.15).
\end{proof}

\begin{rem}
	By \cref{lemma-cotcomplexasfunctor}, we can identify the category $(\Solid_k)_{/L_{k/R}^\vee[1]}$ with the category of functors $(\Mod_{k, \geq 0}^{< \infty}) _{L_{k/R}[-1]} \to \Space$ that are accessible and preserve $\pi_0$-surjective pullbacks; that is, they are left Kan extended from $(\Mod_{k, \geq 0}^{< \infty, \kappa})_{L_{k/R}[-1]/}$ for some uncountable cardinal $\kappa$, where $\Mod_{k, \geq 0}^{< \infty, \kappa}$ is the category of bounded connective chain complexes whose homology has dimension $< \kappa$.
\end{rem}

\begin{defi}
	A \emph{nilpotent formal moduli problem} is an inaccessible nilpotent formal moduli problem $F: \CAlg_{R/ /k}^{\an, \bNil} \to \Space$ such that $T_{k/F}[1] : (\Mod_{k, \geq 0}^{< \infty})_{/L_{k/R}^\vee [1]} \to \Space$ is an object in $(\Solid_k)_{/L_{k/R}^\vee [1]}$.

	We will write $\Moduli_{R/k}^{\blacksquare}$ for the $\infty$-category of nilpotent formal moduli problems.
\end{defi}

\begin{rem}
	It is not clear a priori that the category $\Moduli_{R/k}^{\blacksquare}$ is locally small, but this is the case by \cref{cor-modulicolim}.
\end{rem}

\begin{rem}
	By (\cite{DAGX}, Proposition 1.1.19) and \cref{prop-smallsurjections}, a functor $F: \CAlg_{R/ /k}^{\an, \bNil} \to \Space$ preserves pullbacks of $\pi_0$-surjective morphisms if and only if it preserves pullbacks of the form
	\[
		\begin{tikzcd}
			A \arrow{r} \arrow{d} & B \arrow{d} \\
			k \arrow{r} & k \oplus V[n]
		\end{tikzcd}
	\]
	for $n > 0$ and $V \in \Vect_k$.
\end{rem}

\begin{prop}
	\label{prop-accessiblefmp}
	Let $F: \CAlg_{R/ /k}^{\an, \bNil} \to \Space$ be a nilpotent formal moduli problem such that $T_{k/F}[1]$ is a $\lambda$-ultrasolid module for some uncountable cardinal $\lambda$. Then, $F$ commutes with all $\lambda$-filtered colimits that, when computed in $\CAlg_{R/ /k}^{\an}$, exist in $\CAlg_{R/ /k}^{\an, \bNil}$.
\end{prop}

\begin{proof}
	Let $J: I \to \CAlg_{R/ /k}^{\an, \bNil}$ be a functor, where $I$ is a $\lambda$-filtered category. Firstly, suppose that $J$ factors through the full subcategory $ \tau _{\leq 0}  \CAlg_{R/ /k}^{\an, \bNil}  \subset \CAlg_{R/ /k}^{\an, \bNil}$ of $0$-truncated objects.

	Given $A \in  \tau _{\leq 0} \CAlg_{R/ /k}^{\an, \bNil}$, write $A_{\leq n}$ for the object $A/\mathfrak{m}^{n+1} \in  \tau _{\leq 0} \CAlg_{k/ /k}^{\an, \bNil}$, where $\mathfrak{m}$ is the kernel ideal of $A$. Notice that, for all $A \in  \tau _{\leq 0} \CAlg_{R/ /k}^{\an, \bNil}$, $A_{\leq n} \simeq A $ for $n \gg 0$ and hence $F(A) \simeq \varprojlim F(A_{\leq n})$. This implies that $J \simeq \varprojlim J_{\leq n}$ where $J_{\leq n}$ is the functor $i \mapsto J(i)_{\leq n}$, and additionally $F \circ J \simeq \varprojlim F \circ J_{\leq n}$.

	Since $\aleph_0 < \lambda$ and $\lambda$-small limits commute with $\lambda$-filtered colimits, we may assume that $J = J_{\leq n}$ for some $n$. If $n = 0$, then $F(i) = k$ for all $i \in I$, so the statement is trivially true.

	We now proceed by induction. Suppose $n > 0$. By (\cite{HA}, Theorem 7.4.1.23) and (\cite{SAG}, Proposition 25.3.5.1), we have a pullback diagram of functors
	\[
		\begin{tikzcd}
			J_{\leq n} 	\arrow{r} \arrow{d} & J_{\leq n-1} \arrow{d}  \\
			k \arrow{r} & k \oplus J_n[1]
		\end{tikzcd}
	\]
	where $J_n = \fib (J_{\leq n} \to J_{\leq n-1})$, and this fibre is computed in the category of functors $I \to \Mod_R$. Notice that $J_n$ can be promoted to a functor $I \to \Mod_k$, and its image is actually in the subcategory $\Mod_k^\heartsuit \subset \Mod_k$. Since $F$ commutes with each of these pullbacks pointwise, we have a pullback diagram of the form
	\[
		\begin{tikzcd}
			F \circ J_{\leq n} 	\arrow{r} \arrow{d} & F \circ J_{\leq n-1} \arrow{d}  \\
			* \arrow{r} & F \circ (k \oplus J_n[1])
		\end{tikzcd}
	\]
	Hence, since filtered colimits commute with finite limits, 
	\begin{align*}
		F \circ \colim_I J_{\leq n} \simeq \fib (F \circ \colim_I J_{\leq n-1} \to F (\colim_I k \oplus J_n [1])) \\
	\end{align*}
	Both colimits in the fibre commute with $F$: the first one by inductive hypothesis, and the second one because $T_F$ is a $\lambda$-ultrasolid module. Finite limits commute with filtered colimits so $F \circ \colim_I J_{\leq n} \simeq \colim_I F \circ J_{\leq n}$ as required. 

	We will now proceed in a similar way by using refined Postnikov towers (\cref{construction-refinedPostnikov}) instead of the adic filtration. We have shown the result for when $J$ factors through $ \tau _{\leq 0} \CAlg_{k/ /k}^{\bNil}$. For general $J$, we have an isomorphism $J \simeq \varprojlim  \tau _{\leq n} \circ J$. Given that the Postnikov tower on any element of $\CAlg_{k/ /k}^{\bNil}$ is finite, $F \simeq \varprojlim F \circ \tau_{\leq n}$, and remember that Postnikov towers also commute with filtered colimits. We again apply that $\lambda$-small limits commute with $\lambda$-filtered colimits, so that we only need to show that $ \tau _{\leq n} J$ commutes with $\lambda$-filtered colimits.

	We proceed by induction on $n$. The case $n = 0$ has already been covered. Suppose $n > 0$. We now use the refined Postnikov towers to write $\tau_{\leq n} J \simeq \varprojlim_i \tau_{\leq n}^i J$, where $\tau_{\leq n}^0 \simeq \tau_{\leq n-1}$. For each element in $\CAlg_{R/ /k}^{\an, \bNil}$, this tower is finite, and the functors $\tau_{\leq n} ^i$ commute with filtered colimits, so we can repeat the previous reasoning and reduce it to showing that $\tau_{\leq n} ^i J$ commutes with $\lambda$-filtered colimits. The case $i = 0$ is handled by the inductive hypothesis on $n$. We have a pullback diagram of the form
	\[
		\begin{tikzcd}
			\tau_{\leq n}^{i+1} J \arrow{r} \arrow{d} & \tau_{\leq n}^i J \arrow{d} \\
			k \arrow{r} & k \oplus \mathfrak{m}^{i}\pi_n J/\mathfrak{m}^{i+1} [n+1]
		\end{tikzcd}
	\]
	$F$ preserves this pullback diagram pointwise. We may repeat the same argument as when we inducted on the adic filtration, by applying the inductive hypothesis and the assumption that $F$ is a nilpotent formal moduli problem, so we are done.
\end{proof}

\begin{defi}
	Let $\kappa$ be an uncountable strong limit cardinal. We will write $\CAlg_{R/ /k}^{\an, \bNil, \kappa} \subset \CAlg_{R/ /k}^{\an, \bNil}$ for the full subcategory spanned by those $A \in \CAlg_{R/ /k}^{\an, \bNil}$ such that $\mathfrak{m}^j \pi_i (A)/ \mathfrak{m}^{j+1}$ has dimension less than $\kappa$ as a $k$-vector space for all $i,j$, where $\mathfrak{m}$ is the kernel ideal of $\pi_0 (A)$.
\end{defi}

\begin{lemma}
	\label{lemma-IndbNil}
	If $\kappa$ is a cardinal with $\kappa > \aleph_0$, then $\CAlg_{R/ /k} ^{\an, \bNil} \subset \CAlg_{R/ /k}^{\an}$ is closed under $\kappa$-filtered colimits.

	Additionally, there is a cardinal $\kappa_0$ such that for all cardinals $\kappa >  \kappa_0$ there is an equivalence of $\infty$-categories
	\begin{align*}
		\Ind_\kappa ( \CAlg_{R/ /k}^{\an, \bNil, \kappa}) \simeq \CAlg_{R/ /k}^{\an, \bNil}
	\end{align*}
\end{lemma}

\begin{proof}
	For the first part, consider a diagram $F: I \to \CAlg_{R/ /k}^{\an, \bNil}$, where $I$ is $\kappa$-filtered. Then, $\colim_I F \to k$ is $\pi_0$-surjective. Suppose, for a contradiction, that $\colim_I F \notin \CAlg_{R/ /k}^{\an, \bNil}$. Then, it is either not nilpotent or not bounded. 

	Suppose that $\pi_n (\fib (\colim_I F \to k)) \neq 0$ for infinitely many $n$. For each $n \in \NN$, pick $x_n \in \pi_n (\fib(\colim_I F \to k))$ such that $x_n \neq 0$ for infinitely many $n$. For each $n$, pick $i_n \in I$ and an element $y_n \in \pi_n ( \fib (F(i_n) \to k))$ that lifts $x_n$. Since $I$ is $\kappa$-filtered, there exists $i \in I$ such that there is a map $i_n \to i$ for all $n$. Hence, each of the maps $\fib (F(i_n) \to k)$ factors through $\fib (F(i) \to k)$. By assumption, $F(i)  \in \CAlg_{R / /k}^{\an, \bNil}$, so $\fib (F(i) \to k)$ is eventually coconnective, which implies that $x_n = 0$ for $i \gg 0$, which is a contradiction.

	Alternatively, suppose that $\mathfrak{m} = \ker (\pi_0 (\colim_I F) \to k)$ is not nilpotent. This implies that for each $n$, there exist $x_{n,1}, \dots, x_{n,n} \in \mathfrak{m}$ such that $x_{n,1} \dots  x_{n,n} \neq 0$. For each $n,j$ pick $i_{n,j} \in I$ and a lift $y_{n,j} \in \ker (\pi_0 F(i_{n,j}) \to k)$ of $x_{n,j}$. We may now pick $i \in I$ such that each map $F(i_{n,j}) \to k$ factorises through $F(i)$. Since $F(i) \in \CAlg_{R/ /k}^{\an, \bNil}$, this implies that $x_{n,1}  \dots  x_{n,n} = 0$ for $n \gg 0$, a contradiction.

	Hence, $\colim_I F \in \CAlg_{R/ /k}^{\an, \bNil}$.

	We now prove the second part. Let $\kappa_0$ be such that $k \in \Mod_R $ is $\kappa_0$-compact as an $R$-module. Without loss of generality, $\kappa_0 > \aleph_0$. Consider the object $k \oplus V[n]$, where $V \in \Vect_k$ has dimension less than $\kappa$. This can be written, as an $R$-module, as a $\kappa$-small colimit of $\kappa$-compact objects, so it is $\kappa$-compact as well. Using refined Postnikov towers (\cref{construction-refinedPostnikov}), we can construct any object in $\CAlg_{R/ /k}^{\an, \bNil}$ as a finite number of pullbacks of $\kappa$-compact objects in $\Mod_R$. Since $\Mod_R$ is stable, we see that for any $A \in \CAlg_{R/ /k}^{\an, \bNil, \kappa}$, the underlying chain complex of $A$ is $\kappa$-compact in $\Mod_R$.

	We claim that if $A \in \CAlg_{R/ /k}^{\an}$ is $\kappa$-compact as a chain complex, then it is $\kappa$-compact in $\CAlg_{R/ /k}^{\an}$. It suffices to show that it is $\kappa$-compact as an object in $\CAlg_R^{\an}$. We can write $A$ as its Bar resolution in $\CAlg_R^{\an}$. Then,
	\begin{align*}
		\Map_{\CAlg_R^{\an}} (A, -) \simeq \Tot \Map _{\Mod_R} ((\LSym_R^* )^{\circ \bullet + 1} (A), -)
	\end{align*} 
	Since $\kappa > \aleph_0$, the $\LSym_R^*$ monad preserves $\kappa$-compact objects in $\Mod_R$. Totalisations commute with $\kappa$-filtered colimits, so we see that $A$ is $\kappa$-compact. Hence, all objects in $\CAlg_{R/ /k}^{\an, \bNil, \kappa}$ are $\kappa$-compact in $\CAlg_{R/ /k}^{\an}$.

	By (\cite{HTT}, Proposition 5.3.5.11), there is a fully faithful embedding $\Ind_{\kappa} (\CAlg_{R/ /k}^{\an, \bNil, \kappa}) \to \CAlg_{R/ /k}^{\an}$. Since $\CAlg_{R/ /k}^{\an, \bNil}$ is closed under $\kappa$-filtered colimits, we see that the essential image is contained in $ \CAlg_{R/ /k}^{\an, \bNil}$.

	By (\cite{HTT}, Proposition 5.3.5.11) again, it suffices to show that $\CAlg_{R/ /k}^{\an, \bNil, \kappa}$ generates all of $\CAlg_{R/ /k}^{\an, \bNil}$ under $\kappa$-filtered colimits. We will proceed by induction on refined Postnikov towers.

	Let $A \simeq \colim_I A_i$, where $I$ is $\kappa$-filtered and $A_i \in \CAlg_{R/ /k}^{\an, \bNil, \kappa}$ for all $i$, and let $V \in \Vect_k$. Write $V = \colim _J V_j$, where $J$ is $\kappa$-filtered and $V_j$ has dimension $< \kappa$. Suppose we have a map $A \to k \oplus V[n+1]$. It suffices to show that $A \times_{k \oplus V[n+1]} k$ can be written as a $\kappa$-filtered colimit of objects in $\CAlg_{R/ /k}^{\an, \bNil, \kappa}$.

	By the proof of (\cite{HTT}, Proposition 5.3.5.15), we can write this map as a colimit $\colim_L (A_l \to k \oplus V_l[n+1])$, where $L \subset I \times J$ is $\kappa$-filtered. Then, $A \times_{k \oplus V[n+1]} k \simeq \colim_L A_l \times_{k \oplus V_l [n+1]} k$.

	For all $l$, $A_l \in \CAlg_{R/ /k}^{\an, \bNil, \kappa}$ and $V_l$ has dimension $< \kappa$, so $A_l \times_{k \oplus V_l [n+1]} k \in \CAlg_{R/ /k}^{\an, \bNil, \kappa}$. Hence, we see that $A \times_{k \oplus V[n+1]} k$ is a $\kappa$-filtered colimit of objects in $\CAlg_{R/ /k}^{\an, \bNil, \kappa}$.
\end{proof}

\begin{example}
	If $\kappa = \aleph_0$, $\kappa$-filtered colimits are not well-behaved in $\CAlg_{R/ /k}^{\an, \bNil}$. For example, set $R = k$, and we have the colimit in $\CAlg^{\an}_{k/ /k}$
	\begin{align*}
		k[x^{1/2^\infty}]/x^2 \simeq \varinjlim_n k[x^{1/2^n}]/x^2
	\end{align*}
	where the map $k[x^{1/2^\infty}] \to k$ maps $x$ to zero. We claim that this colimit exists in $\CAlg_{k/ /k}^{\an, \bNil}$, and it is equal to $k$. 

	We will first prove that for $V \in \Mod_{k, \geq 0}$, $\Map_{\CAlg_{k/ /k}^{\an}} (k[x^{1/2^\infty}]/x^2, k \oplus V) \simeq *$, which is equivalent to $L_{k/k[x^{1/2^\infty}/x^2]} \simeq 0$.

	We can write $k[x^{1/2^n}]/x^2 \simeq k[y_n]/y_n^{2^{n+1}}$, where $y_n^{2^n} = x$. We have a derived pushout
	\[
		\begin{tikzcd}
			k[y_n ^{2^{n+1}}] \arrow{r} \arrow{d} & k[y_n] \arrow{d} \\
			k \arrow{r} & k[y_n]/y_n^{2^{n+1}}
		\end{tikzcd}
	\]
	The map $k[x^{1/2^n}]/x^2 \to k[x^{1/2^{n+1}}]/x^2$ is induced by the map $y_n \mapsto y_{n+1}^2$, Since the cotangent complex commutes with colimits,
	\begin{align*}
		L_{k/k[x^{1/2^\infty}]/x^2} \simeq \cofib(\colim L_{k/k[y_n^{2^{n+1}}]} \to \colim L_{k/k[y_n]})
	\end{align*}
	Remember that $L_{k/k[z]} \simeq k[1]$, and a squaring map between polynomial algebras in one generator induces a trivial map in $\cot$. Hence, both elements in the cofibre are zero, and $\cot(k[x^{1/2^\infty}]/x^2) \simeq 0$.

	Now suppose we have a cocone for the above colimit in $\CAlg_{k/ /k}^{\an, \bNil}$ with vertex $R$. Then, it factors through $k[x^{1/2^\infty}]/x^2$. Writing $R$ as a finite pullback of square-zero extensions, we see that $\Map_{\CAlg^{\an}_{k/ /k}} (k[x^{1/2^\infty}]/x^2, R) \simeq *$. Hence, the cocone factorises through $k$, which is initial and final in $\CAlg_{k/ /k}^{\an, \bNil}$, so the cone must be trivial. We conclude that the colimit of this diagram in $\CAlg_{k/ /k}^{\an, \bNil}$ exists and is equal to $k$.
\end{example}

\begin{defi}
	A \emph{$\kappa$-small nilpotent formal moduli problem} is a functor $F: \CAlg_{R/ /k}^{\an, \bNil, \kappa} \to \Space$ such that
	\begin{enumerate}
		\item $F(k) \simeq *$.
		\item If $A \to C$ and $B \to C$ are $\pi_0$-surjective maps in $\CAlg_{R/ /k}^{\an, \bNil, \kappa}$, the diagram
			\[
				\begin{tikzcd}
					F(A \times_C B) \arrow{r} \arrow{d} & F(A) \arrow{d} \\
					F(B) \arrow{r} & F(C)
				\end{tikzcd}
			\]
		is a pullback square.
	\end{enumerate}
	We will write $\Moduli_{R/k}^{\blacksquare, \kappa}$ for the $\infty$-category of $\kappa$-small ultrasolid formal moduli problems.
\end{defi}

\begin{construction}[$\kappa$-small tangent complex]
	Given $F \in \Moduli_{R/k}^{\blacksquare, \kappa}$, we will still write 
	\begin{align*}
		T_{k/F}[1] : (\Mod_{k, \geq 0}^{< \infty, \kappa})_{L_{k/R}[-1]/} \to \Space
	\end{align*}
	for the functor
	\begin{align*}
		(L_{k/R}[-1] \xrightarrow{\alpha} V) \mapsto F(k \oplus_\alpha V)
	\end{align*}
	By \cref{lemma-cotcomplexasfunctor}, we can identify $T_{k/F}[1]$ with an object in $(\Solid_k^\kappa)_{/L_{k/R}^\vee [1]}$.
\end{construction}

\begin{prop}
	\label{prop-smallfmp}
	Let $\kappa > \kappa_0$, where $\kappa_0$ is the cardinal in \cref{lemma-IndbNil}, be an infinite strong limit cardinal. Then, left Kan extension gives a fully faithful embedding
	\begin{align*}
		\Phi: \Moduli_{R/k}^{\blacksquare, \kappa} \to \Moduli_{R/k}^\blacksquare
	\end{align*}
	Additionally, $F \in \Moduli_{R/k}^\blacksquare$ is in the image of this embedding if and only if the tangent complex $T_{k/F}$ is $\kappa$-ultrasolid.
\end{prop}

\begin{rem}
	Using the above, from now on we will identify $\Moduli_{R/k}^{\blacksquare, \kappa}$ with the full subcategory of $\Moduli_{R/k}^\blacksquare$ as described above.
\end{rem}

\begin{proof}
	Left Kan extension gives a fully faithful embedding
	\begin{align*}
		\Moduli_{R/k}^{\blacksquare, \kappa} \to \Fun_\kappa (\CAlg_{R/ /k}^{\an, \bNil}, \Space)
	\end{align*}
	where $\Fun_\kappa (\CAlg_{R/ /k}^{\an, \bNil}, \Space) \subset \Fun (\CAlg_{R/ /k}^{\an, \bNil}, \Space)$ is the full subcategory of $\kappa$-filtered-colimit-preserving functors (\cref{lemma-IndbNil} and \cite{HTT}, Lemma 5.3.5.8). So we just need to check that if $F \in \Moduli_{R/k}^{\blacksquare, \kappa}$, its left Kan extension $\Phi (F) : \CAlg_{R/ /k}^{\an, \bNil} \to \Space$ is a nilpotent formal moduli problem. Given a pullback diagram of the form
	\[
		\begin{tikzcd}
			A	\arrow{r} \arrow{d} & B \arrow{d} \\
			k \arrow{r} & k \oplus V[n]
		\end{tikzcd}
	\]
	where $A \in \CAlg_{R/ /k}^{\an, \bNil}$ and $V \in \Vect_k$, we must show that it is preserved by $\Phi(F)$.

	By (\cite{HTT}, Proposition 5.3.5.15), the map $B \to k \oplus V[n]$ can be written as a $\kappa$-filtered colimit of maps of the form $B_i \to k \oplus V_i[n]$, where $B_i \in \CAlg_{R/ /k}^{\an, \bNil, \kappa}$ and $V_i \in \Vect_k$ has dimension $< \kappa$. Since $\Phi(F)$ commutes with $\kappa$-filtered colimits and finite limits commute with filtered colimits,
	\begin{align*}
		\Phi(F)(A) & \simeq \Phi(F)(\colim B_i \times _{k \oplus V_i[n]} k) \\
				  & \simeq \colim F(B_i \times_{k \oplus V_i[n]} k) \\
				  & \simeq \colim F(B_i) \times _{F(k \oplus V_i[n])} * \\
				  & \simeq \Phi(F)(B) \times _{\Phi(F)(k \oplus V[n])} *
	\end{align*}
	This shows $\Phi(F) \in \Moduli_{R/k}^\blacksquare$.

	We now show the last part. Suppose $F \in \Moduli_{R/k}^\blacksquare$ such that $T_{k/F}[1] $ is $\kappa$-ultrasolid. Then, by \cref{prop-accessiblefmp}, $F$ commutes with $\kappa$-filtered colimits, so $F$ is left Kan extended from $\CAlg_{R/ /k}^{\an, \bNil, \kappa}$ (\cite{HTT}, Lemma 5.3.5.8). 

	Conversely, if $F$ is left Kan extended from $\CAlg_{R/ /k}^{\an, \bNil, \kappa}$, it commutes with $\kappa$-filtered colimits. Then, $T_{k/F}[1] : (\Mod_{k, \geq 0}^{< \infty})_{L_{k/R}[-1] /} \to \Space$ commutes with $\kappa$-filtered colimits, so it is left Kan extended from $(\Mod_{k, \geq 0}^{< \infty, \kappa})_{L_{k/R}[-1]/}$. Thus, $T_{k/F}[1]$ is $\kappa$-ultrasolid (\cref{lemma-cotcomplexasfunctor}).
\end{proof}

\begin{cor}
	\label{cor-modulicolim}
	There is an equivalence of $\infty$-categories
	\begin{align*}
		\Moduli_{R/k}^{\blacksquare} \simeq \varinjlim _\kappa \Moduli_{R/k}^{\blacksquare, \kappa}
	\end{align*}
	where the colimit is taken over all uncountable strong limit cardinals $\kappa$.
\end{cor}

\begin{proof}
	This follows from \cref{prop-smallfmp}, since for any $F \in \Moduli_{R/k}^\blacksquare$, its tangent complex $T_F$ is $\kappa$-ultrasolid for some $\kappa$.
\end{proof}

\begin{prop}
	\label{prop-deftheory}
	Let $\sigma$ be an uncountable strong limit cardinal. The $\infty$-category $\CAlg_{R/ /k}^{\an, \surj}$, the functor $\KD_\sigma: \CAlg_{R/ /k}^{\an, \bNil} \to \left(  \LieAlgd_{\Delta, k/R}^{\pi, \blacksquare, \sigma} \right) ^{\op}$ and the set of spectrum objects $\left\{\{k \oplus V[n]\}_{n > 0} | V \in \Vect_k^\sigma\right\}$ are a deformation theory in the sense of (\cite{DAGX}, Definition 1.3.9).
\end{prop}

\begin{proof}
	Let $\catD \subset \LieAlgd_{\Delta, k/R}^{\pi, \blacksquare, \sigma}$ be the full subcategory spanned by the image of $\CAlg_{R/ /k}^{\an, \bNil, \sigma}$ under $\KD_\sigma$. We need to check conditions (D1)-(D4) of (\cite{DAGX}, Definition 1.3.1 and Definition 1.3.9).

	Firstly, the $\infty$-category $\LieAlgd_{\Delta, k/R}^{\pi, \blacksquare, \sigma}$ is presentable (\cite{monads}, Corollary 6.8). By \cref{construction-smallKD}, $\KD_\sigma$ has a right adjoint $C^*_\sigma : (\LieAlgd_{\Delta, k/R}^{\pi, \blacksquare, \sigma})^{\op} \to \CAlg^{\an, \surj}_{R/ /k}$. This verifies (D1) and (D2).

	By the proof of \cref{prop-KDfullyfaithful}, the Koszul duality adjunction restricts to an adjoint equivalence $\KD_\sigma : \CAlg_{R/ /k}^{\an, \bNil, \sigma} \leftrightarrows \catD^{\op} : C^*$. It follows that for any $\mathfrak{g} \in \catD$ the unit map $\mathfrak{g} \to \KD_\sigma(C_\sigma^* (\mathfrak{g}))$ is an equivalence. This verifies (D3).

	To each spectrum object $\{k \oplus V[n]\}_{n > 0}$ where $V \in \Vect_k^{\sigma}$, we have a functor $\LieAlgd_{\Delta, k/R}^{\pi, \blacksquare, \sigma} \to \Sp$ as follows: given $\mathfrak{g} \in \LieAlgd_{\Delta, k/R}^{\pi, \blacksquare, \sigma}$, the collection of spaces $\{\Map_{\LieAlgd_{\Delta, k/R}^{\pi, \blacksquare, \sigma}} (\KD (k \oplus V[n]), \mathfrak{g}) \}_{n > 0}$ forms a spectrum. Additionally, $\KD(k \oplus V[n]) \simeq \Lie_{\Delta, k/R}^\pi ( V \xrightarrow{0} L_{k/R}^\vee [1] )$, so we can identify this spectrum with $\Map^s _{\Solid_k^\sigma} (V^\vee,	\fib( \mathfrak{g} \to L_{k/R}^\vee [1] ))$, the mapping spectrum in the stable category $\Solid_k^\sigma$. If we consider $\fib (\mathfrak{g} \to L_{k/R}^\vee [1]) \in \Solid_k^\sigma$ as a functor $(\ProVect_{< \sigma} )^{\op} \to \Sp$, this mapping spectrum is just $\fib (\mathfrak{g} \to L_{k/R}^\vee [1]) (V^\vee)$.

	The monad $\Lie_{\Delta, k/R}^\pi$ commutes with sifted colimits, and colimits of ultrasolid modules are computed pointwise, so the functor $\mathfrak{g} \mapsto \fib (\mathfrak{g} \to L_{k/R}^\vee [1]) (V^\vee)$ commutes with sifted colimits. Additionally, the collection of these functors (where we take all $V \in \Vect_k^{\sigma}$) detects equivalences. This is because a map of partition Lie algebroids is an equivalence if and only if it is an equivalence of the underlying ultrasolid modules, and a map of $\sigma$-ultrasolid modules is an equivalence if and only if it is an equivalence at every section. This verifies (D4) and we are done.
\end{proof}

We can now conclude the following by combining \cref{prop-deftheory} and (\cite{DAGX}, Theorem 1.3.12).

\begin{cor}
	\label{cor-smallLP}
	Let $\sigma$ be an uncountable strong limit cardinal. Then, there is an equivalence of $\infty$-categories
	\begin{align*}
		\LieAlgd_{\Delta, k/R}^{\pi, \blacksquare, \sigma} \xrightarrow{\simeq} \Moduli_{R/k}^{\blacksquare, \sigma}
	\end{align*}
\end{cor}

We will now put together this result for all cardinals.

\begin{theorem}
	\label{theorem-nilpotentLP}
	The tangent complex $T_{k/-}: \Moduli_{R/k}^\blacksquare \to \Solid_k$ lifts to an equivalence of $\infty$-categories 
	\begin{align*}
		\Moduli_{R/k}^{\blacksquare} \simeq \LieAlgd_{\Delta, k/R}^{\pi, \blacksquare}
	\end{align*}
\end{theorem}

\begin{proof}
	We just need to check that the equivalence $\LieAlgd_{\Delta, k/R}^{\pi, \blacksquare, \sigma} \simeq \Moduli_{R/k}^{\blacksquare, \sigma}$ in \cref{cor-smallLP} is compatible with the filtered colimits $\LieAlgd_{\Delta, k/R}^{\pi, \blacksquare} \simeq \varinjlim_\sigma \LieAlgd_{\Delta, k/R}^{\pi, \blacksquare, \sigma}$ and $\Moduli_{R/k}^{\blacksquare} \simeq \varinjlim_\sigma \Moduli_{R/k}^{\blacksquare, \sigma}$ (\cref{cor-modulicolim}).

	For an uncountable strong limit cardinal $\sigma$, by the proof of (\cite{DAGX}, Theorem 1.3.12), the equivalence of $\infty$-categories
	\begin{align*}
		\LieAlgd_{\Delta, k/R}^{\pi, \blacksquare, \sigma} \xrightarrow{\Psi_\sigma} \Moduli_{R/k}^{\blacksquare, \sigma}
	\end{align*}
	is given by $\Psi_\sigma(\mathfrak{g}) \simeq \Map_{\LieAlgd_{\Delta, k/R}^{\pi, \blacksquare}} (\KD_\sigma(-), \mathfrak{g}) \simeq \Map_{\CAlg_{R/ /k}^{\an}} (C^*_\sigma(\mathfrak{g}), -) \in \Fun ( \CAlg_{R/ /k}^{\an, \bNil, \sigma}, \Space)$. For $\sigma < \sigma'$, we get a square
	\[
		\begin{tikzcd}
			\LieAlgd_{\Delta, k/R}^{\pi, \blacksquare, \sigma} \arrow{d} \arrow{r}{\Psi_\sigma} & \Moduli_{R/k}^{\blacksquare, \sigma} \arrow{d}{\Phi_\sigma^{\sigma'}} \\
			\LieAlgd_{\Delta, k/R}^{\pi, \blacksquare, \sigma'} \arrow{r}{\Psi_{\sigma'}} & \Moduli_{R/k}^{\blacksquare, \sigma'}
		\end{tikzcd}
	\]
	where the left vertical arrow is given by the fully faithful inclusion $\Solid_k^\sigma \subset \Solid_k^{\sigma'}$, and the functor $\Phi_\sigma^{\sigma'}$ is left Kan extension (\cref{prop-smallfmp}) and it is fully faithful. Thus, it suffices to show that this diagram is commutative.

	Given $F \in \Moduli_{R/k}^{\blacksquare, \sigma}$, $\Phi_\sigma^{\sigma'} (F)$ is the left Kan extension of $F$.	Let $\mathfrak{g} \in \LieAlgd_{\Delta, k/R}^{\pi, \blacksquare, \sigma}$. Then, $\KD_\sigma | _{\CAlg_{R/ /k}^{\an, \bNil, \sigma}} \simeq \KD_{\sigma'}|_{\CAlg_{R/ /k}^{\an, \bNil, \sigma}}$ (\cref{lemma-smallKD}), so 
	\begin{align*}
		\Phi_\sigma^{\sigma'} (\Psi_\sigma (\mathfrak{g})) |_{\CAlg_{R/ /k}^{\an, \bNil, \sigma}} & \simeq \Map_{\LieAlgd_{\Delta, k/R}^{\pi, \blacksquare}} (\KD_\sigma (-), \mathfrak{g}) \\
		&  \simeq \Map_{\LieAlgd_{\Delta, k/R}^{\pi, \blacksquare}} (\KD_{\sigma'} (-), \mathfrak{g})  \\
		& \simeq \Psi_{\sigma'} (\mathfrak{g})|_{\CAlg_{R/ /k}^{\an, \bNil, \sigma}}
	\end{align*}
	Thus, it suffices to show that $\Psi_{\sigma'} (\mathfrak{g})$ is left Kan extended from $ \CAlg_{R/ /k}^{\an, \bNil, \sigma}$. By \cref{prop-accessiblefmp}, it suffices to check that the tangent complex of $\Psi_{\sigma'} (\mathfrak{g})$ is $\sigma$-ultrasolid. By construction, the tangent complex of $\Psi_{\sigma'}(\mathfrak{g})$ is the underlying chain complex of $\mathfrak{g}$, which is $\sigma$-ultrasolid, so the result follows.
\end{proof}

\begin{rem}
	We can easily check that, given a formal moduli problem $F$, with associated partition Lie algebroid $\mathfrak{g}$, the underlying chain complex of $\mathfrak{g}$ is precisely $T_{k/F}[1]$, as stated in the theorem. Indeed, we can write $F = \Psi (\mathfrak{g})$, and then for $V \in (\Mod_{k, \geq 0}^{< \infty})_{L_{k/R}[-1]/}$,
	\begin{align*}
		\Psi(\mathfrak{g}) (\sqz (V)) \simeq \Map_{\LieAlgd_{\Delta, k/R}^{\pi, \blacksquare}} (\KD (\sqz (V)), \mathfrak{g}) \simeq \Map_{(\Solid_k)_{/L_{k/R}^\vee [1]}} (V^\vee, \mathfrak{g})
	\end{align*}
	where we used that $\KD(\sqz (V))$ is the free partition Lie algebroid on $V^\vee$.
\end{rem}

\newpage
\bibliographystyle{amsalpha}
\bibliography{references}

\providecommand{\bysame}{\leavevmode\hbox to3em{\hrulefill}\thinspace}
\providecommand{\MR}{\relax\ifhmode\unskip\space\fi MR }
% \MRhref is called by the amsart/book/proc definition of \MR.
\providecommand{\MRhref}[2]{%
  \href{http://www.ams.org/mathscinet-getitem?mr=#1}{#2}
}
\providecommand{\href}[2]{#2}
\begin{thebibliography}{BGMN22}

\bibitem[AR01]{sifted}
J.~Adámek and J.~Rosický, \emph{On {S}ifted {C}olimits {A}nd {G}eneralized
  {V}arieties}, Theory and Applications of Categories \textbf{8} (2001), no.~3.

\bibitem[Art74]{Artin}
Michael Artin, \emph{Versal deformations and algebraic stacks}, Inventiones
  mathematicae \textbf{27} (1974), 165--189.

\bibitem[Bar14]{barwick_spectra}
C.~Barwick, \emph{Spectral mackey functors and equivariant algebraic k-theory
  (i)}, 2014.

\bibitem[BCN24]{explicit}
Lukas Brantner, Ricardo Campos, and Joost Nuiten, \emph{{PD} {O}perads and
  {E}xplicit {P}artition {L}ie {A}lgebras}, 2024.

\bibitem[BGMN22]{polynomialfunctors}
Clark Barwick, Saul Glasman, Akhil Mathew, and Thomas Nikolaus, \emph{K-theory
  and polynomial functors}, 2022.

\bibitem[BK25]{condensed_localisation}
Christopher Brav and Grigorii Konovalov, \emph{Beilinson-parshin adeles via
  solid algebraic geometry}, 2025.

\bibitem[BM23]{partition}
Lukas Brantner and Akhil Mathew, \emph{Deformation {T}heory and {P}artition
  {L}ie {A}lgebras}, 2023.

\bibitem[BMN25]{foliations}
Lukas Brantner, Kirill Magidson, and Joost Nuiten, \emph{Formal integration of
  derived foliations}, 2025.

\bibitem[BW23]{BrantnerWaldron}
Lukas Brantner and Joe Waldron, \emph{{Purely {I}nseparable {G}alois theory
  {I}: {T}he {F}undamental {T}heorem}}, 2023.

\bibitem[Chi05]{Ching}
Michael Ching, \emph{Bar constructions for topological operads and the
  {G}oodwillie derivatives of the identity}, Geometry \& Topology \textbf{9}
  (2005), no.~2, 833–934.

\bibitem[Del86]{Deligne}
Pierre Deligne, \emph{Letter to {J}. {M}illson and {W}. {G}oldman}, 1986.

\bibitem[Dri14]{Drinfeld}
Vladimir Drinfeld, \emph{A letter from {Kharkov} to {Moscow}}, EMS Surveys in
  Mathematical Sciences \textbf{1} (2014), no.~2, 241--248, Originally written
  in 1988 as a letter to V. Schechtman.

\bibitem[GGN15]{Gepner_2015}
David Gepner, Moritz Groth, and Thomas Nikolaus, \emph{Universality of
  multiplicative infinite loop space machines}, Algebraic \&; Geometric
  Topology \textbf{15} (2015), no.~6, 3107–3153.

\bibitem[GM88]{Goldman}
William~M. Goldman and John~J. Millson, \emph{The deformation theory of
  representations of fundamental groups of compact {K}ähler manifolds},
  Publications Mathématiques de l'IHÉS \textbf{67} (1988), 43--96 (eng).

\bibitem[GM24]{Guillou}
Bertrand~J Guillou and J~Peter May, \emph{Models of g–spectra as presheaves
  of spectra}, Algebraic \& Geometric Topology \textbf{24} (2024), no.~3,
  1225–1275.

\bibitem[HM21]{monads}
Simon Henry and Nicholas~J. Meadows, \emph{Higher {T}heories and {M}onads},
  2021.

\bibitem[hs]{goodtheorysolid}
Peter~Scholze (https://mathoverflow.net/users/6074/peter scholze), \emph{Is
  there a good theory of solid vector spaces?}, MathOverflow,
  URL:https://mathoverflow.net/q/435760 (version: 2022-12-03).

\bibitem[Kat81]{KatzSerreTate}
Nicholas~M. Katz, \emph{Serre--{T}ate local moduli}, Surfaces Algébriques,
  Lecture Notes in Mathematics, vol. 868, Springer, Berlin, Heidelberg, 1981,
  pp.~138--202.

\bibitem[Kon03]{Kontsevich}
Maxim Kontsevich, \emph{{Deformation Quantization of Poisson Manifolds}},
  Letters in Mathematical Physics \textbf{66} (2003), no.~3, 157--216.

\bibitem[KS58]{Kodaira-Spencer}
K.~Kodaira and D.~C. Spencer, \emph{On {D}eformations of {C}omplex {A}nalytic
  {S}tructures, {I}}, Annals of Mathematics \textbf{67} (1958), no.~2,
  328--401.

\bibitem[KS06]{categories}
Masaki Kashiwara and Pierre Schapira, \emph{Categories and {S}heaves},
  Grundlehren der mathematischen Wissenschaften \textbf{332} (2006).

\bibitem[DAG X]{DAGX}
Jacob Lurie, \emph{Derived {A}lgebraic {G}eometry {X}: {F}ormal {M}oduli
  {P}roblems}, Preprint from the author's website.

\bibitem[HTT]{HTT}
\bysame, \emph{Higher {T}opos {T}heory}, Princeton University Press, 2009.

\bibitem[HA]{HA}
\bysame, \emph{Higher {A}lgebra}, Preprint from the author's website, 2017.

\bibitem[SAG]{SAG}
\bysame, \emph{Spectral {A}lgebraic {G}eometry}, Preprint from the author's
  website, 2018.

\bibitem[MA24]{ultrasolid}
Sofía Marlasca~Aparicio, \emph{Ultrasolid {H}omotopical {A}lgebra}, 2024.

\bibitem[Man04]{Manetti}
Marco Manetti, \emph{{Lectures on Deformations of Complex Manifolds}},
  Rendiconti di Matematica e delle sue Applicazioni \textbf{24} (2004), no.~1,
  1--183.

\bibitem[Mao24]{mao}
Zhouhang Mao, \emph{Revisiting derived crystalline cohomology}, 2024.

\bibitem[Pri19]{Pridham}
J.~P. Pridham, \emph{Unifying derived deformation theories}, 2019.

\bibitem[Ram02]{atomised}
Daniel~A. Ramras, \emph{Connectivity of the coset poset and the subgroup poset
  of a group}, 2002.

\bibitem[Sch68]{Schlessinger}
Michael Schlessinger, \emph{Functors of {A}rtin rings}, Transactions of the
  American Mathematical Society \textbf{130} (1968), no.~2, 208--222.

\bibitem[Sch19a]{analytic}
Peter Scholze, \emph{Lectures on {A}nalytic {G}eometry},
  \url{https://www.math.uni-bonn.de/people/scholze/Analytic.pdf}, October 2019.

\bibitem[Sch19b]{condensed}
\bysame, \emph{Lectures on {C}ondensed {M}athematics},
  \url{https://www.math.uni-bonn.de/people/scholze/Condensed.pdf}, May 2019.

\bibitem[Yek10]{completions}
Amnon Yekutieli, \emph{{On {F}latness and {C}ompletion for {I}nfinitely
  {G}enerated {M}odules over {N}oetherian {R}ings}}, 2010.

\end{thebibliography}

\end{document}